\documentclass[english]{article}
\usepackage[T1]{fontenc}
\usepackage{amssymb}

\usepackage[utf8]{inputenc}
\usepackage{babel}
\usepackage{float}
\usepackage{algpseudocode}
\usepackage[ruled,linesnumbered]{algorithm2e}

\usepackage{amsmath}
\usepackage{amsthm}
\usepackage{amssymb}

\usepackage{graphicx}
\usepackage[numbers]{natbib}
\usepackage[unicode=true]
 {hyperref}

\makeatletter
\theoremstyle{plain}
\newtheorem{thm}{\protect\theoremname}
\newenvironment{lyxlist}[1]
	{\begin{list}{}
		{\settowidth{\labelwidth}{#1}
		 \setlength{\leftmargin}{\labelwidth}
		 \addtolength{\leftmargin}{\labelsep}
		 }}
	{\end{list}}
\theoremstyle{remark}
\newtheorem{rem}{\protect\remarkname}
\theoremstyle{plain}
\newtheorem{prop}{\protect\propositionname}

\usepackage{babel}

\providecommand{\theoremname}{Theorem}

\makeatother

\providecommand{\propositionname}{Proposition}
\providecommand{\remarkname}{Remark}
\providecommand{\theoremname}{Theorem}

\begin{document}
\title{Data-Driven Linearization of Dynamical Systems }
\author{George Haller\thanks{Corresponding author. Email: georgehaller@ethz.ch}
and Bálint Kaszás\\
 Institute for Mechanical Systems\\
 ETH Zürich\\
 Leonhardstrasse 21, 8092 Zürich, Switzerland\\
 }

\maketitle
Dynamic Mode Decomposition (DMD) and its variants, such as extended
DMD (EDMD), are broadly used to fit simple linear models to dynamical
systems known from observable data. As DMD methods work well in several
situations but perform poorly in others, a clarification of the assumptions
under which DMD is applicable is desirable. Upon closer inspection,
existing interpretations of DMD methods based on the Koopman operator
are not quite satisfactory: they justify DMD under assumptions that
hold only with probability zero for generic observables. Here, we
give a justification for DMD as a local, leading-order reduced model
for the dominant system dynamics under conditions that hold with probability
one for generic observables and non-degenerate observational data.
We achieve this for autonomous and for periodically forced systems
of finite or infinite dimensions by constructing linearizing transformations
for their dominant dynamics within attracting slow spectral submanifolds
(SSMs). Our arguments also lead to a new algorithm, data-driven linearization
(DDL), which is a higher-order, systematic linearization of the observable
dynamics within slow SSMs. We show by examples how DDL outperforms
DMD and EDMD on numerical and experimental data.

\section{Introduction\label{sec:DMD and EDMD}}

In recent years, there has been an overwhelming interest in devising
linear models for dynamical systems from experimental or numerical
data (see the recent review by \citet{schmid22}). This trend was
largely started by the Dynamic Mode Decomposition (DMD), put forward
in seminal work by \citet{schmid10}. The original exposition of the
method has been streamlined by various authors, most notably by \citet{rowley09}
and \citet{kutz16}.

To describe DMD, we consider an autonomous dynamical system 
\begin{equation}
\dot{\mathbf{x}}=\mathbf{f}(\mathbf{x}),\qquad\mathbf{x}\in\mathcal{\mathbb{R}}^{n},\qquad\mathbf{f}\in C^{1}\left(\mathcal{\mathbb{R}}^{n}\right),\label{eq:nonlinear system0}
\end{equation}
for some $n\in\mathbb{N}^{+}$. Trajectories $\left\{ \mathbf{x}(t;\mathbf{x}_{0})\right\} _{t\in\mathbb{R}}$
of this system evolve from initial conditions $\mathbf{x}_{0}$. The
flow map $\mathbf{F}^{t}\colon\mathcal{\mathbb{R}}^{n}\to\mathcal{\mathbb{R}}^{n}$
is defined as the mapping taking the initial trajectory positions
at time $t_{0}=0$ to current ones at time $t$, i.e., 
\begin{equation}
\mathbf{F}^{t}(\mathbf{x}_{0})=\mathbf{x}(t;\mathbf{x}_{0}).\label{eq:defintion of flow map}
\end{equation}

As observations of the full state space variable $\mathbf{x}$ of
system (\ref{eq:nonlinear system0}) are often not available, one
may try to explore the dynamical system (\ref{eq:nonlinear system0})
by observing $d$ smooth scalar functions $\phi_{1}(x),\ldots,\text{\ensuremath{\phi}}_{d}(x)$
along trajectories of the system. We order these scalar observables
into the observable vector 
\begin{equation}
\boldsymbol{\phi}(x)=\left(\begin{array}{c}
\phi_{1}(\mathbf{x})\\
\vdots\\
\phi_{d}(\mathbf{x})
\end{array}\right)\in C^{1}\left(\mathcal{\mathbb{R}}^{n}\right)\label{eq:definition of observables}
\end{equation}
The basic idea of DMD is to approximate the observed evolution of
$\boldsymbol{\phi}\left(\mathbf{F}^{t}(\mathbf{x}_{0})\right)$ of
the dynamical system with the closest fitting autonomous linear dynamical
system
\begin{equation}
\dot{\boldsymbol{\phi}}=\mathbf{L}\boldsymbol{\phi},\qquad\mathbf{L}\in\mathbb{R}^{d\times d},\label{eq:linear observable dynamics}
\end{equation}
based on available trajectory observations.

This is a challenging objective for multiple reasons. First, the original
dynamical system (\ref{eq:nonlinear system0}) is generally nonlinear
whose dynamics cannot be well approximated by a single linear system
on a sizable open domain. For instance, one may have several isolated,
coexisting attracting or repelling stationary states (such as periodic
orbits, limit cycles or quasiperiodic tori), which linear systems
cannot have. Second, it is unclear why the dynamics of $d$ observables
should be governed by a self-contained autonomous dynamical system
induced by the original system (\ref{eq:nonlinear system0}), whose
dimension is $n$. Third, the result of fitting system (\ref{eq:linear observable dynamics})
to observable data will clearly depend on the initial conditions used,
the number and the functional form of the observables chosen, as well
as on the objective function used in minimizing the fit.

Despite these challenges, we may proceed to find an appropriately
defined closest linear system (\ref{eq:linear observable dynamics})
based on available observable data. We assume that for some fixed
time step $\Delta t$, discrete observations of $m$ initial conditions,
$\mathbf{x}^{1}(t_{0}),\ldots,\mathbf{x}^{m}(t_{0})$, and their images
$\mathbf{F}^{\Delta t}(\mathbf{x}^{1}(t_{0})),\ldots,\mathbf{F}^{\Delta t}\left(\mathbf{x}^{m}(t_{0})\right)$,
under the sampled flow map $\mathbf{F}^{\Delta t}$ are available
in the data matrices 

\begin{equation}
\boldsymbol{\Phi}=\left[\boldsymbol{\phi}\left(\mathbf{x}^{1}(t_{0})\right),\ldots,\boldsymbol{\phi}\left(\mathbf{x}^{m}(t_{0})\right)\right],\qquad\hat{\boldsymbol{\Phi}}=\left[\boldsymbol{\phi}\left(\mathbf{F}^{\Delta t}(\mathbf{x}^{1}(t_{0}))\right),\ldots,\boldsymbol{\phi}\left(\mathbf{F}^{\Delta t}\left(\mathbf{x}^{m}(t_{0})\right)\right)\right],\label{eq:DMD data matrix definitions}
\end{equation}
respectively. We seek the best fitting linear system of the form (\ref{eq:linear observable dynamics})
for which
\begin{equation}
\hat{\boldsymbol{\Phi}}\approx\boldsymbol{\mathcal{D}}\boldsymbol{\Phi},\quad\boldsymbol{\mathcal{D}}=e^{\mathbf{L}\Delta t},\label{eq:approximation by DMD}
\end{equation}
holds. The eigenvalues of such a $\boldsymbol{\mathcal{D}}$ are usually
called DMD eigenvalues, and their corresponding eigenvectors are called
the DMD modes. 

Various norms can be chosen with respect to which the difference of
$\hat{\boldsymbol{\Phi}}$ and $\mathcal{\boldsymbol{\mathcal{D}}}\boldsymbol{\Phi}$
is to be minimized. The most straightforward choice is the Euclidean
matrix norm $\left|\,\cdot\,\right|$, which leads to the minimization
principle 
\begin{equation}
\boldsymbol{\mathcal{D}}^{*}=\underset{\boldsymbol{\mathcal{D}}\in\mathbb{R}^{d\times d}}{\mathrm{argmin}\left|\hat{\boldsymbol{\Phi}}-\boldsymbol{\mathcal{D}}\boldsymbol{\Phi}\right|^{2}}.\label{eq:general DMD problem}
\end{equation}
 An explicit solution to this problem is given by
\begin{equation}
\boldsymbol{\mathcal{D}}=\left(\hat{\boldsymbol{\Phi}}\boldsymbol{\Phi}^{\mathrm{T}}\right)\left(\hat{\boldsymbol{\Phi}}\boldsymbol{\Phi}^{\mathrm{T}}\right)^{\dagger},\label{eq:general solution of DMD problem}
\end{equation}
with the dagger referring to the pseudo-inverse of a matrix (see,
e.g., \citet{kutz16} for details). We note that the original formulation
of \citet{schmid10} is for discrete dynamical processes and assumes
observations of a single trajectory (see also \citet{rowley09}).

Among several later variants of DMD surveyed by \citet{schmid22},
the most broadly used one is the Extended Dynamic Mode Decomposition
(EDMD) of \citet{williams15}. This procedure seeks the best-fitting
linear dynamics for an a priori unknown set of functions $\mathbf{K}(\boldsymbol{\phi}(\mathbf{x}))$
of $\boldsymbol{\phi}(\mathbf{x})$, rather than for $\boldsymbol{\phi}(\mathbf{x})$
itself. In practice, one often chooses $\mathbf{K}$ as an $N(d,k)$-dimensional
vector of $d$-variate scalar monomials of order $k$ or less, where
$N(d,k)=\left(\begin{array}{c}
d+k\\
k
\end{array}\right)$ is the total number of all such monomials. The underlying assumption
of EDMD is that a self-contained linear dynamical system of the form
\begin{equation}
\frac{d}{dt}\mathbf{K}\left(\boldsymbol{\phi}(\mathbf{F}^{t}(\mathbf{x}_{0}))\right)=\mathbf{L}\mathbf{K}\left(\boldsymbol{\phi}(\mathbf{F}^{t}(\mathbf{x}_{0}))\right)\label{eq:EDMD in principle}
\end{equation}
 can be obtained on the feature space $\mathbb{R}^{N(d,k)}$ by optimally
selecting $\mathbf{L}\in\mathbb{R}^{N(d,k)\times N(d,k)}$. For physical
systems, the $N(d,k)$-dimensional ODE in eq. (\ref{eq:EDMD in principle})
defined on the feature space $\mathbb{R}^{N(d,k)}$ can be substantially
higher-dimensional than the $d$-dimensional ODE (\ref{eq:linear observable dynamics}).
In fact, $N(d,k)$ may be substantially higher than the dimension
$n$ of the phase space $\mathbb{R}^{n}$ of the original nonlinear
system (\ref{eq:nonlinear system0}). 

Once the function library used in EDMD is fixed, one again seeks to
choose $\mathbf{L}$ so that
\[
\mathbf{K}(\hat{\boldsymbol{\Phi}})\approx\boldsymbol{\mathcal{D}}\mathbf{K}(\boldsymbol{\Phi}^{\mathrm{}}),\quad\boldsymbol{\mathcal{D}}=e^{\mathbf{L}\Delta t}.
\]
This again leads to a linear optimization problem that can be solved
using linear algebra tools. For higher-dimensional systems, a kernel-based
version of EDMD was developed by \citet{williams15b}. This method
computes inner products necessary for EDMD implicitly, without requiring
an explicit representation of (polynomial) basis functions in the
space of observables. As a result, kernel-based EDMD operates at computational
costs comparable to those of the original DMD. 

\section{Prior justifications for DMD methods}

Available justifications for DMD (see \citet{rowley09}) and EDMD
(see \citet{williams15}) are based on the Koopman operator, whose
basics we review in Appendix \ref{subsec:Koopman} for completeness.
The argument starts with the observation that special observables
falling in invariant subspaces of this operator in the space of all
observables obey linear dynamics. Consequently, DMD should recover
the Koopman operator restricted to this subspace if the observables
are taken from such a subspace. 

In this sense, DMD is viewed as an approximate, continuous immersion
of a nonlinear system into an infinite dimensional linear dynamical
system. While such an immersion is not possible for typical nonlinear
systems with multiple limit sets (see \citet{liu23,liu24}), one still
hopes that this approximate immersion is attainable via DMD or EDMD
for nonlinear systems with a single attracting steady state that satisfies
appropriate nondegeneracy conditions (see \citet{kvalheim23}). In
that case, unlike classic local linearization near fixed points, the
linearization via DMD or EDMD is argued to be non-local, as it covers
the full domain of definition of Koopman eigenfunctions spanning the
underlying Koopman-invariant subspace.

However, Koopman eigenfunctions, whose existence, domain of definition
and exact form are a priori unknown for general systems, are notoriously
difficult--if not impossible--to determine accurately from data.
More importantly, even if Koopman-invariant subspaces of the observable
space were known, any countable set of generically chosen observables
would lie outside those subspaces with probability one. As a consequence,
DMD eigenvectors (which are generally argued to be approximations
of Koopman eigenfunctions and can be used to compute Koopman modes\footnote{Assuming that each coordinate component $x_{j}$ of the full phase
space vector $\mathbf{x}$ of system (\ref{eq:nonlinear system0})
falls in a span of the same set of Koopman eigenfunctions $\left\{ \phi_{1}(\mathbf{x}),\ldots,\phi_{N}(\mathbf{x})\right\} $,
one defines the $j^{th}$ Koopman mode associated with $x_{j}$ as
the vector $\mathbf{v}_{j}=\left(v_{1j},\dots,v_{Nj}\right)^{\mathrm{T}}$
for which $x_{j}=\sum_{i=1}^{N}v_{ij}\phi_{i}(x)$ (see, e.g., \citet{williams15}).}) would also lie outside Koopman-invariant subspaces, given that such
eigenvectors are just linear combinations of the available observables.
Consequently, practically observed data sets would fall under the
realm of Koopman-based explanation for DMD with probability zero.
This is equally true for EDMD, whose flexibility in choosing the function
set $\mathbf{K}(\boldsymbol{\phi}(\mathbf{x}))$ of observables also
introduces further user-dependent heuristics beyond the dimension
$d$ of the DMD. 

One may still hope that by enlarging the dimension $d$ of observables
in DMD and enlarging the function library $\mathbf{K}(\boldsymbol{\phi}(\mathbf{x}))$
in EDMD, the optimization involved in these methods brings DMD and
EDMD eigenvectors closer and closer to Koopman eigenfunctions. The
required enlargements, however, may mean hundreds or thousand of dimensions
even for dynamical system governed by simple, low-dimensional ODEs
(\citet{williams15b}). These enlargements succeed in fitting linear
systems closely to sets of observer trajectories, but they also unavoidably
lead to overfits that give unreliable predictions for initial conditions
not used in the training of DMD or EDMD. Indeed, the resulting large
linear systems can perform substantially worse in prediction than
much lower dimensional linear or nonlinear models obtained from other
data-driven techniques (see, e.g., \citet{alora23b}). 

Similar issues arise in justifying the kernel-based EDMD of \citet{williams15b}
based on the Koopman operator. Additionally, the choice of the kernel
function that represents the inner product of the now implicitly defined
polynomial basis functions remains heuristic and problem-dependent.
Again, the accuracy of the procedure is not guaranteed, as available
observer data is generically not in a Koopman eigenspace. As \citet{williams15b}
write: ``Like most existing data-driven methods, there is no guarantee
that the kernel approach will produce accurate approximations of even
the leading eigenvalues and modes, but it often appears to produce
useful sets of modes in practice if the kernel and truncation level
of the pseudoinverse are chosen properly.'' 

Finally, a lesser known limitation of the Koopman-based approach to
DMD is the limited domain in the phase space over which Koopman eigenfunctions
(and hence their corresponding invariant subspaces) are typically
defined in the observable space. Specifically, at least one principal
Koopman eigenfunction necessarily blows up near basin boundaries of
attracting and repelling fixed points and periodic orbits (see Proposition
\ref{prop:Koopman blow-up} of our Appendix \ref{subsec:Koopman-eigenfunctions}
for a precise statement and Theorem 3 of \citet{kvalheim23} for a
more general related result). 

Expansions of observables in terms of such blowing-up eigenfunctions
have even smaller domains of convergence, as was shown explicitly
in a simple example by \citet{page19}. This is a fundamental obstruction
to the often envisioned concept of global linear models built of different
Koopman eigenfunctions over multiple domains of attraction (see, e.g.,
\citet{williams15}, p. 1309). While it is broadly known that such
models would be discontinuous along basin boundaries (\citet{liu23,liu24,kvalheim23}),
it is rarely noted (see \citet{kvalheim23} for such a rare exception)
that such models would also generally blow up at those boundaries
and hence would become unmanageable even before reaching the boundaries.

For these reasons, an alternative mathematical foundation for DMD
is desirable. Ideally, such an approach should be defined on an equal
or lower dimensional space, rather than on higher or even infinite-dimensional
spaces, as suggested by the Koopman-based view on DMD. This should
help in avoiding overfitting and computational difficulties. Additionally,
an ideal treatment of DMD should also provide specific non-degeneracy
conditions on the underlying dynamical system, on the available observables,
and on the specific data to be used in the DMD procedure. 

In this paper, we develop a treatment of DMD that satisfies these
requirements. This enables to derive conditions for DMD to approximate
the dominant linearized observable dynamics near hyperbolic fixed
points and periodic orbits of finite- and infinite-dimensional dynamical
systems. 

Our approach to DMD also leads to a refinement of DMD which we call
data-driven linearization (or DDL, for short). DDL effectively carries
out exact local linearization via nonlinear coordinate changes on
a lower-dimensional attracting invariant manifold (spectral submanifold)
of the dynamical system. We illustrate the increased accuracy and
domain of validity of DDL models relative to those obtained from DMD
and EDMD on examples of autonomous and forced dynamical systems.

\section{A simple justification for the DMD algorithm \label{sec:DMD-in-phase-space}}

Here we give an alternative interpretation of DMD and EDMD as approximate
models for a dynamical system known through a set of observables.
The main idea (to be made more precise shortly) is to view DMD executed
on $d$ observables $\phi_{1},\ldots,\phi_{d}$ as a model reduction
tool that captures the leading-order dynamics of $n\geq d$ phase
space variables along a $d$-dimensional slow manifold in terms of
$\phi_{1},\ldots,\phi_{d}$. 

Such manifolds arise as slow spectral submanifolds (SSMs) under weak
non-degeneracy assumptions on the linearized spectrum at stable hyperbolic
fixed points of the $n$-dimensional dynamical system (see \citet{cabre03,haller16,cenedese22a}).
Specifically, a slow SSM $\mathcal{W}\left(E\right)$ is tangent to
the real eigenspace $E$ spanned by the $d$ slowest decaying linearized
modes at the fixed point. If $m$ sample trajectories $\left\{ \mathbf{x}^{j}(t)\right\} _{j=1}^{m}$
are released from a set of initial conditions $\left\{ \mathbf{x}^{j}(0)\right\} _{j=1}^{m}$at
time $t=0$, then due to their fast decay along the remaining fast
spectral subspace $F$ , the $\mathbf{x}^{j}(t)$ trajectories will
become exponentially close to $\mathcal{W}\left(E\right)$ by some
time $t=t_{0}\geq0$, and closely synchronize with its internal dynamics,
as seen in Fig. \ref{fig: geometry of DMD}.

\begin{figure}[H]
\centering{}\includegraphics[width=0.8\textwidth]{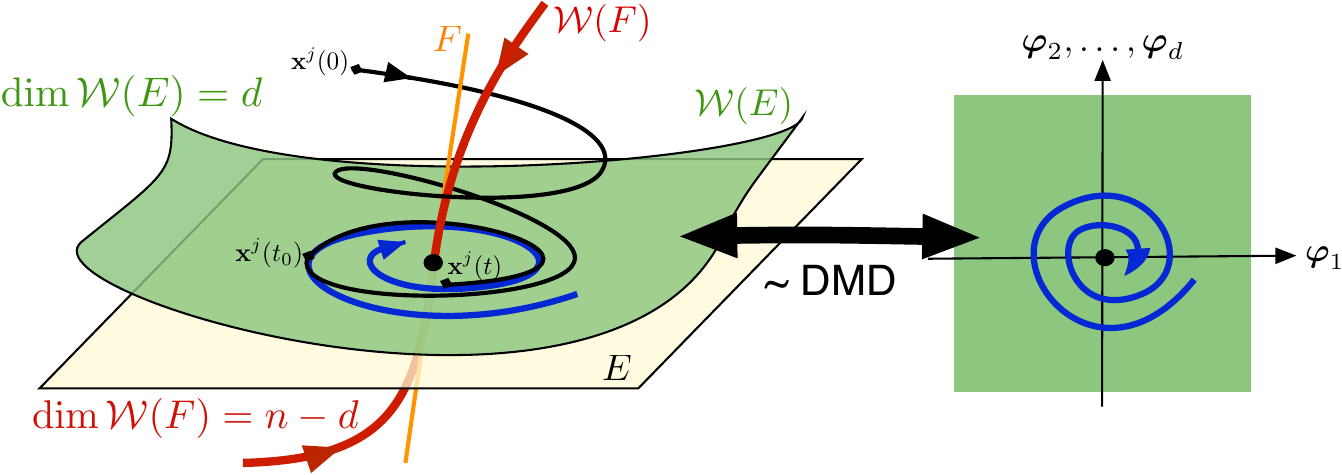}\caption{The geometric meaning of DMD performed on $d$ observables, $\phi_{1}(\mathbf{x}),\ldots,\phi_{d}(\mathbf{x})$,
after initial fast transients die out at an exponential rate in the
data. DMD then identifies the leading-order (linear) dynamics on a
$d$-dimensional attracting spectral submanifold (SSM) $\mathcal{W}\left(E\right)$
tangent to the $d$-dimensional slow spectral subspace $E$. These
linearized dynamics can be expressed in terms of the SSM-restricted
observables $\boldsymbol{\varphi}=\boldsymbol{\phi}\vert_{\mathcal{W}\left(E\right)}$.
Also shown is the spectral subspace $F$ of the faster decaying modes
and its associated nonlinear continuation, the fast spectral subspace
$\mathcal{W}\left(F\right)$. \label{fig: geometry of DMD}}
\end{figure}

If $\mathcal{W}\left(E\right)$ admits a non-degenerate parametrization
with respect to the observables $\phi_{1},\ldots,\phi_{d}$, then
one can pass to these observables as new coordinates in which to approximate
the leading order, linearized dynamics inside $\mathcal{W}\left(E\right)$.
Specifically, DMD executed over time the interval $\left[t_{0},t_{0}+\Delta t\right]$
provides the closest linear fit to the reduced dynamics of $\mathcal{W}\left(E\right)$
from the available observable histories $\phi_{1}\left(\mathbf{x}^{j}(t)\right),\ldots,\phi_{d}\left(\mathbf{x}^{j}(t)\right)$
for $t\in\left[t_{0},t_{0}+\Delta t\right]$, as illustrated in Fig.
\ref{fig: geometry of DMD}. This fit is a close representation of
the actual linearized dynamics on the SSM $\mathcal{W}\left(E\right)$
if the trajectory data $\left\{ \mathbf{x}^{j}(t)\right\} _{j=1}^{m}$
is sufficiently diverse for $t\geq t_{0}$. The resulting DMD model
with coefficient matrix $\mathcal{\boldsymbol{\mathcal{D}}}$ will
then be smoothly conjugate to the linearized reduced dynamics on $\mathcal{W}\left(E\right)$
with an error of the order of the distance of$\left\{ \mathbf{x}^{j}(t)\right\} _{j=1}^{m}$
from $\mathcal{W}\left(E\right)$ between the times $t_{0}$ and $t_{0}+\Delta t$. 

The slow SSM $\mathcal{W}\left(E\right)$ may also contain unstable
modes in applications. Similar results hold for that case as well,
provided that we select a small enough $\Delta t$, ensuring that
the trajectories $\left\{ \mathbf{x}^{j}(0)\right\} _{j=1}^{m}$are
not ejected from the vicinity of the fixed point origin for $t\in\left[t_{0},t_{0}+\Delta t\right]$.
As DMD will show sensitivity with respect to the choice of $t_{0}$
and $\Delta t$ in this case, we will not discuss the justification
of DMD near unstable fixed points beyond Remarks \ref{rem:unstable fixed points}
and \ref{rem:mixed mode SSM in infinite dimensions}.

In the following sections, we make this basic idea more precise both
for finite and infinite-dimensional dynamical systems. This approach
also reveals explicit, previously undocumented non-degeneracy conditions
on the underlying dynamical system, on the available observable functions
and on the specific data set used, under which DMD should give meaningful
results.

\subsection{Justification of DMD for continuous dynamical systems\label{subsec:Justification-of-DMD-cont-systems}}

We start by assuming that the observed dynamics take place in a domain
containing a fixed point, which is assumed to be at the origin, without
loss of generality, i.e., 
\begin{equation}
\mathbf{f}(\mathbf{0})=\mathbf{0}.\label{eq:hyperbolic fixed point}
\end{equation}
We can then rewrite the dynamical system \eqref{eq:nonlinear system0}
in the more specific form
\begin{equation}
\dot{\mathbf{x}}=\mathbf{Ax}+\tilde{\mathbf{f}}(\mathbf{x}),\qquad\mathbf{x}\in\mathcal{\mathbb{R}}^{n},\qquad\mathbf{A}=D\mathbf{f}(\mathbf{0}),\qquad\tilde{\mathbf{f}}(\mathbf{x})=o\left(\left|\mathbf{x}\right|\right),\label{eq:more specific nonlinear system}
\end{equation}
where the classic $o\left(\left|\mathbf{x}\right|\right)$ notation
refers to the fact that $\lim_{\mathbf{x}\to\mathbf{0}}\left[\left|\tilde{\mathbf{f}}(\mathbf{x})\right|/\left|\mathbf{x}\right|\right]=0$. 

Most expositions of DMD methods and their variants do not state assumption
(\ref{eq:hyperbolic fixed point}) explicitly and hence may appear
less restrictive than our treatment here. However, all of them implicitly
assume the existence of such a fixed point, as all of them end up
returning homogeneous linear ODEs or mappings with a fixed point at
the origin. Indeed, all known applications of these methods that produce
reasonable accuracy target the dynamics of ODEs or discrete maps near
their fixed points. 

Assumption (\ref{eq:hyperbolic fixed point}) can be replaced with
the existence of a limit cycle in the original system \eqref{eq:nonlinear system0},
in which case the first return map (or Poincaré map) defined near
the limit cycle will have a fixed point. We give a separate treatment
on justifying DMD as a linearization for such a Poincaré map in Section
\ref{subsec:Justification-of-DMD-disc-systems}. 

We do not advocate, however, the often used procedure of applying
DMD to fit a linear system to the flow map (rather than the Poincaré
map) near a stable limit cycle. Such a fit only produces the desired
limiting periodic behavior if one or more of the DMD eigenvalues are
artificially constrained to be on the complex unit circle by the user
of DMD. This renders the DMD model both structurally unstable and
conceptually inaccurate for prediction. Indeed, the model will approximate
the originally observed limit cycle and convergence to it only within
a measure zero, cylindrical set of its phase space. Outside this set,
all trajectories of the DMD model will converge to some other member
of the infinite family of periodic orbits or invariant tori within
the center subspace corresponding to the unitary eigenvalues. These
periodic orbits or tori have a continuous range of locations and amplitudes,
and hence represent spurious asymptotic behaviors that are not seen
in the original dynamical system \eqref{eq:nonlinear system0}.

In the special case of $d=n$ and for the special observable $\boldsymbol{\phi}(x)=\mathbf{x}$,
the near-linear form of eq. (\ref{eq:more specific nonlinear system})
motivates the DMD procedure because a linear approximation to the
system near $\mathbf{x}=\mathbf{0}$ seems feasible. It is a priori
unclear, however, to what extent the nonlinearities distort the linear
dynamics and how DMD would account for that. Additionally, in a data-driven
analysis, choosing the full phase space variable $\mathbf{x}$ as
the observable $\boldsymbol{\phi}(\mathbf{x})$ is generally unrealistic.
For these reasons, a mathematical justification of DMD requires further
assumptions, as we discuss next.

Let $\lambda_{1},\ldots,\lambda_{n}\in\mathbb{C}$ denote the eigenvalues
of $A$ and let $\mathbf{e}_{1},\ldots,\mathbf{e}_{n}\in\mathbb{C}^{n}$
denote the corresponding generalized eigenvectors. We assume that
at least one of the modes of the linearized system at $\mathbf{x}=\mathbf{0}$
decays exponentially and there is at least one other mode that decays
slower or even grows. More specifically, for some positive integer
$d<n$, we assume that the spectrum of $A$ can be partitioned as
\begin{equation}
\mathrm{Re}\text{\ensuremath{\lambda_{n}}}\leq\ldots\leq\mathrm{Re}\text{\ensuremath{\lambda_{d+1}} }<\mathrm{Re}\text{\ensuremath{\lambda_{d}} }\leq\ldots\leq\mathrm{Re}\text{\ensuremath{\lambda_{1}<0}. }\label{eq:eigenvalue configuration for E}
\end{equation}
This guarantees the existence of a $d$-dimensional, normally attracting\emph{
slow spectral subspace}
\begin{equation}
E=\mathrm{span\,\left\{ \mathrm{Re}\,\text{\ensuremath{\mathbf{e}_{1}},\,}\mathrm{Im\,}\text{\ensuremath{\mathbf{e}_{1}},}\ldots,\mathrm{Re\,}\text{\ensuremath{\mathbf{e}_{d}},\,}\mathrm{Im}\,\text{\ensuremath{\mathbf{e}_{d}}}\right\} }\label{eq:spectral subspace}
\end{equation}
for the linearized dynamics, with linear decay rate towards $E$ strictly
dominating all decay rates inside $E$.  Note that the set of vectors ${\mathrm{Re }\  \mathbf{e}_1, \mathrm{Im }\  \mathbf{e}_1, …, \mathrm{Re }\  \mathbf{e}_d, \mathrm{Im } \ \mathbf{e}_d}$ is, in general, not linearly independent, but they span a $d$-dimensional subspace. We also define the (real)
spectral subspace of faster decaying linear modes:
\begin{equation}
F=\mathrm{span\,\left\{ \mathrm{Re}\,\text{\ensuremath{\mathbf{e}_{d+1}},\,}\mathrm{Im\,}\text{\ensuremath{\mathbf{e}_{d+1}},}\ldots,\mathrm{Re\,}\text{\ensuremath{\mathbf{e}_{n}},\,}\mathrm{Im}\,\text{\ensuremath{\mathbf{e}_{n}}}\right\} }.\label{eq:spectral subspace-1}
\end{equation}

We will also use matrices containing the left and right eigenvectors
of the operator $\mathbf{A}$ and of its restrictions, $\mathbf{A}\vert_{E}$
and $\mathbf{A}\vert_{F}$, to its spectral subspaces $E$ and $F$,
respectively. Specifically, we let

\begin{align}
\mathbf{T} & =\left[\mathbf{T}_{E},\mathbf{T}_{F}\right],\qquad\mathbf{P}=\left[\begin{array}{c}
\mathbf{P}_{E}\\
\mathbf{P}_{F}
\end{array}\right],\label{eq:Tdef}
\end{align}
where the columns of $\mathbf{T}_{E}\in\mathbb{R}^{n\times d}$ are
the real and imaginary parts of the generalized right eigenvectors
of $\mathbf{A}\vert_{E}$ and the columns of $\mathbf{T}_{F}\in\mathbb{R}^{n\times(n-d)}$
are defined analogously for $\mathbf{A}\vert_{F}.$ Similarly, the
rows of $\mathbf{P}_{E}\in\mathbb{R}^{d\times n}$ are the real and
imaginary parts of the generalized left eigenvectors of $\mathbf{A}\vert_{E}$
and the rows of $\mathbf{P}_{F}\in\mathbb{R}^{(n-d)\times d}$ are
defined analogously for$\mathbf{A}\vert_{F}.$ Under assumption (\ref{eq:eigenvalue configuration for E}),
$F$ is always a fast spectral subspace, containing all trajectories
of the linearized system that decay faster to the origin as any trajectory
in $E.$ 

We will use the notation 
\begin{equation}
\mathbf{X}=\left[\mathbf{x}^{1}(t_{0}),\ldots,\mathbf{x}^{m}(t_{0})\right],\quad\hat{\mathbf{X}}=\left[\mathbf{F}^{\Delta t}(\mathbf{x}^{1}(t_{0})),\ldots,\mathbf{F}^{\Delta t}\left(\mathbf{x}^{m}(t_{0})\right)\right]\label{eq:trajectory data matrices}
\end{equation}
for trajectory data in the underlying dynamical system (\eqref{eq:nonlinear system0})
on which the observable data matrices $\boldsymbol{\Phi}$ and $\hat{\boldsymbol{\Phi}}$
defined in (\ref{eq:DMD data matrix definitions}) are defined. In
truly data-driven applications, the matrices $\mathbf{X}$ and $\hat{\mathbf{X}}$
are not known. We will nevertheless use them to make precise statements
about a required dominance of the slow linear modes of $E$ in the
available data. Such a dominance will arise for generic initial conditions
if one selects the initial conditions $\mathbf{x}^{1}(t_{0}),\ldots,\mathbf{x}^{m}(t_{0})$
after initial fast transients along $F$ have died out. This can be
practically achieved by initializing $\mathbf{x}^{1}(t_{0}),\ldots,\mathbf{x}^{m}(t_{0})$
after a linear spectral analysis of the observable data matrix $\boldsymbol{\Phi}$
only returns a number of dominant frequencies consistent with a $d$-dimensional
SSM.

We now state a theorem that provides a general justification for the
DMD procedure under explicit nondegeneracy conditions and with specific
error bounds. Specifically, we give a minimal set of conditions under
which DMD can be justified as an approximate, leading-order, $d$-dimensional
reduced-order model for an nonlinear system of dimension $n\geq d$
near its fixed point. Based on relevance for applications, we only
state Theorem \ref{thm: DMD} for stable hyperbolic fixed points,
but discuss subsequently in Remark \ref{rem:unstable fixed points}
its extension to unstable fixed points.
\begin{thm}
\emph{\label{thm: DMD}{[}}\textbf{\emph{Justification of DMD for
ODEs with stable hyperbolic fixed points{]}}}\emph{ Assume that}
\end{thm}
\begin{lyxlist}{00.00.0000}
\item [{\emph{(A1)}}] \emph{The $\mathbf{x}=\mathbf{0}$ is a stable hyperbolic
fixed point of system }\eqref{eq:more specific nonlinear system}
with a spectral gap, i.e., the spectrum $\mathrm{Spect}\left[\mathbf{A}\right]$
satisfies eq. (\ref{eq:eigenvalue configuration for E}). 
\item [{\emph{(A2)}}] \emph{$\mathbf{f}\in C^{2}$ in a neighborhood of
the origin.}
\item [{\emph{(A3)}}] \emph{For some integer $d\in\left[1,n\right]$, a
$d$-dimensional observable function $\boldsymbol{\phi}\in C^{2}$
and the $d$-dimensional slow spectral subspace $E$ of the hyperbolic
fixed point $\mathbf{x}=\mathbf{0}$ of system (}\ref{eq:more specific nonlinear system}\emph{)
satisfy the non-degeneracy condition 
\begin{equation}
\mathrm{rank\,}\left[D\boldsymbol{\phi}\left(\mathbf{0}\right)\vert_{E}\right]=d.\label{eq:nondegeneracy of observables for DMD1}
\end{equation}
}
\item [{\emph{(A4)}}] \emph{The data matrices $\boldsymbol{\Phi}$ and
$\hat{\boldsymbol{\Phi}}$ are non-degenerate and the initial conditions
in $\mathbf{X}$ and $\hat{\mathbf{X}}$ have been selected after
fast transients from the modes outside $E$ have largely died out,
i.e., 
\begin{equation}
\mathrm{rank_{\mathrm{row}}\,}\boldsymbol{\Phi}=d,\qquad\left|\mathbf{P}_{F}^{\mathrm{}}\mathbf{X}\right|,\left|\mathbf{P}_{F}^{\mathrm{}}\hat{\mathbf{X}}\right|\leq\left|\mathbf{P}_{E}^{\mathrm{}}\mathbf{X}\right|^{1+\beta},\label{eq:nondegeneracy assumption on data1}
\end{equation}
for some $\beta\in(0,1]$.}
\end{lyxlist}
\emph{Then the DMD computed from $\boldsymbol{\Phi}$ and $\hat{\boldsymbol{\Phi}}$
yields a matrix $\boldsymbol{\mathcal{D}}$ that is locally topologically
conjugate with order $\mathcal{O}\left(\left|\mathbf{P}_{E}^{\mathrm{}}\mathbf{X}\right|^{\beta}\right)$
error to the linearized dynamics on a $d$-dimensional, slow attracting
spectral submanifold} $\mathcal{W}(E)\in C^{1}$ tangent to $E$ at
$\mathbf{x}=\mathbf{0}$\emph{. Specifically, we have 
\begin{equation}
\boldsymbol{\mathcal{D}}=D\boldsymbol{\phi}(\mathbf{0})\mathbf{T}_{E}e^{\mathbf{\Lambda}_{E}\Delta t}\left(D\boldsymbol{\phi}(\mathbf{0})\mathbf{T}_{E}\right)^{-1}+\mathcal{O}\left(\left|\mathbf{P}_{E}^{\mathrm{}}\mathbf{X}\right|^{\beta}\right).\label{eq:D matrix from data1}
\end{equation}
}
\begin{proof}
Under assumptions (A1) and (A2), any trajectory in a neighborhood
of the origin in the nonlinear system \eqref{eq:more specific nonlinear system}
converges at an exponential rate $e^{\mathrm{Re}\text{\ensuremath{\lambda_{d+1}} }t}$
to a $d$-dimensional attracting spectral submanifold $\mathcal{W}(E)$
tangent to a $d$-dimensional attracting slow spectral subspace $E$
of the linearized system at the origin. This follows from the $C^{1}$
linearization theorem of \citet{hartman60}, which is applicable to
$C^{2}$ dynamical systems with a stable hyperbolic fixed point. Under
assumption (A3), the $d$-dimensional observable function $\boldsymbol{\phi}(\mathbf{x})$
restricted to $\mathcal{W}(E)$ can be used to parametrize $\mathcal{W}(E)$
near the origin, and hence a $d$-dimensional, self-contained nonlinear
dynamical system can be written down for the restricted observable
$\boldsymbol{\varphi}=\boldsymbol{\phi}\vert_{\mathcal{W}(E)}$ along
$\mathcal{W}(E)$. Under the first assumption in (A4), the available
observational data matrices $\boldsymbol{\Phi}$ and $\hat{\boldsymbol{\Phi}}$
are rich enough to characterize the reduced dynamics on $\mathcal{W}(E)$.
Under the second assumption in (A4), transients from the faster modes
outside $E$ have largely died out before the selection of the initial
conditions in $\mathbf{X}$, so that the linear part of the dynamics
on $\mathcal{W}(E)$ can be approximately inferred from $\boldsymbol{\Phi}$
and $\hat{\boldsymbol{\Phi}}$. In that case, up to an error proportional
to the distance of the training data from $\mathcal{W}(E)$, the matrix
$\boldsymbol{\mathcal{D}}\in\mathbb{R}^{d\times d}$ returned by DMD
is similar to the time-$\Delta t$ flow map of the linearized flow
of the underlying dynamical system restricted to $\mathcal{W}(E)$.
This linearized flow then acts as a local reduced-order model with
which nearby trajectory observations synchronize exponentially fast
in the observable space. We give a more detailed proof of the theorem
in Appendix \eqref{sec: Proof of Theorem 1}. 
\end{proof}
\begin{rem}
\label{rem:unstable fixed points}In Theorem \ref{thm: DMD}, we can
replace assumption (A1) with
\begin{equation}
\mathrm{Re}\text{\ensuremath{\lambda_{n}}}\leq\ldots\leq\mathrm{Re}\text{\ensuremath{\lambda_{d+1}} }<0,\qquad\mathrm{Re}\text{\ensuremath{\lambda_{d+1}} }<\mathrm{Re}\text{\ensuremath{\lambda_{d}} }\leq\ldots\leq\mathrm{Re}\text{\ensuremath{\lambda_{1}}, }\quad\mathrm{Re}\text{\ensuremath{\lambda_{j}\neq0,\quad j=1,\ldots,n.} }\label{eq:eigenvalue configuration for E-2}
\end{equation}
This means that the $\mathbf{x}=\mathbf{0}$ fixed point is only assumed
hyperbolic with a spectral gap and $\mathbf{A}$ has an attracting
$d$-dimensional spectral subspace $E$ that possibly contains some
instabilities, i.e., eigenvalues with positive real parts. Then the
statements of Theorem \ref{thm: DMD} still hold, but $\mathcal{W}(E)$
will only be guaranteed $C^{1}$ at $\mathbf{x}=\mathbf{0}$ and Hölder-continuous
at other points near the fixed point. This follows by replacing the
linearization theorem of \citet{hartman60} with that of \citet{vanstrien90},
which still enables us to use eq. (\ref{eq:Hartman linearization})
in the proof. Therefore, slow subspaces $E$ containing a mixture
of stable and unstable modes can also be allowed, as long as $F$
contains only fast modes consistent with the splitting assume in eq.
(\ref{eq:eigenvalue configuration for E}). In that case, however,
the time $t_{0}+\Delta t$ must be chosen carefully to ensure that
$\left|\mathbf{P}_{F}^{\mathrm{}}\hat{\mathbf{X}}\right|\leq\left|\mathbf{P}_{E}^{\mathrm{}}\mathbf{X}\right|^{1+\beta}$
still holds, i.e., the data used in DMD still samples a neighborhood
of the origin.
\end{rem}
\begin{rem}
In related work, \citet{Bollt_2018} construct the transformation
relating a pair of conjugate dynamical systems based on a limited
set of matching Koopman eigenfunctions, which are either known explicitly
or constructed from EDMD with dictionary learning (EDMD-DL; see \citet{Li_2017}).
In principle, this could be used to construct linearizing transformation
as well. However, even when the eigenfunctions are approximated from
data, the approach assumes that the linearized system, as well as
a linearized trajectory and its preimage under the linearization,
are available. As these assumptions are not satisfied in practice,
only very simple and low-dimensional analytic examples are treated
by \citet{Bollt_2018}.
\end{rem}
In Appendix \eqref{sec: Proof of Theorem 1}, Remarks \ref{rem:oscillatory modes}-\ref{rem:higher smoothness under nonresonance}
summarize technical points on the application and possible further
extensions of Theorem \ref{thm: DMD}. In practice, Theorem \ref{thm: DMD}
provides previously unspecified non-degeneracy conditions on the linear
part of the dynamical system to be analyzed via DMD (assumption (A1)),
on the regularity of the nonlinear part of the system (assumption
(A2)), on the type of observables available for the analysis (assumption
(A3)) and on the specific observable data used in the analysis (assumption
(A4)). The latter assumption requires that there have to be at least
as many independent observations in time as observables. This specifically
excludes the popular use of tall $\boldsymbol{\Phi}$ observable data
matrices which provide more free parameters to pattern-match observational
data but will also lead to an overfit that diminishes the predictive
power of the DMD model on initial conditions not used in its training.

To illustrate these points, we demonstrate the necessity of assumptions
(A2)-(A4) of Theorem \ref{thm: DMD} in Appendix \ref{sec:Necessity-of-assumptions}
on simple examples. 

\subsection{Justification of DMD for discrete and for time-periodic continuous
dynamical systems\label{subsec:Justification-of-DMD-disc-systems}}

The linearization results we have applied to deduce Theorem \ref{thm: DMD}
are equally valid for discrete dynamical systems defined by iterated
mappings. Such mappings are of the form
\begin{equation}
\mathbf{x}_{n+1}=\mathbf{f}(\mathbf{x}_{n})=\mathbf{A}\mathbf{x}_{n}+\tilde{\mathbf{f}}(\mathbf{x}_{n}),\qquad\mathbf{x}_{j}\in\mathbb{R}^{n},\qquad\mathbf{A}\in\mathbb{R}^{n\times n},\qquad\tilde{\mathbf{f}}(\mathbf{x})=o\left(\left|\mathbf{x}\right|\right).\label{eq:discrete dynamical system}
\end{equation}
 We will use a similar ordering for the eigenvalues of $\mathbf{A}$
as in the continuous time case:

\begin{equation}
\left|\text{\ensuremath{\lambda_{n}}}\right|\leq\ldots\leq\left|\text{\ensuremath{\lambda_{d+1}} }\right|<\left|\text{\ensuremath{\lambda_{d}} }\right|\leq\ldots\leq\left|\text{\ensuremath{\lambda_{1}} }\right|<1.\label{eq:eigenvalue configuration for E-1}
\end{equation}
 As in the continuous time case, we will use the observable data matrices 

\begin{equation}
\boldsymbol{\Phi}=\boldsymbol{\phi}\left(\mathbf{X}\right),\qquad\hat{\boldsymbol{\Phi}}=\boldsymbol{\phi}\left(\mathbf{f}\left(\mathbf{X}\right)\right),\label{eq:DMD data matrix definitions-1}
\end{equation}
with the initial conditions for the map $\mathbf{f}$ stored in $\mathbf{X}$. 

With these ingredients, we need only minor modifications in the assumptions
of the theorems that account for the usual differences between the
spectrum of an ODE and a map.
\begin{thm}
\emph{\label{thm:DMD-maps}{[}}\textbf{\emph{Justification of DMD
for maps with stable hyperbolic fixed points}}\emph{{]} Assume that }
\end{thm}
\begin{lyxlist}{00.00.0000}
\item [{(A1)}] \emph{$\mathbf{x}=\mathbf{0}$ is a stable hyperbolic fixed
point of system }(\ref{eq:discrete dynamical system}), i.e., assumption
(\ref{eq:eigenvalue configuration for E-1}) holds.
\item [{\emph{(A2)}}] \emph{In eq. (}\ref{eq:discrete dynamical system}),\emph{
$\tilde{\mathbf{f}}\in C^{2}$ holds in a neighborhood of the origin.}
\item [{\emph{(A3)}}] \emph{For some integer $d\in\left[1,n\right]$, a
$d$-dimensional observable function $\boldsymbol{\phi}\in C^{2}$
and the $d$-dimensional slow spectral subspace $E$ of the hyperbolic
fixed point $\mathbf{x}=\mathbf{0}$ of system (}\ref{eq:discrete dynamical system}\emph{)
satisfy the non-degeneracy condition 
\begin{equation}
\mathrm{rank\,}\left[D\boldsymbol{\phi}\left(\mathbf{0}\right)\vert_{E}\right]=d.\label{eq:nondegeneracy of observables for DMD1-1}
\end{equation}
}
\item [{\emph{(A4)}}] \emph{The data matrices $\boldsymbol{\Phi}$ and
$\hat{\boldsymbol{\Phi}}$ collected from iterations of system (}\ref{eq:discrete dynamical system}\emph{)
are non-degenerate and are dominated by data near $E$, i.e., 
\begin{equation}
\mathrm{rank_{\mathrm{row}}\,}\boldsymbol{\Phi}=d,\qquad\left|\mathbf{P}_{F}^{\mathrm{}}\mathbf{X}\right|,\left|\mathbf{P}_{F}^{\mathrm{}}\hat{\mathbf{X}}\right|\leq\left|\mathbf{P}_{E}^{\mathrm{}}\mathbf{X}\right|^{1+\beta},\label{eq:nondegeneracy assumption on data1-1}
\end{equation}
for some $\beta\in(0,1)$.}
\end{lyxlist}
\emph{Then the DMD computed from $\boldsymbol{\Phi}$ and $\hat{\boldsymbol{\Phi}}$
yields a matrix $\boldsymbol{\mathcal{D}}$ that is locally topologically
conjugate with order $\mathcal{O}\left(\left|\mathbf{P}_{E}^{\mathrm{}}\mathbf{X}\right|^{\beta}\right)$
error to the linearized dynamics on a $d$-dimensional attracting
spectral submanifold }$\mathcal{W}(E)\in C^{1}$\emph{ tangent to
$E$ at $\mathbf{x}=\mathbf{0}$. Specifically, we have 
\begin{equation}
\boldsymbol{\mathcal{D}}=D\boldsymbol{\phi}(\mathbf{0})\mathbf{T}_{E}\mathbf{\Lambda}_{E}\left(D\boldsymbol{\phi}(\mathbf{0})\mathbf{T}_{E}\right)^{-1}+\mathcal{O}\left(\left|\mathbf{P}_{E}^{\mathrm{}}\mathbf{X}\right|^{\beta}\right).\label{eq:D matrix from data1-1}
\end{equation}
The spectral submanifold $\mathcal{W}(E)$ and its reduced dynamics
are of class $C^{1}$ at the origin, and at least Hölder continuous
in a neighborhood of the origin.}
\begin{proof}
The proof is identical to the proof of Theorem \ref{thm: DMD} but
uses the discrete version of the linearization result by \citet{hartman60}
for stable hyperbolic fixed points of maps.
\end{proof}
Theorem \ref{thm: DMD} can be immediately applied to justify DMD
as a linearization tool for period-one maps (or Poincaré maps) of
time-periodic, non-autonomous dynamical systems near their periodic
orbits. This requires the data matrices \emph{$\boldsymbol{\Phi}$
}and\emph{ $\hat{\boldsymbol{\Phi}}$} to contain trajectories of
such a Poincaré map. Remark \pageref{rem:unstable fixed points} on
the treatment of slow spectral subspaces $E$ containing possible
instabilities also applies here under the modified assumption
\begin{equation}
\left|\text{\ensuremath{\lambda_{n}}}\right|\leq\ldots\leq\left|\text{\ensuremath{\lambda_{d+1}} }\right|<1,\qquad\left|\text{\ensuremath{\lambda_{d+1}} }\right|<\left|\text{\ensuremath{\lambda_{d}} }\right|\leq\ldots\leq\left|\text{\ensuremath{\lambda_{1}} }\right|,\quad\left|\text{\ensuremath{\lambda_{j}} }\right|\neq1,\quad j=1,\ldots,n,\label{eq:eigenvalue configuration for E-1-1}
\end{equation}
which only requires the fixed point to be hyperbolic and $\mathbf{A}$
to have a $d$-dimensional normally attracting subspace.

\subsection{Justification of DMD for infinite-dimensional dynamical systems\label{subsec:Justification-of-DMD-inf-dimensional-systems}}

Most data sets of interest arguably arise from infinite-dimensional
dynamical systems of fluids and solids. Examples include experimental
or numerical data describing fluid motion, continuum vibrations, climate
dynamics or salinity distribution in the ocean. In the absence of
external forcing, these problems are governed by systems of autonomous
nonlinear partial differential equations that can often be viewed
as evolutionary differential equations in a form similar to eq. (\ref{eq:nonlinear system0}),
but defined on an appropriate infinite-dimensional Banach space. Accordingly,
time-sampled solutions of these equations can be viewed as iterated
mappings of the form (\ref{eq:discrete dynamical system}) but defined
on Banach spaces.

Our approach to justifying DMD generally carries over to this infinite-dimensional
setting, as long as the observable vector $\boldsymbol{\phi}(\mathbf{x})$
remains finite-dimensional, and both the Banach space and the discrete
or continuous dynamical system defined on it satisfy appropriate regularity
conditions. These regularity conditions tend to be technical, but
when they are satisfied, they do guarantee the extension of Theorems
\ref{thm: DMD}-\ref{thm:DMD-maps} to Banach spaces. This offers
a justification to use DMD to obtain an approximate finite-dimensional
linear model for the dynamics of the underlying continuum system on
a finite-dimensional attracting slow manifold (or inertial manifold)
in the neighborhood of a non-degenerate stationary solution. 

To avoid major technicalities, we only state here a generalized version
of Theorem \ref{thm: DMD} to justify the use of DMD for observables
defined on Banach spaces for a discrete evolutionary process with
a stable hyperbolic stationary state. We consider mappings of the
form
\begin{equation}
\mathbf{x}_{n+1}=\mathbf{f}(\mathbf{x}_{n})=\mathbf{A}\mathbf{x}_{n}+\tilde{\mathbf{f}}(\mathbf{x}_{n}),\qquad\mathbf{x}_{j}\in\mathcal{B},\qquad\tilde{\mathbf{f}}\colon U\subset\mathcal{B}\to\mathcal{B},\qquad\tilde{\mathbf{f}}(\mathbf{0})=\mathbf{0}\in U,\label{eq:discrete dynamical system-1}
\end{equation}
where $\mathcal{B}$ is a Banach space, $U$ is an open set in $\mathcal{B}$,
and $\mathbf{A}\colon\mathcal{B}\to\mathcal{B}$ is an invertible
linear operator that is bounded in the norm defined on $\mathcal{B}$.
The function $\mathbf{f}$ can be here the time-sampled version of
an infinite-dimensional flow map of an autonomous evolutionary PDE
or the Poincaré map of a time-periodic evolutionary PDE. We assume
that for some $\alpha\in\left(0,1\right)$, $\tilde{\mathbf{f}}\in C^{1,\alpha}(U)$
holds, i.e., $\tilde{\mathbf{f}}$ is (Fréchet-) differentiable in
$U$ and its derivative, $D\tilde{\mathbf{f}}$, is Hölder-continuous
in $\mathbf{x}\in U$ with Hölder exponent $\alpha$. 

The \emph{spectral radius} of $A$ is defined as
\[
\rho(\mathbf{A})=\lim_{k\to\infty}\left|\mathbf{A}^{k}\right|^{\frac{1}{k}}.
\]
We recall that in the special case $\mathcal{B}=\mathbb{R}^{n}$ treated
in Section \ref{subsec:Justification-of-DMD-disc-systems}, we have
$\rho(\mathbf{A})=\max_{1\leq j\leq n}\left|\text{\ensuremath{\lambda_{j}}}\right|$.
For some $\alpha\in\left(0,1\right)$, the linear operator $\mathbf{A}$
is called \emph{$\alpha$-contracting} if 
\begin{equation}
\rho(\mathbf{A})^{1+\alpha}\rho\left(\mathbf{A}^{-1}\right)<1,\label{eq:alpha-contracting condition}
\end{equation}
which can only hold if $\rho(\mathbf{A})<1$ (see \citet{newhouse17}).
Therefore, in the simple case of $\mathcal{B}=\mathbb{R}^{n}$, $A$
is \emph{$\alpha$-contracting} if it is a contraction (i.e., all
its eigenvalues are less than one in norm) and 
\[
\left|\lambda_{1}\right|^{1+\alpha}<\left|\lambda_{n}\right|,
\]
showing that the spectrum of $\mathbf{A}$ is confined to an annulus
of outer radius $\left|\lambda_{1}\right|<1$ and inner radius $\left|\lambda_{1}\right|^{1+\alpha}$.
We can now state our main result on the justification of DMD for infinite-dimensional
discrete dynamical systems.
\begin{thm}
\emph{\label{thm:DMD-infinite-dim}{[}}\textbf{\emph{Justification
of DMD for infinite-dimensional maps with stable hyperbolic fixed
points}}\emph{{]} Assume that }
\end{thm}
\begin{lyxlist}{00.00.0000}
\item [{(A1)}] For some $\alpha\in\left(0,1\right)$, the linear operator
$\mathbf{A}$ is \emph{$\alpha$-contracting }(and hence the $x=0$
fixed point of system (\ref{eq:discrete dynamical system-1}) is linearly
stable).
\item [{\emph{(A2)}}] \emph{In eq. (}\ref{eq:discrete dynamical system-1}),\emph{
}$\tilde{\mathbf{f}}\in C^{1,\alpha}(U)$\emph{ holds in a $U$ neighborhood
of the origin.}
\item [{\emph{(A3)}}] \emph{For some integer $d\in\mathbb{N}^{+}$, there
is a splitting $\mathcal{B}=E\oplus F$ of $\mathcal{B}$ into two
$A$-invariant subspaces $E,F\subset\mathcal{B}$ such that $E$ is
$d$-dimensional and slow, i.e., 
\begin{equation}
\rho\left(\mathbf{A}\vert_{E}\right)<\frac{1}{\rho\left(\mathbf{A}^{-1}\vert_{F}\right)}\label{eq:spectral gap condiition for infinite dimensions}
\end{equation}
Furthermore, a $d$-dimensional observable function $\boldsymbol{\phi}\in C^{2}$
satisfies the non-degeneracy condition 
\begin{equation}
\mathrm{rank\,}\left[D\boldsymbol{\phi}\left(\mathbf{0}\right)\vert_{E}\right]=d.\label{eq:nondegeneracy of observables for DMD1-1-1}
\end{equation}
}
\item [{\emph{(A4)}}] \emph{The data matrices $\boldsymbol{\Phi}$ and
$\hat{\boldsymbol{\Phi}}$ collected from iterations of system (}\ref{eq:discrete dynamical system-1}\emph{)
are non-degenerate and are dominated by data near $E$, i.e., 
\begin{equation}
\mathrm{rank_{\mathrm{row}}\,}\boldsymbol{\Phi}=d,\qquad\left|\mathbf{P}_{F}^{\mathrm{}}\mathbf{X}\right|,\left|\mathbf{P}_{F}^{\mathrm{}}\hat{\mathbf{X}}\right|\leq\left|\mathbf{P}_{E}^{\mathrm{}}\mathbf{X}\right|^{1+\beta},\label{eq:nondegeneracy assumption on data1-1-1}
\end{equation}
for some $\beta\in(0,1)$.}
\end{lyxlist}
\emph{Then the DMD computed from $\boldsymbol{\Phi}$ and $\hat{\boldsymbol{\Phi}}$
yields a matrix $\boldsymbol{\mathcal{D}}$ that is locally topologically
conjugate with order $\mathcal{O}\left(\left|\mathbf{P}_{E}^{\mathrm{}}\mathbf{X}\right|^{\beta}\right)$
error to the linearized dynamics on a $d$-dimensional attracting
spectral submanifold $\mathcal{W}(E)$ tangent to $E$ at $\mathbf{x}=\mathbf{0}$.
Specifically, we have 
\begin{equation}
\boldsymbol{\mathcal{D}}=D\boldsymbol{\phi}(\mathbf{0})\mathbf{T}_{E}\mathbf{\Lambda}_{E}\left(D\boldsymbol{\phi}(\mathbf{0})\mathbf{T}_{E}\right)^{-1}+\mathcal{O}\left(\left|\mathbf{P}_{E}^{\mathrm{}}\mathbf{X}\right|^{\beta}\right).\label{eq:D matrix from data1-1-1}
\end{equation}
 The spectral submanifold $\mathcal{W}(E)$ and its reduced dynamics
are of class $C^{1}$ in a neighborhood of the origin.}
\begin{proof}
The proof follows the steps in the proof of Theorem \ref{thm:DMD-maps}
but uses an infinite-dimensional linearization result, Theorem 3.1
of \citet{newhouse17}, for stable hyperbolic fixed points of maps
on Banach spaces. Specifically, if $\mathbf{A}$ is $\alpha$-contracting,
then \citet{newhouse17} shows the existence of a near-identity linearizing
transformation $\mathbf{x}=\mathbf{y}+\mathbf{h}(\mathbf{y})$ for
the discrete dynamical system (\ref{eq:discrete dynamical system-1})
such that $\mathbf{h}\in C^{1,\alpha}(B)$ holds on a small enough
ball $B\subset U$ centered at $\mathbf{x}=\mathbf{0}$. Using this
linearization theorem instead of its finite-dimensional version from
\citet{hartman60}, we can follow the same steps as in the proof of
Theorem \ref{thm:DMD-maps} to conclude the statement of the theorem.
\end{proof}
In Appendix \ref{sec: Proof of Theorem 1}, Remarks \ref{rem:smoothness in infinite dimensions}-\ref{rem:mixed mode SSM in infinite dimensions}
summarize technical remarks on possible further extensions of Theorem
\ref{thm:DMD-infinite-dim}.

\section{Data-Driven Linearization (DDL)\label{sec:DDL}}

\subsection{Theoretical foundation for DDL}

Based on the results of the previous section, we now refine the first-order
approximation to the linearized dynamics yielded by DMD near a hyperbolic
fixed point. Specifically, we construct the specific nonlinear coordinate
change that linearizes the restricted dynamics on the attracting spectral
submanifold $\mathcal{W}(E)$ illustrated in Fig. \ref{fig: geometry of DMD}.
This classic notion of linearization on $\mathcal{W}(E)$ yields a
$d$-dimensional linear reduced model, which can be of significantly
lower dimension than the original $n$-dimensional nonlinear system.
This is to be contrasted with the broadly pursued Koopman embedding
approach (see, e.g., \citet{rowley09,budisic12,mezic13}), which seeks
to immerse nonlinear systems into linear systems of dimensions substantially
higher (or even infinite) relative to $n$. 

The following result gives the theoretical basis for our subsequent
data-driven linearization (DDL) algorithm. We will use the notation
$C^{a}$ for the class of real analytic functions.
\begin{thm}
\emph{\label{thm:DDL}{[}}\textbf{\emph{DDL principle for ODEs with
a stable hyperbolic fixed points}}\emph{{]} Assume that the $\mathbf{x}=\mathbf{0}$
stable hyperbolic fixed point of system \eqref{eq:more specific nonlinear system}
and the spectrum of $\mathbf{A}$ has a spectral gap as in eq. (\ref{eq:eigenvalue configuration for E}).
Assume further that for some $r\in\mathbb{N}^{+}\cup\left\{ \infty,a\right\} $,
the following conditions are satisfied:}
\end{thm}
\begin{lyxlist}{00.00.0000}
\item [{(B1)}] \emph{$\mathbf{f}_{2}\in C^{r}$ and for all nonnegative
integers $m_{j}$, the nonresonance conditions }
\begin{equation}
\lambda_{k}\neq\sum_{j=1}^{n}m_{j}\lambda_{j},\qquad k=1,\ldots,n,\qquad2\leq\sum_{j=1}^{n}m_{j}\leq Q\leq r,\quad Q:=\left\lfloor \frac{\max_{i}\left|\mathrm{Re}\,\lambda_{i}\right|}{\min_{i}\left|\mathrm{Re}\,\lambda_{i}\right|}\right\rfloor +1,\label{eq:nonresonance condition2-2}
\end{equation}
\emph{hold for the eigenvalues of $\mathbf{A}$, with $\left\lfloor \,\cdot\:\right\rfloor $
referring to the integer part of a positive number.}
\item [{(B2)}] \emph{For some integer $d\in\left[1,n\right]$, a $d$-dimensional
observable function $\boldsymbol{\phi}\in C^{r}$ and the $d$-dimensional
slow spectral subspace $E$ of the stable fixed point $\mathbf{x}=\mathbf{0}$
of system (}\ref{eq:more specific nonlinear system}\emph{) satisfy
the non-degeneracy condition. 
\begin{equation}
\mathrm{rank\,}\left[D\boldsymbol{\phi}\left(\mathbf{0}\right)\vert_{E}\right]=d.\label{eq:nondegeneracy of observables for DMD-1-2}
\end{equation}
}
\end{lyxlist}
\emph{Then the following hold: }

\emph{(i) On the unique $d$-dimensional attracting spectral submanifold
$\mathcal{W}(E)\in C^{r}$ tangent to $E$ at $x=0$, the reduced
observable vector $\boldsymbol{\varphi}=\boldsymbol{\phi}\vert_{\mathcal{W}(E)}$
can be used to describe the reduced dynamics as
\begin{equation}
\dot{\boldsymbol{\varphi}}=\mathbf{B}\boldsymbol{\varphi}+\mathbf{q}\left(\boldsymbol{\varphi}\right),\qquad\mathbf{B}=D\boldsymbol{\phi}(\mathbf{0})\mathbf{T}_{E}\mathbf{\Lambda}_{E}\left(D\boldsymbol{\phi}(\mathbf{0})\mathbf{T}_{E}\right)^{-1},\quad\mathbf{q}\left(\boldsymbol{\varphi}\right)=\mathcal{O}\left(\left|\boldsymbol{\varphi}\right|^{2}\right).\label{eq:gamma ODE}
\end{equation}
(ii) There exists a unique, $C^{r}$ change of coordinates 
\begin{equation}
\boldsymbol{\varphi}=\boldsymbol{\kappa}(\boldsymbol{\gamma})=\boldsymbol{\gamma}+\boldsymbol{\ell}(\boldsymbol{\gamma}),\label{eq:linearizing transformation on W(E)}
\end{equation}
 that transforms the reduced dynamics on $\mathcal{W}(E)$ to its
linearization 
\begin{equation}
\dot{\boldsymbol{\gamma}}=\mathbf{B}\boldsymbol{\gamma}\label{eq:linearized flow on SSM}
\end{equation}
inside the domain of attraction of $\mathbf{x}=\mathbf{0}$ within
the spectral submanifold $\mathcal{W}(E).$ }

\emph{(iii) The transformation (}\ref{eq:linearizing transformation on W(E)}\emph{)
satisfies the $d$-dimensional system of nonlinear PDEs
\begin{equation}
D_{\boldsymbol{\gamma}}\boldsymbol{\ell}(\boldsymbol{\gamma})\mathbf{B}\boldsymbol{\gamma}=\mathbf{B}\boldsymbol{\ell}(\boldsymbol{\gamma})+\mathbf{q}\left(\boldsymbol{\gamma}+\boldsymbol{\ell}(\boldsymbol{\gamma})\right).\label{eq:invariance PDE}
\end{equation}
If $r\in\mathbb{N}^{+}\cup\left\{ \infty\right\} $, solutions of
this PDE can locally be approximated as}
\begin{equation}
\boldsymbol{\ell}(\boldsymbol{\gamma})=\sum_{\left|\mathbf{k}\right|=2}^{r}\mathbf{l}_{\mathbf{k}}\boldsymbol{\gamma}^{\mathbf{k}}+o\left(\left|\boldsymbol{\gamma}\right|^{r}\right),\quad\mathbf{k}\in\mathbb{N}^{d},\quad\mathbf{l}_{\mathbf{k}}\in\mathbb{R}^{d},\quad\boldsymbol{\gamma}^{\mathbf{k}}:=\gamma_{1}^{k_{1}}\cdots\gamma_{d}^{k_{d}}.\label{eq:DDL transformation formal}
\end{equation}
\emph{ If $r=a$, then the local approximation (}\ref{eq:DDL transformation formal})
\emph{can be refined to a convergent Taylor series 
\begin{equation}
\boldsymbol{\ell}(\boldsymbol{\gamma})=\sum_{\left|\mathbf{k}\right|=2}^{\infty}\mathbf{l}_{\mathbf{k}}\boldsymbol{\gamma}^{\mathbf{k}},\quad\mathbf{k}\in\mathbb{N}^{d},\quad\mathbf{l}_{\mathbf{k}}\in\mathbb{R}^{d},\quad\boldsymbol{\gamma}^{\mathbf{k}}:=\gamma_{1}^{k_{1}}\cdots\gamma_{d}^{k_{d}}\label{eq:DDL transformation convergent}
\end{equation}
in a neighborhood of the origin. In either case, the coefficients
$\mathbf{l}_{\mathbf{k}}$ can be determined by substituting the expansion
for $\boldsymbol{\ell}(\boldsymbol{\gamma})$ into the PDE (\ref{eq:invariance PDE}),
equating coefficients of equal monomials $\boldsymbol{\gamma}^{\mathbf{k}}$
and solving the corresponding recursive sequence of $d$-dimensional
linear algebraic equations for increasing $\left|\mathbf{k}\right|$.}
\begin{proof}
The proof builds on the existence of the $d$-dimensional spectral
submanifold $\mathcal{W}\left(E\right)$ guaranteed by Theorem \ref{thm: DMD}.
For a $C^{r}$ dynamical system with $r\in\mathbb{N}^{+}\cup\left\{ \infty,a\right\} $,
$\mathcal{W}\left(E\right)$ is also $C^{r}$ smooth based on the
linearization theorems of \citet{poincare1892} and \citet{sternberg57},
as long as the nonresonance condition (\ref{eq:nonresonance condition2-2})
holds. Condition (\ref{eq:nondegeneracy of observables for DMD-1-2})
then ensures that $\mathcal{W}\left(E\right)$ can be parametrized
locally by the restricted observable vector $\boldsymbol{\varphi}$
and hence its reduced dynamics can be written as a nonlinear ODE for
$\boldsymbol{\varphi}$. This ODE can again be linearized by a near-identity
coordinate change (\ref{eq:linearizing transformation on W(E)}) using
the appropriate linearization theorem of the two cited above. The
result is the restricted linear system (\ref{eq:linearized flow on SSM})
to which the dynamics is $C^{r}$ conjugate within the whole domain
of attraction of the $\boldsymbol{\varphi}=\mathbf{0}$ fixed point
inside $\mathcal{W}\left(E\right)$. The invariance PDE (\ref{eq:invariance PDE})
can be obtained by substituting the linearizing transformation (\ref{eq:linearizing transformation on W(E)})
into the reduced dynamics on $\mathcal{W}\left(E\right)$. This PDE
can then be solved via a Taylor expansion up to order $r$. We give
more a more detailed proof in Appendix \ref{sec:Proof-of-Theorem-DDL}.
\end{proof}
Note that $1:1$ resonances are not excluded by the condition (\ref{eq:nonresonance condition2-2}),
and hence repeated eigenvalues arising from symmetries in physical
systems are still amenable to DDL. Also of note is that the non-resonance
conditions (\ref{eq:nonresonance condition2-2}) do not exclude frequency-type
resonances among imaginary parts of oscillatory eigenvalues. Rather,
they exclude simultaneous resonances of the same type between the
real and the imaginary parts of the eigenvalues. Such resonances will
be absent in data generated by generic oscillatory systems. 

Assuming hyperbolicity is essential for Theorem \ref{thm:DDL} to
hold, since in this case the linearization is the same as transforming
the dynamics to the Poincaré-normal form. For a non-hyperbolic fixed
point, this normal form transformation results in nonlinear dynamics
on the center manifold. This would, however, only arise in highly
non-generic systems, precisely tuned to be at criticality. Since this
is unlikely to happen in experimentally observed or numerically simulated
systems, the hyperbolicity assumption is not restrictive. 

Finally, under the conditions of Theorem \ref{thm:DMD-infinite-dim},
the DDL results of Theorem \ref{thm:DDL} also apply to data from
infinite-dimensional dynamical systems, such as the fluid sloshing
experiments we will analyze using DDL in Section \ref{subsec:Sloshing}.
In practice, the most restrictive condition of Theorem \ref{thm:DMD-infinite-dim}
is (A1), which requires the solution operator to have a spectrum uniformly
bounded away from zero. Such uniform boundedness is formally violated
in important classes of infinite-dimensional evolution equations,
presenting a technical challenge for the direct applications of SSM
results to certain delay-differential equations (see \citet{szaksz24})
and partial differential equations (see, e.g., \citet{kogelbauer18}
and \citet{buza24}). However, this challenge only concerns rigorous
conclusions on the existence and smoothness of a finite-dimensional,
attracting SSM. If the existence of such an SSM is convincingly established
from an alternative mathematical theory (as is \citet{buza24}) or
inferred from data (as in \citet{szaksz24}), then the DDL algorithm
based on Theorem \ref{thm:DDL} can be used to obtain a data-driven
linearization of the dynamics on that SSM. 

\subsection{DDL vs. EDMD\label{subsec:DDL-vs.-EDMD}}

Here we examine whether there is a possible relationship between DDL
and the extended DMD (or EDMD) algorithm of \citet{williams15}. For
simplicity, we assume analyticity for the dynamical system ($r=a$)
and hence we can write the inverse of the linearizing transformation
(\ref{eq:DDL transformation formal}) behind the DDL algorithm as
a convergent Taylor expansion of the form
\begin{equation}
\boldsymbol{\gamma}=\boldsymbol{\kappa}^{-1}\left(\boldsymbol{\varphi}\right)=\boldsymbol{\varphi}+\sum_{\left|\mathbf{k}\right|=2}^{\infty}\mathbf{q}_{\mathbf{k}}\boldsymbol{\varphi}^{\mathbf{k}}.\label{eq: inverse linearizing transformation}
\end{equation}
We then differentiate this equation in time to obtain from the linearized
equation (\ref{eq:linearized flow on SSM}) a $d$-dimensional system
of equations
\[
\dot{\boldsymbol{\varphi}}+\sum_{\left|\mathbf{k}\right|=2}^{\infty}\mathbf{q}_{\mathbf{k}}\frac{d}{dt}\boldsymbol{\varphi}^{\mathbf{k}}=\mathbf{B}\boldsymbol{\varphi}+\sum_{\left|\mathbf{k}\right|=2}^{\infty}\mathbf{B}\mathbf{q}_{\mathbf{k}}\boldsymbol{\varphi}^{\mathbf{k}}
\]
that the restricted observable $\boldsymbol{\varphi}$ and its monomials
$\boldsymbol{\varphi}^{\mathbf{k}}$ must satisfy. This last equation
can be rewritten as a $d$-dimensional autonomous system of linear
system of ODEs,
\begin{equation}
\left[\begin{array}{cc}
\mathbf{I}_{d\times d} & \mathbf{Q}_{2}\end{array}\right]\frac{d}{dt}\left[\begin{array}{c}
\boldsymbol{\varphi}\\
\mathbf{K}_{\geq2}\left(\boldsymbol{\varphi}\right)
\end{array}\right]=\left[\begin{array}{cc}
\mathbf{B}\,\, & \mathbf{B}\mathbf{Q}_{2}\end{array}\right]\left[\begin{array}{c}
\boldsymbol{\varphi}\\
\boldsymbol{K}_{\geq2}\left(\boldsymbol{\varphi}\right)
\end{array}\right],\label{eq:d-dime linear eq. for linearization coefficients}
\end{equation}
for the reduced observable $\boldsymbol{\varphi}$ and the infinite-dimensional
vector $\mathbf{K}_{\geq2}\left(\varphi\right)$ of all nonlinear
monomials of $\boldsymbol{\varphi}$. Here $\mathbf{I}_{d\times d}$
denotes the $d$-dimensional identity matrix and $\mathbf{Q}_{2}$
contains all coefficients $\mathbf{q}_{\mathbf{k}}$ as column vectors
starting from order $\left|\mathbf{k}\right|=2$. 

If we truncate the infinite-dimensional vector of monomials $\mathbf{K}_{\geq2}\left(\boldsymbol{\varphi}\right)$
to the vector $\mathbf{K}_{2}^{k}\left(\boldsymbol{\varphi}\right)$
of nonlinear monomials up to order $k$, then eq. (\ref{eq:d-dime linear eq. for linearization coefficients})
becomes 
\begin{equation}
\left[\begin{array}{cc}
\mathbf{I}_{d\times d} & \mathbf{Q}_{2}^{k}\end{array}\right]\frac{d}{dt}\left[\begin{array}{c}
\boldsymbol{\varphi}\\
\mathbf{K}_{2}^{k}\left(\boldsymbol{\varphi}\right)
\end{array}\right]=\left[\begin{array}{cc}
\mathbf{B} & \mathbf{B}\mathbf{Q}_{2}^{k}\end{array}\right]\left[\begin{array}{c}
\boldsymbol{\varphi}\\
\mathbf{K}_{2}^{k}\left(\boldsymbol{\varphi}\right)
\end{array}\right].\label{eq:underdetermined problem}
\end{equation}
This is a $d$-dimensional implicit system of linear  ODEs for the
dependent variable vector $\left(\boldsymbol{\varphi},\mathbf{K}_{2}^{k}\left(\boldsymbol{\varphi}\right)\right)$
whose dimension is always larger than $d$. Consequently, the operator
$\left[\begin{array}{cc}
\mathbf{I}_{d\times d} & \mathbf{Q}_{2}^{k}\end{array}\right]$ is never invertible and hence, contrary to the assumption of EDMD,
there is no well-defined linear system of ODEs that governs the evolution
of an observable vector and the monomials of its components. 

The above conclusion remains unchanged even if one attempts to optimize
with respect to the choice of the coefficients $\mathbf{q}_{\mathbf{k}}$
in the matrix $\mathbf{Q}_{2}^{k}$.

\subsection{Implementation and applications of DDL}

\subsubsection{Basic implementation of DDL for model reduction and linearization}

Theorem \ref{thm:DDL} allows us to define a numerical procedure to
construct a linearizing transformation on the $d$-dimensional attracting
slow manifold\emph{ $\mathcal{W}(E)$ }systematically from data. From
eq. \eqref{eq:underdetermined problem}, the matrices $\mathbf{B}$
and $\mathbf{Q}$ are to be determined, given a set of observed trajectories.
In line with the notation used in Section \ref{subsec:DDL-vs.-EDMD},
let the data matrix $\mathbf{K}_{2}^{k}\left(\boldsymbol{\varphi}\right)$
contain monomials (from order $2$ to order $k)$ of the observable
vector $\boldsymbol{\varphi}$ and let $\hat{\mathbf{K}_{2}^{k}}\left(\boldsymbol{\varphi}\right)$
contain denote the evaluation of $\mathbf{K}_{2}^{k}\left(\boldsymbol{\varphi}\right)$
time $\Delta t$ later. Passing to the discrete version of the invariance
equation (\ref{eq:underdetermined problem}), we obtain
\[
\left[\begin{array}{cc}
\mathbf{I} & \mathbf{Q}\end{array}\right]\left[\begin{array}{c}
\hat{\boldsymbol{\varphi}}\\
\hat{\mathbf{K}_{2}^{k}}\left(\boldsymbol{\varphi}\right)
\end{array}\right]=\left[\begin{array}{cc}
\boldsymbol{\mathcal{B}} & \boldsymbol{\mathcal{B}}\mathbf{Q}\end{array}\right]\left[\begin{array}{c}
\boldsymbol{\varphi}\\
\mathbf{K}_{2}^{k}\left(\boldsymbol{\varphi}\right)
\end{array}\right]
\]
for some matrices $\mathbf{Q}\in\mathbb{R}^{d\times N(d,k)-d}$ and
$\boldsymbol{\mathcal{B}}=e^{\mathbf{B}\Delta t}\in\mathbb{R}^{d\times d}$.
Moreover, the inverse transformation of the linearization on the SSM
$\mathcal{W}\left(E\right)$ is well-defined, and hence with an appropriate
matrix $\mathbf{Q}^{inv}\in\mathbb{R}^{d\times N(d,k)-d}$, we can
write 
\[
\left[\begin{array}{cc}
\mathbf{I} & \mathbf{Q}^{inv}\end{array}\right]\left[\begin{array}{c}
\boldsymbol{\varphi}+\mathbf{Q}\mathbf{K}_{2}^{k}\left(\boldsymbol{\varphi}\right)\\
\mathbf{K}_{2}^{k}\left(\boldsymbol{\varphi}+\mathbf{Q}\mathbf{K}_{2}^{k}\left(\boldsymbol{\varphi}\right)\right)
\end{array}\right]=\boldsymbol{\varphi}.
\]

This allows us to define the cost functions 
\begin{align}
\mathcal{L}^{(1)}(\mathbf{Q},\boldsymbol{\mathcal{B}}) & =\left|\left[\begin{array}{cc}
\mathbf{I} & \mathbf{Q}\end{array}\right]\left[\begin{array}{c}
\hat{\boldsymbol{\varphi}}\\
\hat{\mathbf{K}_{2}^{k}}\left(\boldsymbol{\varphi}\right)
\end{array}\right]-\left[\begin{array}{cc}
\boldsymbol{\mathcal{B}} & \boldsymbol{\mathcal{B}}\mathbf{Q}\end{array}\right]\left[\begin{array}{c}
\boldsymbol{\varphi}\\
\mathbf{K}_{2}^{k}\left(\boldsymbol{\varphi}\right)
\end{array}\right]\right|^{2},\label{eq:ddl_costfunction_separate}\\
\mathcal{L}^{(2)}\left(\mathbf{Q},\mathbf{Q}^{inv}\right) & =\left|\mathbf{Q}\mathbf{K}_{2}^{k}\left(\boldsymbol{\varphi}\right)+\mathbf{Q}^{inv}\mathbf{K}_{2}^{k}\left(\boldsymbol{\varphi}+\mathbf{Q}\mathbf{K}_{2}^{k}\left(\boldsymbol{\varphi}\right)\right)\right|^{2},\nonumber 
\end{align}
where $\mathcal{L}^{(1)}$ measures the invariance error along the
observed trajectories and $\mathcal{L}^{(2)}$ measures the error
due to the computation of the inverse. We aim to jointly minimize
$\mathcal{L}^{(1)}$ and $\mathcal{L}^{(2)}$. To this end, we define
the combined cost function
\begin{equation}
\mathcal{L}_{\nu}\left(\mathbf{Q},\mathbf{Q}^{inv},\boldsymbol{\mathcal{B}}\right)=\mathcal{L}^{(1)}(\mathbf{Q},\boldsymbol{\mathcal{B}})+\nu\mathcal{L}^{(2)}\left(\mathbf{Q},\mathbf{Q}^{inv}\right),\label{eq:ddl_costfunction}
\end{equation}
for some $\nu\geq0$. In our examples, we choose $\nu=1$, which puts
the same weight on both terms in the cost function \eqref{eq:ddl_costfunction}.
Minimizers of $\mathcal{L}_{\nu}$ provide optimal solutions to the
DDL principle and can be written as
\begin{equation}
(\mathbf{Q}^{\star},\mathbf{Q}^{inv,\star}\boldsymbol{\mathcal{B}}^{\star})=\underset{\mathbf{Q},\mathbf{Q}^{inv},\boldsymbol{\mathcal{B}}}{\mathrm{argmin}\,\,}\mathcal{L_{\nu}}(\mathbf{Q},\mathbf{Q}^{inv},\boldsymbol{\mathcal{B}}),\label{eq:ddl_argmin}
\end{equation}
or, equivalently, as solutions of the system of equations
\begin{align}
\frac{\partial\mathcal{L}_{\nu}}{\partial Q_{ij}} & =0\quad i=1,...,d,j=1,...,N(d,k)-d,\label{eq:ddl_equations}\\
\frac{\partial\mathcal{L}_{\nu}}{\partial Q_{ij}^{inv}} & =0,\quad i=1,...,d,j=1,...,N(d,k)-d,\\
\frac{\partial\mathcal{L}_{\nu}}{\partial\mathcal{B}_{ij}} & =0\quad i,j=1,...,d.\nonumber 
\end{align}

The optimal solution \eqref{eq:ddl_argmin} does not necessarily coincide
with the Taylor-coefficients of the linearizing transformation \eqref{eq:DDL transformation convergent}.
Instead of giving the best local approximation, $(\mathbf{Q}^{\star},\mathbf{Q}^{inv,\star}\boldsymbol{\mathcal{B}}^{\star})$
approximates the linearizing transformation and the linear dynamics
in a least-squares sense over the domain of the training data. This
means that DDL is not hindered by the convergence properties of the
analytic linearization. Note that for $d=1$, one can estimate the
radius of convergence of \eqref{eq:DDL transformation convergent},
for example, by constructing the Domb-Skyes plot (see \citet{domb_1957}),
or by finding the radius of the circle in the complex plane onto which
the roots of the truncated expansion accumulate under increasing orders
of truncation (see \citet{jentzsch_1916,ponsioen2018}). For $d>1$,
such analysis is more difficult, since multivariate Taylor-series
have more complicated domains of convergence. In our numerical examples,
we estimate the domain of convergence of such analytic linearizations
as the domain on which $\boldsymbol{\kappa}\circ\boldsymbol{\kappa}^{-1}=\mathbf{I}$
holds to a good approximation. As we will see, this domain of convergence
may be substantially smaller the domain of validity of transformations
determined in a fully data-driven way. 

Since the cost function \eqref{eq:ddl_costfunction_separate} is not
convex, the optimization problem \eqref{eq:ddl_argmin} has to be
solved iteratively starting from an initial guess $(\mathbf{Q}_{0},\mathbf{Q}_{0}^{inv},\boldsymbol{\mathcal{B}}_{0})$.
For the examples presented in the paper, we use the Levenberg--Marquardt
algorithm (see \citet{bates88}), but other nonlinear optimization
methods, such as gradient descent or Adam (see \citet{adam_2015})
could also be used. For our implementation, which is available from
the repository \citet{kaszas_haller_24_code}, we used the Scipy and
Pytorch libraries of \citet{scipy2020,pytorch_2017}. In summary,
we will use the following Algorithm \ref{algorithm} in our examples for model reduction
via DDL. 

\begin{algorithm}[H]
 \KwData{ \\ \quad $d$: Dimension of the slow SSM $\mathcal{W(E)}$. \\ \quad $k$: Maximal polynomial order. \\ \quad $\boldsymbol{\varphi},\hat{\boldsymbol{\varphi}} \in \mathbb{R}^{d\times n_{samples}}$: Data matrices of the reduced coordinates and their forward-shifted images.  The input data may need to be truncated, in order to ensure that it lies sufficiently close to the SSM. This needs to be checked, e.g., by time-frequency analysis (see \citet{cenedese22a}). \\
 \quad $\nu\geq 0$: Weight parameter in the cost function \\
 \quad $tol$: Tolerance value for the cost function.}
 \KwResult{\\ \quad $\mathbf{Q}\in\mathbb{R}^{d\times N(d,k)-d}$: Coefficients of the transformation $\boldsymbol{\varphi} = \boldsymbol{\kappa}(\boldsymbol{\gamma})$. \\ \quad $\mathbf{Q}^{inv}\in\mathbb{R}^{d\times N(d,k)-d}$: Coefficients of the inverse transformation $\boldsymbol{\gamma} = \boldsymbol{\kappa}^{-1}(\boldsymbol{\varphi})$. \\ \quad $\boldsymbol{\mathcal{B}}$: The linear dynamics.}
\vspace{5pt}
 Compute the monomials $\mathbf{K}_{2}^{k}\left(\boldsymbol{\varphi}\right)$, $\hat{\mathbf{K}_{2}^{k}}\left(\boldsymbol{\varphi}\right)\in \mathbb{R}^{N(d,k)-d\times n_{samples}}$
 
 Choose an initial guess, either randomly or inferred from DMD \\
 $(\mathbf{Q},\mathbf{Q}^{inv},\boldsymbol{\mathcal{B}})\gets (\mathbf{Q}_{0},\mathbf{Q}_{0}^{inv},\boldsymbol{\mathcal{B}}_{0})$
 
 \While{$\mathcal{L}_\nu(\mathbf{Q},\mathbf{Q}^{inv},\boldsymbol{\mathcal{B}})>tol$}{
  Compute the monomials $\mathbf{K}_{2}^{k}\left(\boldsymbol{\varphi}+\mathbf{Q}\mathbf{K}_{2}^{k}\left(\boldsymbol{\varphi}\right)\right) $

  Evaluate $\mathcal{L}_\nu\left(\mathbf{Q},\mathbf{Q}^{inv},\boldsymbol{\mathcal{B}}\right)=\mathcal{L}^{(1)}(\mathbf{Q}, \boldsymbol{\mathcal{B}})+\nu \mathcal{L}^{(2)}\left(\mathbf{Q},\mathbf{Q}^{inv}\right)$

  Perform {\em Optimization step} (e.g., Levenberg–Marquardt or gradient descent)\\
    $(\mathbf{Q},\mathbf{Q}^{inv},\boldsymbol{\mathcal{B}})\gets (\mathbf{Q}_{new},\mathbf{Q}_{new}^{inv},\boldsymbol{\mathcal{B}}_{new})$
 }
 \caption{Model reduction with DDL \label{algorithm}}
\end{algorithm}

\begin{rem}
\label{rem:loss_function} The expressions \eqref{eq:ddl_costfunction_separate}-\eqref{eq:ddl_costfunction}
define one of the possible choices for the cost function. With $\nu=0$,
\eqref{eq:ddl_costfunction} simply corresponds to a one-step-ahead
prediction with the linearized dynamics. Alternatively, a multi-step
prediction can also be enforced. For a training trajectory $\boldsymbol{\varphi}(t),$
the invariance 
\[
\left[\begin{array}{cc}
\mathbf{I} & \mathbf{Q}\end{array}\right]\left[\begin{array}{c}
\boldsymbol{\varphi}\\
\mathbf{K}_{2}^{k}\left(\boldsymbol{\varphi}\right)
\end{array}\right]=\left[\begin{array}{cc}
\boldsymbol{\mathcal{B}}^{1:m} & \boldsymbol{\mathcal{B}}^{1:m}\mathbf{Q}\end{array}\right]\left[\begin{array}{c}
\boldsymbol{\varphi}(\mathbf{0})\\
\mathbf{K}_{2}^{k}\left(\boldsymbol{\varphi}(\mathbf{0})\right)
\end{array}\right]
\]
could be required, where $\boldsymbol{\mathcal{B}}^{1:m}$ is a tensor
composed of powers of the linear map $\boldsymbol{\mathcal{B}}$.
Optimizing over the entire trajectory is, however, more costly than
simply minimizing \eqref{eq:ddl_costfunction_separate}, and we found
no noticeable improvement in accuracy in our numerical examples. 
\end{rem}

\subsubsection{Relationship with DMD implementations}

Note that setting $\mathbf{Q}=\mathbf{Q}^{inv}=\mathbf{0}$ in the
optimization problem (\ref{eq:ddl_argmin}) turns the problem into
DMD. In this case, the usual DMD algorithm surveyed in the Introduction
returns 
\[
\mathcal{\boldsymbol{\mathcal{B}}}_{0}=\underset{\boldsymbol{\mathcal{B}}}{\mathrm{argmin}\,\,}\underset{}{\mathcal{L}_{0}(\mathbf{0},\mathbf{0},\boldsymbol{\mathcal{B}})},
\]
which is a good initial guess for the non-convex optimization problem
\eqref{eq:ddl_argmin}. More importantly, since Theorem \ref{thm:DDL}
guarantees the existence of a near-identity linearizing transformation,
we expect that the true minimizer is close to the DMD-solution. Therefore,
we may explicitly expand the cost function \eqref{eq:ddl_costfunction_separate}
around the DMD solution as 
\begin{align*}
\mathcal{L_{\nu}}(\mathbf{Q},\mathbf{Q}^{inv},\mathbf{\mathcal{B}}) & =\mathcal{L}_{\nu}(\mathbf{0},\mathbf{0},\boldsymbol{\mathcal{B}}_{0})+D\mathcal{L}_{(\mathbf{0},\mathbf{0},\mathcal{\boldsymbol{\mathcal{B}}}_{0})}\cdot\left(\begin{array}{c}
\mathbf{Q}\\
\mathbf{Q}^{inv}\\
\mathcal{\boldsymbol{\mathcal{B}}}-\boldsymbol{\mathcal{B}}_{0}
\end{array}\right)\\
 & +\frac{1}{2}\left[D^{2}\mathcal{L}_{(\mathbf{0},\mathbf{0},\mathcal{\boldsymbol{\mathcal{B}}}_{0})}\cdot\left(\begin{array}{c}
\mathbf{Q}\\
\mathbf{Q}^{inv}\\
\boldsymbol{\mathcal{B}}-\boldsymbol{\mathcal{B}}_{0}
\end{array}\right)\right]\cdot\left(\begin{array}{c}
\mathbf{Q}\\
\mathbf{Q}^{inv}\\
\boldsymbol{\mathcal{B}}-\boldsymbol{\mathcal{B}}_{0}
\end{array}\right)\\
 & +\mathcal{}\left(\left|\mathbf{Q}\right|^{3},\left|\mathbf{Q}^{inv}\right|^{3},\left|\boldsymbol{\mathcal{B}}-\boldsymbol{\mathcal{B}}_{0}\right|^{3}\right),
\end{align*}
where $D\mathcal{L}_{(\mathbf{0},\mathbf{0},\boldsymbol{\mathcal{B}}_{0})}$
and $D^{2}\mathcal{L}_{(\mathbf{0},\mathbf{0},\boldsymbol{\mathcal{B}}_{0})}$
are the Jacobian and the Hessian of the cost function evaluated at
the DMD solution, respectively. Since the Jacobian is nonsingular
at the DMD solution, the minimum of the quadratic approximation of
the cost function satisfies the linear equation
\begin{equation}
-D\mathcal{L}_{(\mathbf{0},\mathbf{0},\mathcal{\boldsymbol{\mathcal{B}}}_{0})}=D^{2}\mathcal{L}_{(\mathbf{0},\mathbf{0},\boldsymbol{\mathcal{B}}_{0})}\left(\begin{array}{c}
\mathbf{Q}\\
\mathbf{Q}^{inv}\\
\boldsymbol{\mathcal{B}}-\boldsymbol{\mathcal{B}}_{0}
\end{array}\right).\label{eq:approximate_ddl_sol}
\end{equation}

This serves as the first-order correction to the DMD-solution in the
DDL procedure. The equation \eqref{eq:approximate_ddl_sol} is explicitly
solvable and is equivalent to performing a single Levenberg--Marquardt
step on the non-convex cost function \eqref{eq:ddl_costfunction_separate},
with the DMD solution $(\mathbf{0},\mathbf{0},\mathcal{\boldsymbol{\mathcal{B}}}_{0})$
serving as an initial guess.

Minimization of \eqref{eq:ddl_costfunction_separate} leads to a non-convex
optimization problem. Besides computing the leading-order approximation
\eqref{eq:approximate_ddl_sol}, a possible workaround to this challenge
is to carry out the linearization in two steps. First, one can fit
a polynomial map to the reduced dynamics by linear regression. Then,
if the reduced-dynamics is non-resonant, it can be analytically linearized.
\citet{axas22} follow this approach to automatically find the extended
normal form style reduced dynamics on SSMs using the implementation
of SSMTool by \citet{jain23}. Although this procedure does convert
the DDL principle into a convex problem, the drawback is that the
linearization is obtained as a Taylor-expansion, with possibly limited
convergence properties.

\subsubsection{Using DDL to construct spectral foliations\label{subsec:foliation}}

The mathematical foundation of SSM-reduced modeling is that any trajectory
converging to a slow SSM is guaranteed to synchronize up to an exponentially
decaying error with one of the trajectories on the SSM. This follows
from the general theory of invariant foliations by \citet{fenichel71},
when applied to the $d$-dimensional normally hyperbolic invariant
manifold $\mathcal{W}(E)$.\footnote{More precisely, Fenichel's foliation results become applicable after
the wormhole construct in Proposition B1 of \citet{eldering18} is
applied to extend $\mathcal{W}(E)$ smoothly into a compact normally
attracting invariant manifold without boundary. This is needed because
Fenichel's results only apply to compact normally attracting invariant
manifolds with an empty or overflowing boundary, whereas the boundary
of $\mathcal{W}(E)$ is originally inflowing.} The main result of the theory is that off-SSM initial conditions
synchronizing with the same on-SSM trajectory turn out to form a class
$C^{r-1}$ smooth, $\left(n-d\right)$-dimensional manifold, denoted
$\mathcal{F}_{\mathbf{p}}$, which intersects $\mathcal{W}(E)$ in
a unique point $\mathbf{p}\in\mathcal{W}(E)$. The manifold $\mathcal{F}_{\mathbf{p}}$
is called the stable fiber emanating from the base point $\mathbf{p}$.
Fenichel proves that any off-SSM trajectory $\mathbf{x}(t;\mathbf{x}_{0})$
with initial condition $\mathbf{x}_{0}\in\text{\ensuremath{\mathcal{F}_{\mathbf{p}_{0}}} }$
converges to the specific on-SSM trajectory $\mathbf{p}(t;\mathbf{p}_{0})\in\mathcal{W}(E)$
with initial condition $\mathbf{p}_{0}\in\mathcal{W}(E)$ faster than
any other nearby trajectory might converge to $\mathbf{p}(t;\mathbf{p}_{0}).$
Recently, \citet{szalai20} studied this foliation in more detail
under the name ``invariant spectral foliation'', discussed its uniqueness
in an appropriate smoothness class and proposed its use in model reduction.

To predict the evolution of a specific, off-SSM initial condition
$\mathbf{x}_{0}$ up to time $t$ from an SSM-based model, we first
need to relate that initial condition to the base point $\mathbf{p}_{0}$
of the stable fiber $\mathcal{F}_{\mathbf{p}_{0}}$. Next, we need
to run the SSM-based reduced model up to time $t$ to obtain $\mathbf{p}(t;\mathbf{p}_{0})$.
Based on the exponentially fast convergence of the full solution $\mathbf{x}(t;\mathbf{x}_{0})$
to the SSM-reduced solution $\mathbf{p}(t;\mathbf{p}_{0}),$ we obtain
an accurate longer-term prediction for $\mathbf{x}(t;\mathbf{x}_{0})$
using this procedure. Such a longer-term prediction is helpful, for
instance, when we wish to predict steady states, such as fixed points
and limit cycles, from the SSM-reduced dynamics. 

Constructing this spectral foliation directly from data, however,
is challenging for nonlinear systems. Indeed, one would need a very
large number of initial conditions that cover uniformly a whole open
neighborhood of the fixed point in the phase space. For example, while
one or two training trajectories are generally sufficient to infer
accurate SSM-reduced models even for very high-dimensional systems
(see e.g., \citet{cenedese22a,cenedese22b,axas22}), thousands of
uniformly distributed initial conditions in a whole open set of a
fixed point are required to infer accurate spectral foliation-based
models even for low-dimensional systems (see \citet{szalai20}). The
latter number and distribution of initial conditions is unrealistic
to acquire in a truly data-driven setting. 

To avoid constructing the full foliation, one may simply project an
initial condition $x_{0}$ orthogonally to an observed spectral submanifold
$\mathcal{W}(E)$ to obtain $\mathbf{p}_{0}$, but this may result
in large errors if $E$ and $F$ are not orthogonal. In that case,
$\mathcal{W}(E)$ may divert substantially from $E$ (see \citep{ROWLEY_2005,Otto_2022,Otto_2023_covariance}
for a discussion of the limitations of this projection for general
invariant manifolds). 

A better solution is to project $\mathbf{x}_{0}$ orthogonally to
the slow spectral subspace $E$ over which $\mathcal{W}(E)$ is a
graph in an (often large) neighborhood of the fixed point. This approach
assumes that $E$ and $F$ are nearly normal and $\mathcal{W}(E)$
is nearly flat. As the latter is typically the case for delay-embedded
observables (\citet{axas23}), orthogonal projection onto $E$ has
been the choice so far in data-driven SSM-based reduction via the
SSMLearn algorithm (\citet{cenedese21}). This approach has produced
highly accurate reduced-order models in a number of examples (see
\citet{cenedese22a,cenedese22b,axas22}). There are nevertheless examples
in which the linear part of the dynamical system is significantly
non-normal and hence $E$ and $F$ are not close to being orthogonal
(see \citet{Bettini_2024}).

Near hyperbolic fixed points, the use of DDL eliminates the need to
construct involved nonlinear spectral foliations. Indeed, let us assume
that the slow spectral subspace $E$ in Theorem \ref{thm:DDL} can
be decomposed into a direct sum $E=E_{1}\oplus E_{2}$, where $E_{1}$
denotes the slowest spectral subspace with dim $E_{1}=d_{1}$ and
$E_{2}$ denotes the second-slowest spectral subspace with dim $E_{2}=d_{2}$,
as sketched in Fig. \ref{fig:fibers_geometry}. 
\begin{figure}
\centering{}\includegraphics[width=0.8\textwidth]{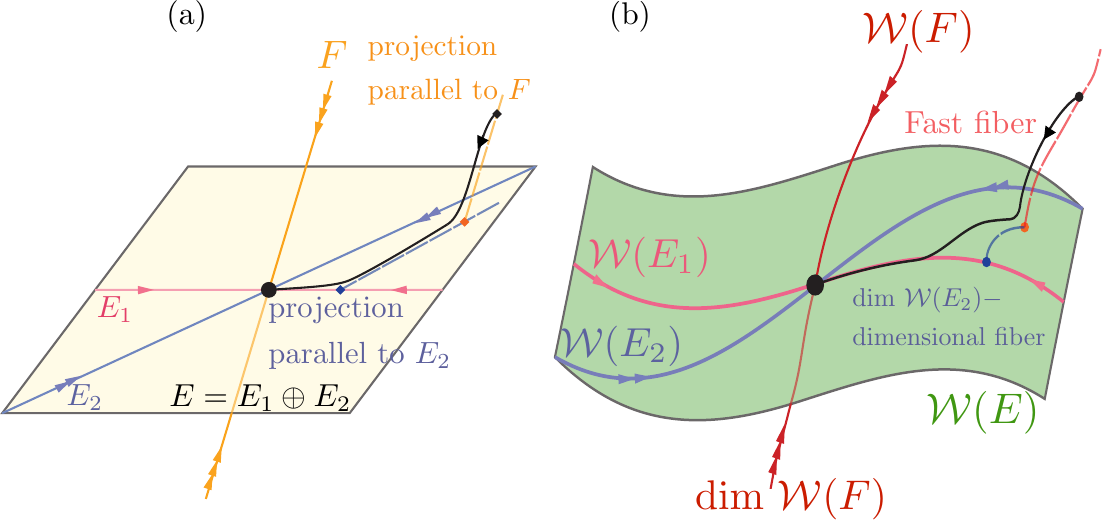}\caption{(a): The linearized phase space geometry governed by the slow spectral
subspace $E=E_{1}\oplus E_{2}$ and the slow invariant foliation within
$E$. (b): Phase space geometry in the original coordinates. \label{fig:fibers_geometry}}
\end{figure}
Reducing the dynamics to the SSM $\mathcal{W}(E)$ is accurate for
transient times given by the decay rate of $E_{2}.$ This initial
reduction can be done simply by a normal projection onto $E$. Inside
$E$, one can simply locate spectral foliations of the DDL-linearized
systems explicitly and map them back to the original nonlinear system
under the DDL transformation $\boldsymbol{\kappa}\left(\boldsymbol{\gamma}\right)$.
The unique class $C^{a}$ foliation of a linear system within $E$
is the family of stable fibers forming the affine space
\[
\mathcal{F}_{\mathbf{p}}=\mathbf{p}+E_{2},
\]
where $\mathbf{p}\in E_{1}$. The trajectories started inside $\mathcal{F}_{\mathbf{p}}$
all synchronize with $\mathbf{p}\in E_{1}$. The linear projection
$\mathbf{P}_{E_{2}}$ onto $E_{1}$ along directions parallel to $E_{2}$,
when applied to an initial condition $\mathbf{y}\in\mathcal{F}_{\mathbf{p}}$,
returns the base point 
\begin{equation}
\mathbf{P}_{E_{2}}\mathbf{y}=\mathbf{p}.\label{eq:linear projection}
\end{equation}

In the nonlinear system \eqref{eq:nonlinear system0}, the leaves
of the smooth foliation within $\mathcal{W}(E)$ 
\[
\mathcal{F}_{\boldsymbol{\kappa}\left(\mathbf{p}\right)}^{0}=\boldsymbol{\kappa}\left(\mathcal{F}_{\mathbf{p}}\right)\subset\mathcal{W}(E),
\]
where $\boldsymbol{\kappa}\left(\mathbf{p}\right)\in\mathcal{W}(E_{1})$
is the image of $\mathbf{p}$ under the mapping $\boldsymbol{\kappa}$
defined in (\ref{eq:linearizing transformation on W(E)}). The SSM
$\mathcal{W}(E)$ can then be parametrized via the foliation
\[
\mathcal{W}(E)=\bigcup_{\mathbf{q}\in\mathcal{W}(E_{1})}\mathcal{F}_{\mathbf{q}}^{0}.
\]

\subsubsection{Using DDL to predict nonlinear forced response from unforced data\label{subsec:predicting forced response from DDL}}

We now discuss how DDL performed near the fixed point of an autonomous
dynamical system can be used to predict nonlinear forced response
under additional weak periodic forcing in the domain of DDL. The addition
of such small forcing is frequent in structural vibration problems
in which the unforced structure (e.g., a beam or disk) is rigid enough
to react with small displacements under practically relevant excitation
levels (see, e.g., \citet{cenedese22a,cenedese22b} for specific examples).

We append system \eqref{eq:more specific nonlinear system} with a
small, time-periodic forcing term $\varepsilon\mathbf{F}(\mathbf{x},t)$
to obtain the system
\begin{equation}
\dot{\mathbf{x}}=\mathbf{A}\mathbf{x}+\tilde{\mathbf{f}}(\mathbf{x})+\varepsilon\mathbf{F}(\mathbf{x},t),\qquad\mathbf{x}\in\mathcal{\mathbb{R}}^{n},\qquad\mathbf{A}=D\mathbf{f}(0),\qquad\tilde{\mathbf{f}}(\mathbf{x})=\mathcal{O}\left(\left|\mathbf{x}\right|^{2}\right),\qquad0\leq\epsilon\ll1,\label{eq:forced system}
\end{equation}
with $\mathbf{F}(\mathbf{x},t)=\mathbf{F}(\mathbf{x},t+T)$ for some
period $T>0$. If the conditions of Theorem \ref{thm:DDL} hold for
the system \eqref{eq:forced system} for $\varepsilon=0$\emph{, }then,
for $\epsilon>0$ small enough, exists a unique $d$-dimensional,
$T$-periodic, attracting spectral submanifold $\mathcal{W}_{\varepsilon}(E,t)\in C^{r}$
of a locally unique attracting $T$-periodic orbit $\mathbf{x}_{\epsilon}(t)$
perturbing from $\mathbf{x}=\mathbf{0}$ (see, e.g., \citet{cabre03,haller16}).
The manifold $\mathcal{W}_{\varepsilon}(E,t)$ is $O(\varepsilon)$
$C^{1}$-close to $\mathcal{W}_{0}(E,t)\equiv\mathcal{W}(E)$ and
hence its reduced dynamics can be parametrized using the reduced observable
vector $\boldsymbol{\varphi}=\boldsymbol{\phi}\vert_{\mathcal{W}(E)}$
in the form \emph{
\begin{align}
\dot{\boldsymbol{\varphi}} & =\mathbf{B}\boldsymbol{\varphi}+\mathbf{q}\left(\boldsymbol{\varphi}\right)+\varepsilon\hat{\mathbf{F}}(\boldsymbol{\varphi},t),\qquad\mathbf{B}=D\boldsymbol{\phi}(\mathbf{0})\mathbf{T}_{E}\boldsymbol{\Lambda}_{E}\left(D\boldsymbol{\phi}(\mathbf{0})\mathbf{T}_{E}\right)^{-1},\quad\mathbf{q}\left(\boldsymbol{\varphi}\right)=\mathcal{O}\left(\left|\boldsymbol{\varphi}\right|^{2}\right),\label{eq:gamma ODE-1}\\
 & \hat{\mathbf{F}}(\boldsymbol{\varphi},t)=\left(D\boldsymbol{\phi}(\mathbf{0})\mathbf{T}_{E}\right)^{-1}\left(\mathbf{I}+D\mathbf{h}\left(\boldsymbol{\varphi},\mathbf{0}\right)\right){}^{-1}\mathbf{F}(\mathbf{0},t)+\mathcal{O}\left(\varepsilon\left|\boldsymbol{\varphi}\right|^{2}\right),\nonumber 
\end{align}
}where we have relegated the details of this calculation to Appendix
\ref{sec:Proof-of-Theorem-DDL-forced}.\emph{ }

Then the unique, $C^{r}$ change of coordinates, 
\begin{equation}
\boldsymbol{\varphi}=\boldsymbol{\kappa}\left(\boldsymbol{\gamma}\right)=\boldsymbol{\gamma}+\boldsymbol{\ell}(\boldsymbol{\gamma}),\label{eq:linearizing transformation on W(E)-1}
\end{equation}
guaranteed by statement (iii) of Theorem \ref{thm:DDL} transforms
the reduced dynamics (\ref{eq:gamma ODE-1}) to its final form
\begin{equation}
\dot{\boldsymbol{\gamma}}=\mathbf{B}\boldsymbol{\gamma}+\varepsilon(\mathbf{I}+D\boldsymbol{\ell}(\boldsymbol{\gamma}))^{-1}\hat{\mathbf{F}}(\mathbf{0},t).\label{eq:linearized flow on SSM-1}
\end{equation}
The transformation is valid on trajectories of \eqref{eq:forced system}
as long as they remain in the domain of definition of the coordinate
change (\ref{eq:linearizing transformation on W(E)-1}).

Note that eq. (\ref{eq:linearized flow on SSM-1}) is a weakly perturbed,
time-periodic nonlinear system. The matrix $\mathbf{B}$ and the nonlinear
terms $\boldsymbol{\ell}(\boldsymbol{\gamma})$ can be determined
using data from the unforced ($\varepsilon=0$) system. As a result,
nonlinear time-periodic forced response can be predicted \emph{solely
from unforced data} by applying numerical continuation to system (\ref{eq:linearized flow on SSM-1})
for $\varepsilon>0$. This is not expected to be as accurate as SSM-based
forced response prediction (see, e.g., \citet{cenedese22a,cenedese22b,axas22,axas23}),
but nevertheless offers a way to make predictions for non-linearizable
forced response based solely on DDL performed on unforced data. These
predictions are valid for forced trajectories that stay in the domain
of convergence of DDL carried out on the unforced system. We will
illustrate such predictions using actual experimental data from fluid
sloshing in Section \ref{subsec:Sloshing}.

Setting $\boldsymbol{\ell}(\boldsymbol{\gamma})=0$ in formula (\ref{eq:linearized flow on SSM-1})
enables us to carry out a forced-response prediction based on DMD
as well. Such a prediction will be fundamentally linear with respect
to the forcing and can only be reasonably accurate for very small
forcing amplitudes, as we will indeed see in examples. There is no
systematic way to model the addition of non-autonomous forcing in
the EDMD procedure, and hence EDMD will not be included in our forced
response prediction comparisons.

We also note, that one might be tempted to solve an approximate version
of \eqref{eq:linearized flow on SSM-1} by assuming
\begin{equation}
\varepsilon(\mathbf{I}+D\boldsymbol{\ell}(\boldsymbol{\gamma}))^{-1}\approx\varepsilon\mathbf{I}.\label{eq:approximate ddl assumption}
\end{equation}
This assumption simplifies the computation of the forced response
of the nonlinear system \eqref{eq:linearized flow on SSM-1} to those
of a simple linear system. Although the forced response computed using
this approximate DDL method turns out to be more accurate than DMD
on our example, we do not recommend this approach. This is because
neglecting the nonlinear effects of the coordinate change in \eqref{eq:linearized flow on SSM-1}
is, in general, inconsistent with $\boldsymbol{\ell}(\boldsymbol{\gamma})\neq0$.
We give more detail on this approximation in Appendix \ref{sec:approximate ddl}
of the Supplementary Information.

\section{Examples}

In this section, we compare the DMD, EDMD and DDL algorithms on specific
examples. When applicable, we also compute the exact analytic linearization
of the dynamical system near its fixed point as a benchmark. On a
slow SSM $\mathcal{W}\left(E\right)$, an observer trajectory $\boldsymbol{\varphi}(t)$,
starting from a select initial condition $\boldsymbol{\varphi}(0)$,
will be tracked as the image of the linearized reduced observer trajectory
$\boldsymbol{\gamma}(t)$ under the linearizing transformation (\ref{eq:linearizing transformation on W(E)-1}):
\begin{equation}
\boldsymbol{\varphi}(t)=\boldsymbol{\kappa}\left(e^{\mathbf{B}t}\boldsymbol{\gamma}(0)\right)=e^{\mathbf{B}t}\boldsymbol{\gamma}(0)+\boldsymbol{\ell}\left(e^{\mathbf{B}t}\boldsymbol{\gamma}(0)\right),\qquad\boldsymbol{\gamma}(0)=\boldsymbol{\kappa}^{-1}\left(\boldsymbol{\varphi}(0)\right).\label{eq:setting for DMD examples}
\end{equation}
When model reduction has also taken place, i.e., when the observable
vector $\boldsymbol{\varphi}$ is not defined on the full phase space,
we will nevertheless provide a prediction in the full phase space
via the parametrization of the slow SSM. 

By Theorem \ref{thm: DMD}, DMD can be interpreted as setting $\boldsymbol{\ell}(\boldsymbol{\gamma})\equiv0$
in (\ref{eq:setting for DMD examples}) and finding the linear operator
$\mathbf{B}$ as a best fit from the available data. In contrast,
DDL finds the linear operator $\mathbf{B}$, the transformation $\boldsymbol{\varphi}=\boldsymbol{\gamma}+\boldsymbol{\ell}(\boldsymbol{\gamma})$,
and its inverse simultaneously. As we explained in Section \ref{subsec:DDL-vs.-EDMD},
EDMD cannot quite be interpreted in terms of the linearizing transformation
(\ref{eq:setting for DMD examples}) as it is an attempt to immerse
the dynamics into a higher dimensional space. For our EDMD tests,
we will use monomials of the observable vector $\boldsymbol{\varphi}$. 

\subsection{1D nonlinear system with two isolated fixed points}

Consider the one-dimensional ODE obtained as the radial component
of the Stuart-Landau equation, i.e., 
\[
\dot{r}=\mu r-r^{3},
\]
which can be rescaled to 
\begin{equation}
\dot{R}=R-R^{3}.\label{eq:1deq}
\end{equation}

For $R\geq0$, the system has a repelling fixed point at $R=0$ and
an attracting one at $R=1$. \citet{page19} show that local expansion
of observables in terms of the Koopman eigenfunctions computed near
each fixed point are possible, but the expansions at the two fixed
points are not compatible with each other and both diverge at $R=\sqrt{2}/2$.
This is a consequence of the more general result that the Koopman
eigenfunctions themselves typically blow up near basin boundaries
(see our Proposition \ref{prop:Koopman blow-up} in Appendix A of
the Supplementary Information). Both DMD and EDMD can nevertheless
be computed from data, even for a trajectory crossing the turning
point at $R=\sqrt{2}/2$, but the resulting models cannot have any
connection to the Koopman operator. 

In each comparison performed on system (\ref{eq:1deq}), we generate
a single trajectory in the domain of attraction of the $R=1$ fixed
point and use it as training data for DMD, EDMD and DDL. In each subplot
of Fig. \ref{fig:1d_example}, the single training trajectory starts
from the intersection of the red horizontal line ``IC of training
trajectory'' with the  $t=0$ dashed line. We then also generate a new test trajectory (black) with
its initial condition denoted with a black dot over the line $t=0$.
We place this initial condition slightly outside the domain of linearization
for system \eqref{eq:1deq} (under the grey line labeled ``Turning
point''). We use DMD, order-$k=5$ EDMD, and DDL trained on a single
training trajectory to make predictions for the black testing trajectory
(not used in the training).

Figure \ref{fig:1d_example}a shows DDL to be the most accurate of
the three methods when applied to forward-time ($t\geq0$) segments
of the test trajectory. If we try to predict the backward-time ($t<0$)
segment of the same trajectory as it leaves the training domain, DDL
diverges immediately upwards, whereas DMD and EDMD diverge more gradually
downwards. As we increase the training domain in Fig \ref{fig:1d_example}b,
DDL continues to be the most accurate in both forward and backward
time until it reaches the domain of its training range in backward
time. At that point, it diverges quickly upwards, while DMD and EDMD
diverge more slowly downwards. 

Importantly, increasing the approximation order for DDL first to $k=10$
then to $k=18$ (see Figs. \ref{fig:1d_example}c-d), makes DDL predictions
more and more accurate in backward time inside the training domain.
At the same time, the same increase in order makes EDMD less and less
accurate inside the same domain. This is not surprising for EDMD because
it seeks to approximate the dynamics within a Koopman-invariant subspaces
for increasing $k$, and Koopman mode expansions blow up at the ``Turning
point line'', as shown both analytically and numerically by \citet{page19}.
Interestingly, however, EDMD becomes less accurate even within the
domain of linearization under increasing $k$. This is clearly visible
in Fig. \ref{fig:1d_example}d which shows spurious, growing oscillations
in the EDMD predictions close to the $R=1$ fixed point.

In summary, of the three methods tested, DDL makes the most accurate
predictions in forward time. This remains true in backward time as
longs as the trajectory remains in the training range used for the
three methods, even if this range is larger than the theoretical domain
of linearization. Inside the training range, an increase of the order
$k$ of the monomials used increases the accuracy of DDL but introduces
growing errors in EDMD. 

\begin{figure}
\centering{}\includegraphics[width=1\textwidth]{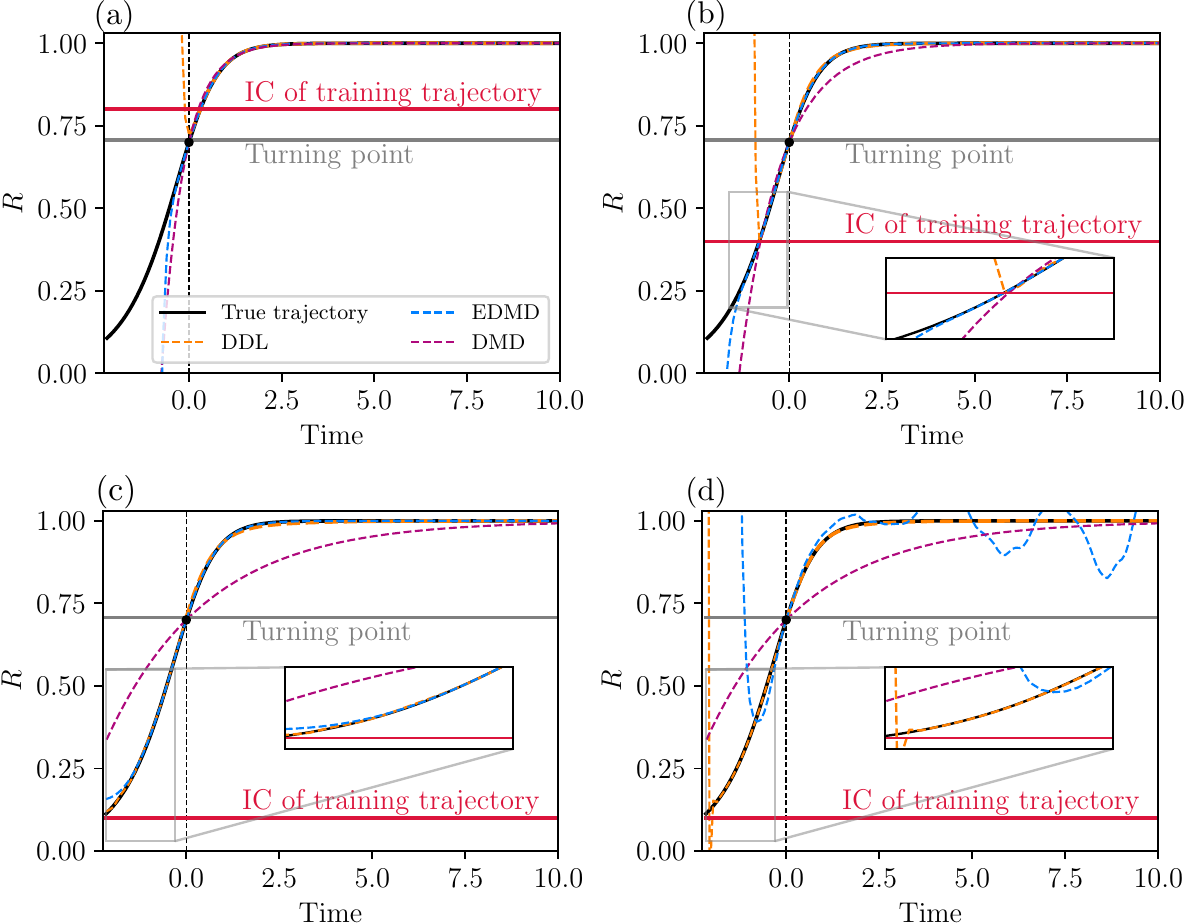}\caption{Predictions of DMD, EDMD and DDL on trajectories of (\ref{eq:1deq}).
(a) Training trajectory starts inside the domain of convergence of
the linearization, i.e. $R(0)=0.8$ (see \citet{page19}). For both
DDL and EDMD the order of the monomials used is $k=5$. (b) Same for
a different training trajectory with $R(0)=0.4$ and $k=5$ (c) Same
for $R(0)=0.1$ and $k=10$ (d) Same for $R(0)=0.1$ and $k=18$.
\label{fig:1d_example}}
\end{figure}

\subsection{3D linear system studied via nonlinear observables}

\citet{wu21} studied the ability of DMD to recover a 3D linear system
based on the time history of three nonlinear observables evaluated
on the trajectories of the system. To define the linear system, they
use a block-diagonal matrix $\boldsymbol{\Lambda}$ and a basis transformation
matrix $\mathbf{R}$ of the form
\begin{equation}
\boldsymbol{\Lambda}=\left(\begin{array}{ccc}
a & -b & 0\\
b & a & 0\\
0 & 0 & c
\end{array}\right),\quad a,b,c\in\mathbb{R},\qquad\mathbf{R}=\left(\begin{array}{ccc}
1 & 0 & \sin\theta_{1}\cos\theta_{2}\\
0 & 1 & \sin\theta_{1}\sin\theta_{2}\\
0 & 0 & \cos\theta_{2}
\end{array}\right),\label{eq:3d_lin_system1}
\end{equation}
to define the linear discrete dynamical system 
\begin{equation}
\mathbf{x}(n+1)=\left(\mathbf{R}\boldsymbol{\Lambda}\mathbf{R}^{-1}\right)\mathbf{x}(n).\label{eq:3D discrete example}
\end{equation}
The linear change of coordinates $\mathbf{R}$ rotates the real eigenspace
of $\boldsymbol{\Lambda}$ corresponding to the eigenvalue $c$ and
hence introduces non-normality in system (\ref{eq:3D discrete example}).
This system is then assumed to be observed via a 3D nonlinear observable
vector
\begin{equation}
\mathbf{y}(\mathbf{x})=\left(\begin{array}{c}
x_{1}+0.1\left(x_{1}^{2}+x_{2}x_{3}\right)\\
x_{2}+0.1\left(x_{2}^{2}+x_{1}x_{3}\right)\\
x_{3}+0.1\left(x_{3}^{2}+x_{1}x_{2}\right)
\end{array}\right).\label{eq:3d_lin_system2}
\end{equation}
Ideally, DMD should closely approximate the linear dynamics of system
(\ref{eq:3D discrete example}) because the observable function defined
in eq. (\ref{eq:3d_lin_system2}) is close to the identity and has
only weak nonlinearities. \citet{wu21} find, however, that this system
poses a challenge for DMD, which produced inaccurate predictions for
the spectrum of $\mathbf{R}\boldsymbol{\Lambda}\mathbf{R}^{-1}$. 

Following one of the parameter settings of \citet{wu21}, we set $a=0.45\sqrt{3}$,
$b=0.5$, $c=0.6$, $\theta_{1}=1.5$, and $\theta_{2}=0$. We initialize
three training trajectories with $\left\Vert \mathbf{x}(0)\right\Vert <1,$
each containing $100$ iterations of system (\ref{eq:3D discrete example}).
We then compute the predictions of a $5^{th}$ order DDL model and
compare to those of DMD and EDMD on a separate test trajectory not
used in training these three methods. The predictions and the spectrum
obtained from the three methods are shown in Fig. \ref{fig:3d_lin_example}. 

\begin{figure}
\centering{}\includegraphics[width=0.9\textwidth]{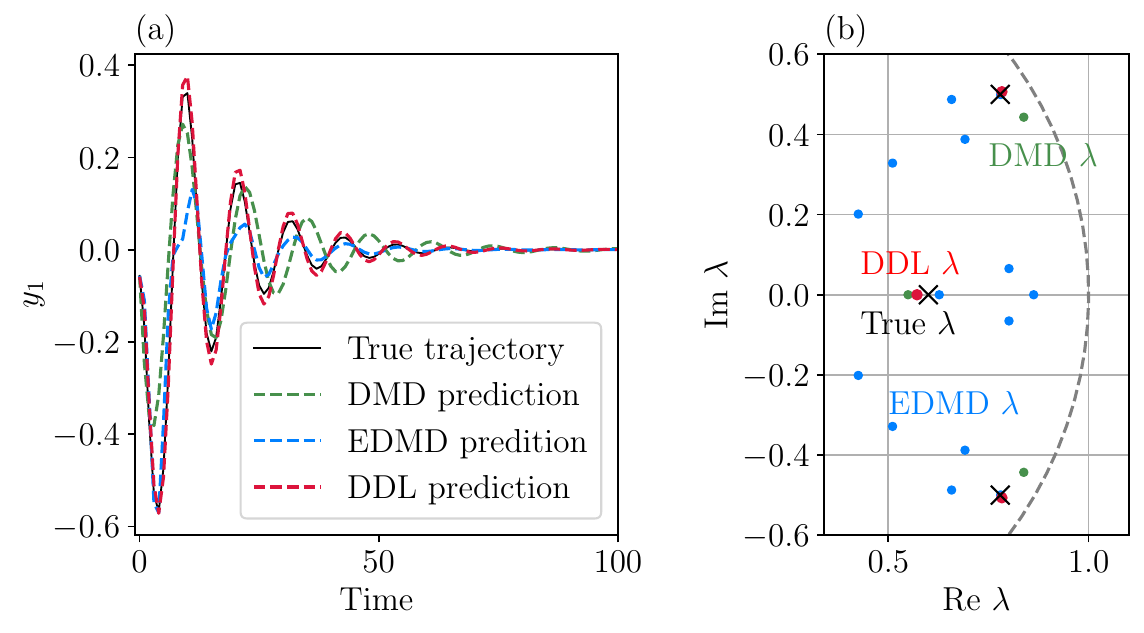}\caption{Predictions by DMD, EDMD, and DDL on the discrete dynamical system
\eqref{eq:3d_lin_system1},\eqref{eq:3d_lin_system2}.(a) Predicted
and true $y_{1}-$components of a test trajectory. (b) Spectra identified
by DMD, EDMD, and DDL superimposed on the true spectrum (marked by
crosses). The dashed line represents the unit circle. \label{fig:3d_lin_example}}
\end{figure}

The predictions of DMD and EDMD can only be considered accurate for
very low amplitude oscillations, while DDL returns accurate predictions
throughout the whole trajectory. This example consists of linear dynamics
and monomial observables of the state, and hence should be an ideal
test case for EDMD. Yet, EDMD is inaccurate in identifying the spectrum
of system (\ref{eq:3D discrete example}). Indeed, as seen in Fig.
\ref{fig:3d_lin_example}b, a number of spurious eigenvalues arise
from EDMD, both real and complex. DMD performs clearly better but it
is still markedly less accurate than DDL. These inaccuracies in the
predictions of EDMD and DMD spectra are also reflected by considerable
errors in their predictions for trajectories, as seen in Fig. \ref{fig:3d_lin_example}a.
In contrast, DDL produces the most accurate prediction for the test
trajectory. 

\subsection{Damped and periodically forced Duffing equation \label{subsec:Duffing}}

We consider the damped and forced Duffing equation

\begin{align}
\dot{x} & =y,\nonumber \\
\dot{y} & =x-x^{3}-dy+\varepsilon\cos\Omega t,\label{eq:duffing_non_transformed}
\end{align}
with damping coefficient $d=0.0141$, forcing frequency $\Omega$
and forcing amplitude $\varepsilon$. We perform a change of coordinates
$\left(x,y\right)\mapsto\boldsymbol{\varphi}=\left(\varphi_{1},\varphi_{2}\right)$
that moves the stable focus at $(x,y)=(1,0)$ to the origin and makes
the linear part block-diagonal. The resulting system is of the form
\begin{equation}
\dot{\boldsymbol{\varphi}}=\mathbf{A}\boldsymbol{\varphi}+\mathbf{f}\left(\boldsymbol{\varphi}\right)+\varepsilon\hat{\mathbf{F}}(t),\qquad\mathbf{f}\left(\boldsymbol{\varphi}\right)=\mathcal{O}\left(\left|\boldsymbol{\varphi}\right|^{2}\right),\label{eq:duffing_transformed}
\end{equation}
where 
\begin{equation}
\mathbf{A}=\begin{pmatrix}-\alpha & -\omega\\
\omega & -\alpha
\end{pmatrix},\qquad\omega=1.4142,\quad\alpha=0.00707,\label{eq:duffing_linearpart}
\end{equation}

and $\hat{\mathbf{F}}(t)$ is the transformed image of the physical
forcing vector in \eqref{eq:duffing_non_transformed}. We first consider
the unforced system with $\varepsilon=0$. In this case, the 2D slow
SSM of the fixed point coincides with the phase space $\mathbb{R}^{2}$
and hence no further model reduction is possible. However, since the
non-resonance conditions (\ref{eq:nonresonance condition2-2}) hold
for the linear part \eqref{eq:duffing_linearpart}, the system is
analytically linearizable near the origin. The linearizing transformation
and its inverse can both be computed from eq. \eqref{eq:duffing_transformed},
as outlined in eq. \eqref{eq:DDL transformation convergent}. For
reference, we carry out this linearization analytically up to order
$k=9$. The Taylor series of the linearization is estimated to converge
for $\left|\boldsymbol{\varphi}\right|<R_{crit}\approx0.15$. The
details of the calculation can be found in the repository \citep{kaszas_haller_24_code}.

We now compare the analytic linearization results it to DMD, EDMD
and DDL, with all three trained on the same three trajectories, launched
both inside and outside the domain of convergence of the analytic
linearization. The polynomial order of approximation is $k=5$ for
both the EDMD and the DDL algorithms. The performance of the various
methods is compared in Fig. \ref{fig:duffing_1}. Close to the fixed
point, in the domain of convergence of the analytic linearization,
all three methods perform well. Moving away from the fixed point,
the analytic linearization is no longer possible. Both DMD and EDMD
perform worse, while DDL continues to accurately linearize the system
even outside the domain of convergence of the analytic linearization.

\begin{figure}
\begin{centering}
\includegraphics[width=1\textwidth]{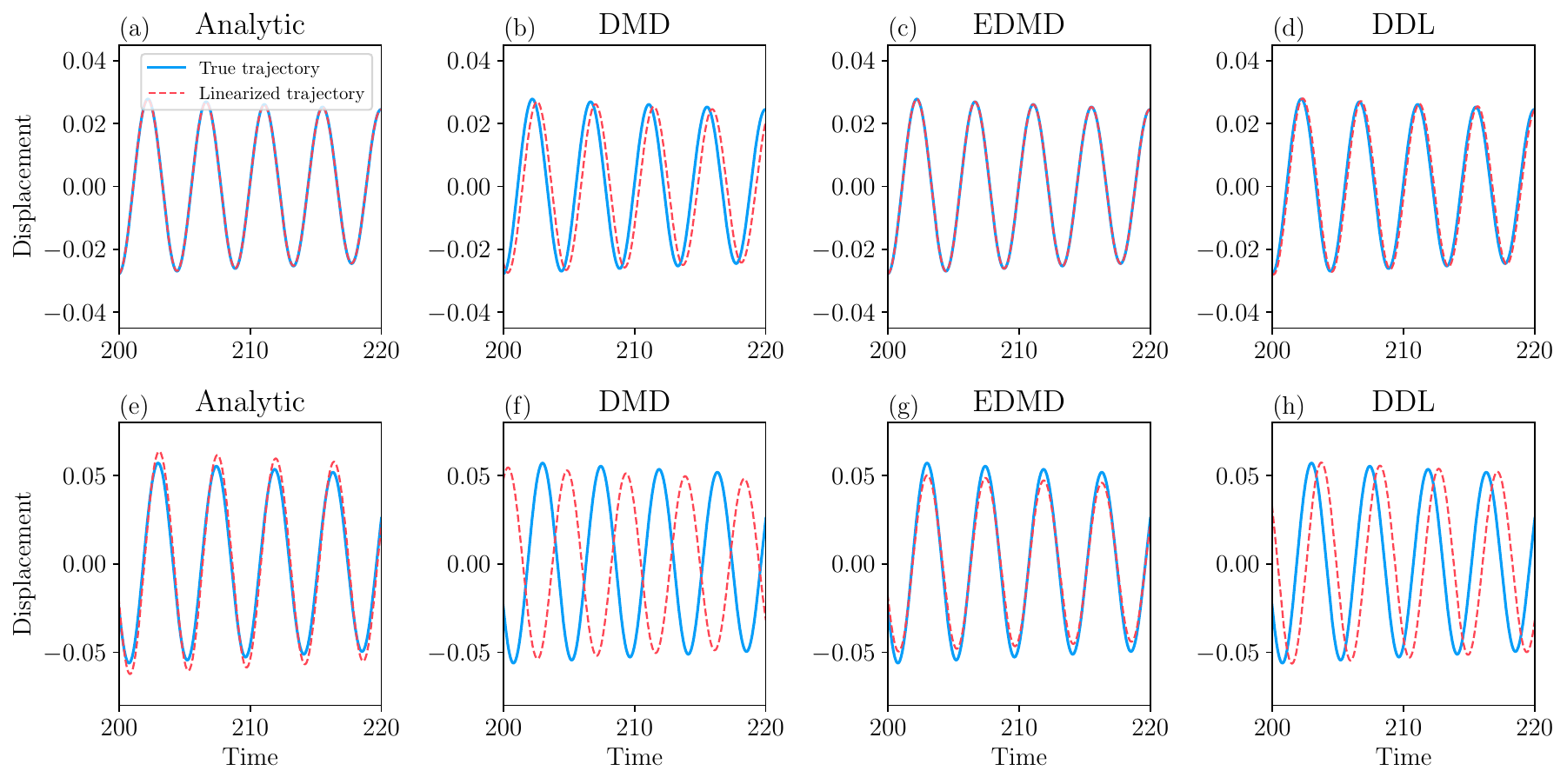}
\par\end{centering}
\caption{\label{fig:duffing_1}Comparison of the time evolution of the linearized
trajectories (red) and the full trajectories of the nonlinear system
(\ref{eq:duffing_transformed}) (blue), \eqref{eq:duffing_transformed}.
(a)-(d): Analytic linearization, DMD, EDMD and DDL models trained
and evaluated on trajectories inside the domain of convergence. (e)-(h):
Same as (a)-(d) but outside the domain of convergence of the analytic
linearization.}
\end{figure}

Using formula (\ref{eq:linearized flow on SSM-1}) and our DDL-based
model, we can also predict the response of system \eqref{eq:duffing_transformed}
for the forcing term of the form 
\begin{equation}
\varepsilon\hat{\mathbf{F}}(t)=\varepsilon\left(\begin{array}{c}
-0.006\\
1.225
\end{array}\right)\cos\Omega t,\label{eq:duffing_forcing}
\end{equation}
without using any data from the forced system. As the forced DDL model
(\ref{eq:linearized flow on SSM-1}) is nonlinear, it can capture
non-linearizable phenomena such as coexisting of stable and unstable
periodic orbits arising under the forcing. We can also make a forced
response prediction from DMD simply by setting $D\boldsymbol{\ell}(\boldsymbol{\gamma})=\boldsymbol{0}$
in eq. (\ref{eq:linearized flow on SSM-1}). As an inhomogeneous linear
system of ODEs, however, this forced DMD model cannot predict coexisting
stable and unstable periodic orbits.

In Fig. \ref{fig:duffing_2}, we compare the forced predictions of
the analytic linearization, DMD, and DDL to those computed from the
nonlinear system directly via the continuation software COCO of \citet{dankowicz2013}.
Since the forced and linearized systems are also nonlinear, we use
the same continuation software to determine the stable and unstable
branches of periodic orbits. 

As expected, the analytic linearization is accurate while the forced
response is inside the domain of convergence but deteriorates quickly
for larger amplitudes. DMD gives good predictions for the peaks of
the forced response diagrams, but cannot account for any of the nonlinear
softening behavior, i.e., the overhangs in the curves that signal
multiple coexisting periodic responses at the same forcing frequency.
In contrast, while the DDL model of order $k=5$ starts becoming inaccurate
for peak prediction at larger amplitudes outside the domain of analytic
linearization, it continues to capture accurately the overhangs arising
from non-linearizable forced response away from the peaks. Notably,
DDL even identifies the unstable branches (in dashed lines) of the
periodic response accurately. For completeness, we also show results
of approximate DDL, by assuming \eqref{eq:approximate ddl assumption}
in Appendix \ref{sec:approximate ddl}.

\begin{figure}
\centering{}\includegraphics[width=0.9\textwidth]{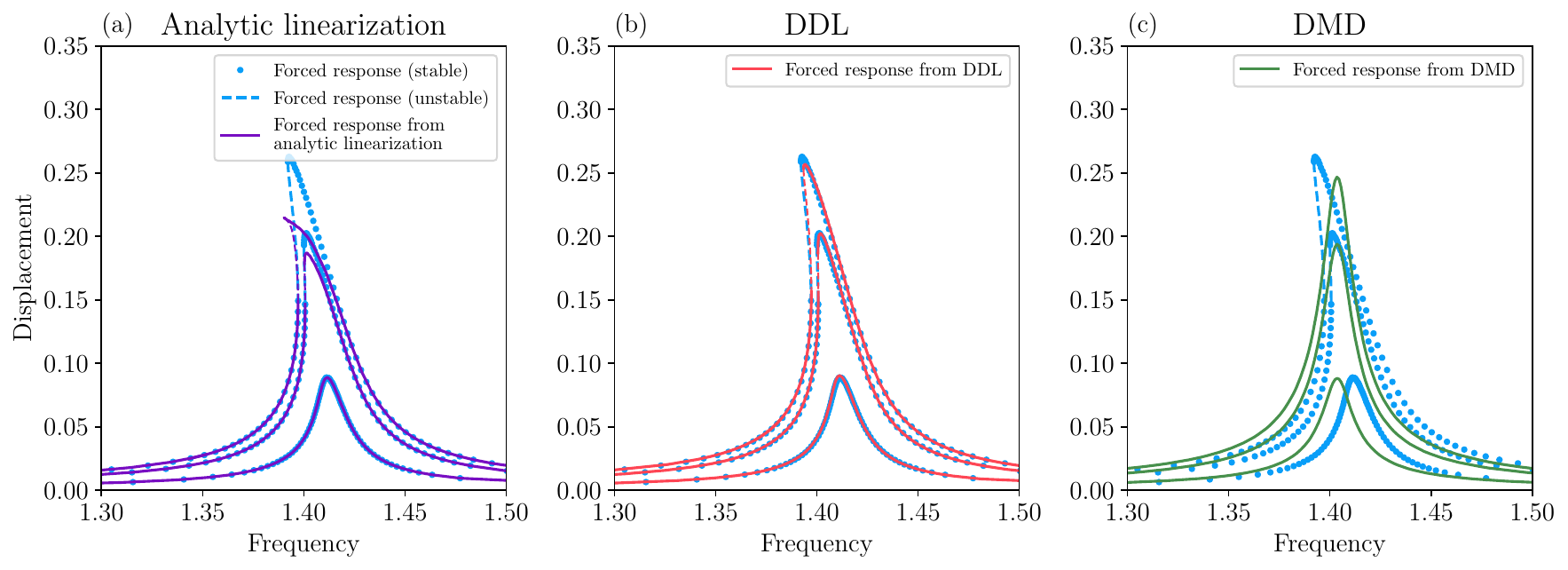}\caption{Periodic response of the Duffing oscillator under the forcing \eqref{eq:duffing_transformed}.
The three distinctly colored forced response curves correspond to
$\varepsilon=0.001,0.002,0.0028$. (a) Actual nonlinear forced response
from numerical continuation (blue) and prediction for it from analytic
linearization (purple) (b) Forced response predictions from DDL (red)
(c) Forced response predictions from DMD (green). The training data
for panels (b) and (c) is the same unforced trajectory data set as
the one used in Fig. \ref{fig:duffing_1}. \label{fig:duffing_2}}
\end{figure}

\subsection{Water sloshing experiment in a tank\label{subsec:Sloshing}}

In this section, we analyze experimental data generated by \citet{bauerlein2021}
for forced and unforced fluid sloshing in a tank. Previous studies
of this data set used nonlinear SSM-reduction to predict forced response
( \citet{cenedese22a,axas22,axas23}). Here we will use DMD and DDL
to extract and compare linear reduced-order models from unforced trajectory
data, then use them to predict and verify forced response curves obtained
from forced trajectory data. Neither DMD nor DDL is expected to outperform
the fully nonlinear approach of SSM reduction, so we will only compare
them against each other.

The tank in the experiments is mounted on a platform that is displaced
sinusoidally in time with various forcing amplitudes and frequencies
(Fig. \ref{fig:sloshing}(a)). To train DMD and DDL, we use unforced
sloshing data obtained by freezing the movement of the tank near a
resonance and recording the ensuing decaying oscillations of the water
surface with a camera under they die out. The resulting videos serve
as input data to our analysis. Specifically, the horizontal position
of the center of mass of the fluid is extracted from each video frame
tracked and used as the single scalar observable. 

During such a resonance decay experiment, the system approaches its
stable unforced equilibrium via oscillations that are dominated by
a single mode. In terms of the phase space geometry, this means an
approach to a stable fixed point along its 2D slow SSM $\mathcal{W}\left(E\right)$
tangent to the slowest 2D real eigenspace $E$ . As we only have a
single observable from the videos, we use delay embedding to generate
a larger observable space that can accommodate the 2D manifold $\mathcal{W}\left(E\right)$.
As discussed by \citet{cenedese22a}, we need an at least 5D observable
space for this purpose by the Takens embedding theorem. In this space,
$\mathcal{W}\left(E\right)$ turns out to be nearly flat for short
delays (see \citet{axas23}), which allows us to use a linear approximation
for its parametrization. The reduced coordinates on $\mathcal{W}\left(E\right)\approx E$
can then be identified via a singular value decomposition of the data
after one removes initial transients from the experimental data. The
end of the transients can be identified as a point beyond which a
frequency analysis of the data shows only one dominant frequency,
the imaginary part of the eigenvalue corresponding to $E$. 

All this analysis has been carried out using the publicly available
SSMLearn package \citet{cenedese21}. With $\mathcal{W}\left(E\right)$
identified, we use the DDL method with order $k=5$ to find the linearizing
transformation and the linearized dynamics on $\mathcal{W}\left(E\right)$.
In Fig \ref{fig:sloshing}(b) we show the prediction of the model
on a decaying trajectory reserved for testing. The displacements are
reported as percentage values, with respect to the depth of the tank.
In Fig. \ref{fig:sloshing}c-d, we show predictions from the DMD and DDL models 
for the forced response, compared with the experimentally observed
response. Since the exact forcing function is unknown, we follow the
calibration procedure outlined by \citet{cenedese22a} to find an
equivalent forcing amplitude in the reduced-order model. 

We present data for three forcing amplitudes. The DDL predictions
are accurate up to $0.17\%$ amplitude, even capturing the softening
trend. The largest-amplitude forcing resulted in response significantly
outside the range of the training data; in this range, we were unable
to find the converged forced response from DDL. We also show the corresponding
DMD-predictions in Fig. \ref{fig:sloshing}c. Although the linear
response can formally be evaluated for any forcing amplitude, DMD
shows no trace of the softening trend, and is even inaccurate for
low forcing amplitudes. 

\begin{figure}
\centering{}\includegraphics[width=0.99\textwidth]{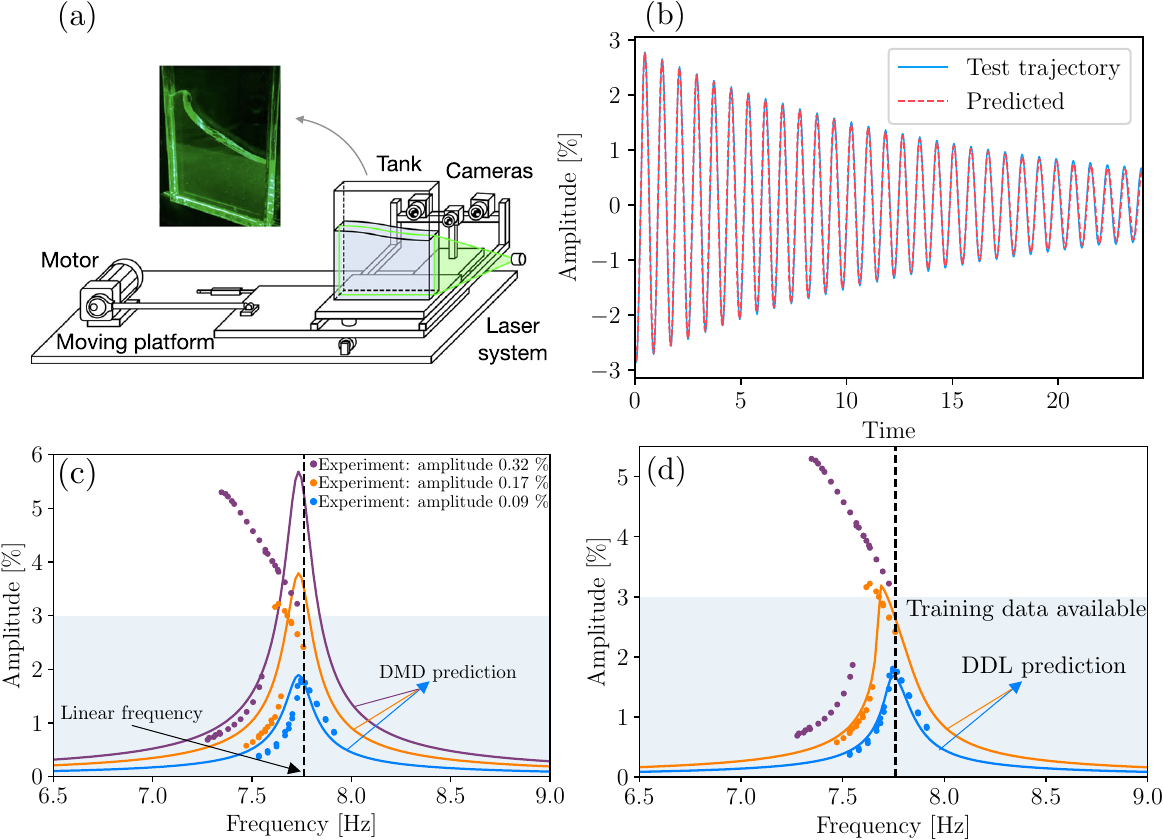}\caption{(a) Schematic representation of the experimental setup (adopted from
\citet{cenedese22a}). (b) Prediction of the decay of a test trajectory
with order-$k=5$ DDL. (c) Prediction of the forced response from
DMD. (d) Prediction of the forced response from DDL. Light shading
indicates the domain, in which training data for DMD and DDL was available.
\label{fig:sloshing}}
\end{figure}

\subsection{Model reduction and foliation in a nonlinear oscillator chain \label{subsec:Oscillator-chain}}

As a final example, we consider the dynamics of a chain of nonlinear
oscillators, which has been analyzed in the SSMLearn package \citet{cenedese21}.
Denoting the positions of the oscillators as $q_{i}$ for $i=1,...,5$,
we assume that the springs and dampers are linear, except for the
first oscillator. The non-dimensionalized equations of motion can
be written as 

\begin{equation}
\mathbf{M\ddot{\mathbf{q}}}+\mathbf{C}\dot{\mathbf{q}}+\mathbf{K}\mathbf{q}+\mathbf{f}(\mathbf{q},\dot{\mathbf{q}})=\mathbf{0},\label{eq:oscillator chain second order}
\end{equation}
where $\mathbf{M}=\mathbf{I}$; the springs have the same linear stiffness
$k=1$ which is encoded in $\mathbf{K}$ via nearest-neighbor coupling.
The damping is assumed to be proportional, i.e., we specifically set
$\mathbf{C}=0.002\mathbf{M}+0.005\mathbf{K}$. 

Three numerically generated training trajectories show decay to the
$\mathbf{q}=\mathbf{0}$ fixed point, as expected from the damped
nature of the linear part of the system. In this example, we also
seek to capture some of the transients, which motivates us to select
the slow SSM $\mathcal{W}(E)$ to be 4D, tangent to the spectral subspace
$E=E_{1}\oplus E_{2}$ spanned by the the slowest mode ($E_{1})$
and the second slowest mode ($E_{2}$). As the mode corresponding
to $E_{2}$ does disappear over time from the decaying signal, there
is no resonance between the eigenvalues and hence Theorem \ref{thm:DDL}
is applicable. As already noted, numerical data from a generic physical
system described by eq. (\ref{eq:oscillator chain second order})
will be free from resonances. An exception is a $1\colon1$ resonance
arising from a perfect symmetry, but this resonance is not excluded
by Theorem 4 and has is amenable to DDL.

Within the 4D SSM $\mathcal{W}(E)$, we also demonstrate how to optimally
reduce the dynamics to its 2D slowest SSM $\mathcal{W}(E_{1})$. As
explained in Section \ref{subsec:foliation}, to find the trajectory
in $\mathcal{W}(E_{1})$ with which a given trajectory close to $\mathcal{W}(E)$
ultimately synchronizes, we need to project along a point $\mathbf{q}_{0}$
of the full trajectory $\mathbf{q}(t)$ first onto $\mathcal{W}(E)$
orthogonally to obtain a point $\mathbf{q}_{0}^{4D}\in\mathcal{W}(E)$.
We then need to identify the stable fiber $\mathcal{F}_{\mathbf{q}_{0}^{2D}}$
in $\mathcal{W}(E)$ for which $\mathbf{q}_{0}^{4D}\in\mathcal{F}_{\mathbf{q}_{0}^{2D}}$
holds. Finally, one has to project along $\mathcal{F}_{\mathbf{q}_{0}^{2D}}$
to locate its base point $\mathbf{q}_{0}^{2D}\in\mathcal{W}(E_{1})$.
The trajectory through $\mathbf{q}_{0}^{2D}$ in $\mathcal{W}(E_{1})$
will then be the one with which the full trajectory $\mathbf{q}(t)$
will synchronize faster than with any other trajectory. As noted in
Section \ref{subsec:foliation}, computing the full nonlinear stable
foliation 
\begin{equation}
\mathcal{W}(E)=\bigcup_{\mathbf{q}_{0}^{2D}\in\mathcal{W}(E_{1})}\mathcal{F}_{\mathbf{q}_{0}^{2D}}^{0}\label{eq:foliation in oscillator chain}
\end{equation}
of $\mathcal{W}(E)$ is simple in the linearized coordinates, in which
it can be achieved via a linear projection along the faster eigenspace
$E_{2}$. 

We use a third-order polynomial approximation for $\mathcal{W}(E)$
based on the three training trajectories. The polynomials depend on
the reduced coordinates we introduce along $E$ using a singular value
decomposition of the trajectory data. These reduced coordinates are
shown in Fig. \ref{fig:oscillator_chain_1}, where we show a representative
training trajectory, the 2D slow SSM $E_{1}$, as well as the foliation
\eqref{eq:foliation in oscillator chain} computed from DDL for this
specific problem. 

\begin{figure}
\centering{}\includegraphics[width=0.99\textwidth]{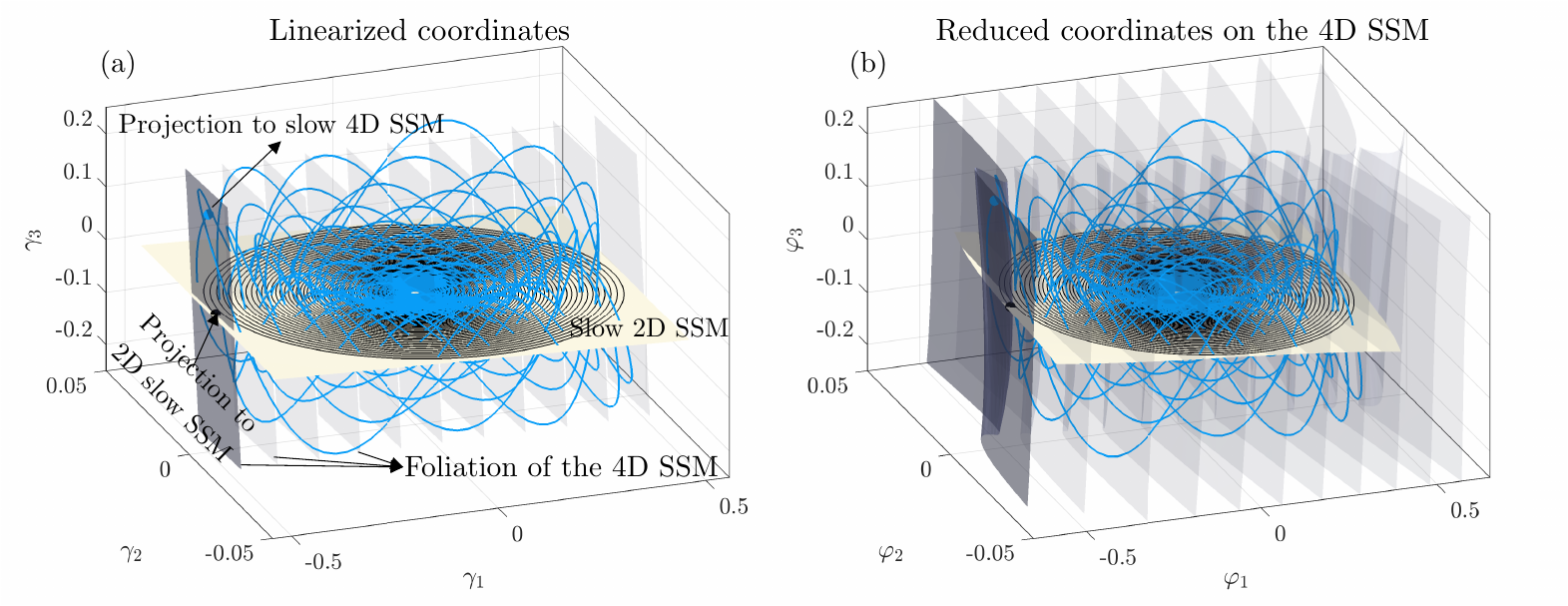}\caption{Reduced coordinates of the 4D SSM $\mathcal{W}(E)$ for the oscillator
chain. The slow 2D SSM $\mathcal{W}(E_{1})$, a typical trajectory
and its projection to the slow SSM along the fibers $\mathcal{F}_{\mathbf{q}_{0}^{2D}}^{0}$
are also shown. Panel (a) shows the linearized coordinates computed
from DDL, and (b) shows their image under the inverse of the linearizing
transformation. The order of approximation used in DDL is 3. \label{fig:oscillator_chain_1}}
\end{figure}

We also evaluate the DDL-based predictions on $\mathcal{W}(E)$ and
$\mathcal{W}(E_{1})$ by comparing them to predictions from DMD and
EDMD. Performing DMD and EDMD with the data first projected to $E$
can be interpreted as finding the linear approximation to the dynamics
in $\mathcal{W}(E)$. Similarly, performing DMD and EDMD with the
data first projected to $E_{1}$ can be interpreted as finding the
linear approximation to the dynamics in $\mathcal{W}(E_{1})$. These
are to be contrasted with performing DDL that finds the linearized
reduced dynamics within $\mathcal{W}(E)$, which in turn contains
the linearized reduced dynamics within $\mathcal{W}(E_{1})$. Figure
\ref{fig:oscillator_chain_2} shows that DMD and EDMD both perform
similarly to DDL on $\mathcal{W}(E)$. However, the 2D DMD and EDMD
results obtained for $\mathcal{W}(E_{1})$ are noticeably less accurate
than the DDL results. 

\begin{figure}
\centering{}\includegraphics[width=0.6\textwidth]{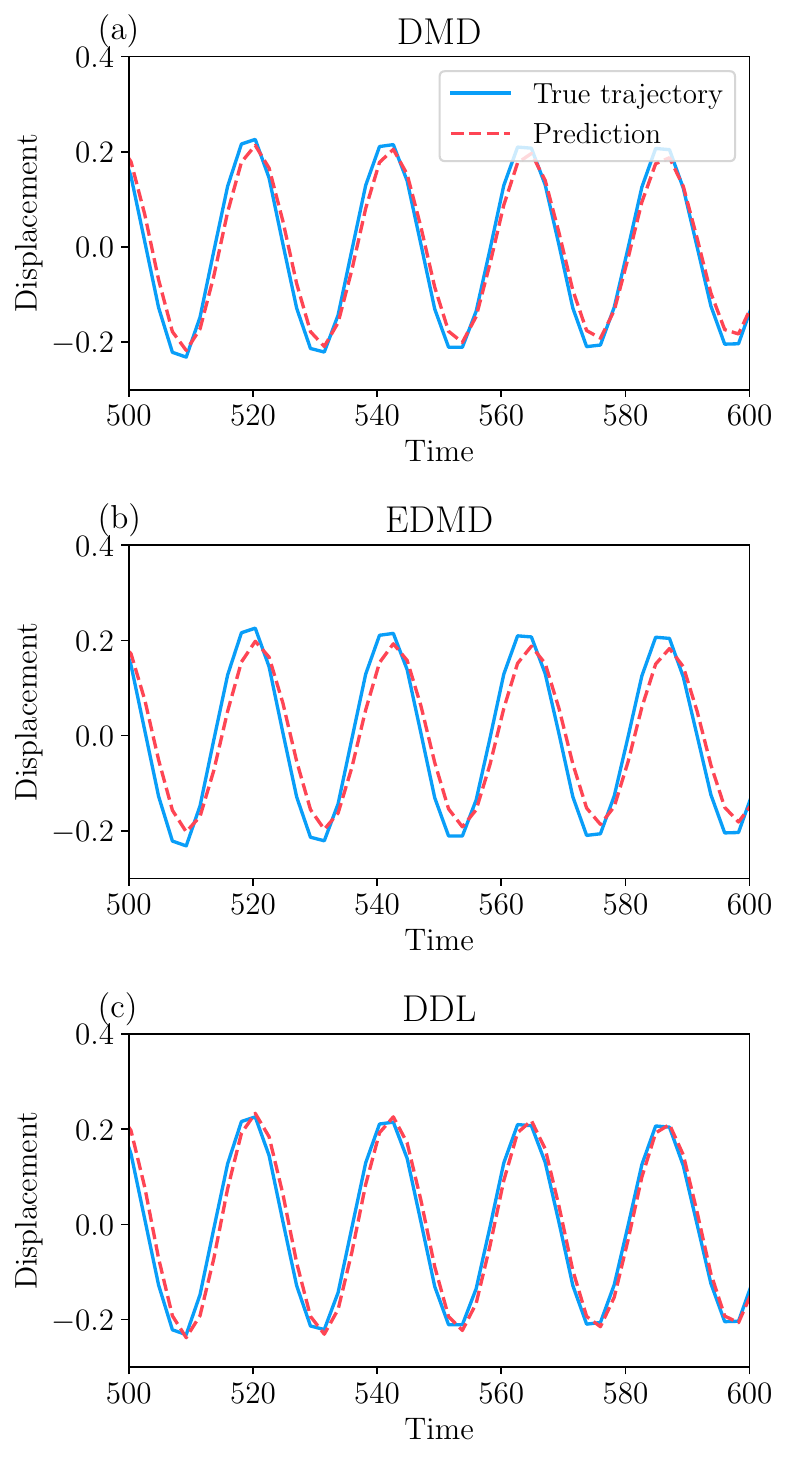}\caption{Predictions of (a) DMD (b) EDMD and (c) DDL models on a test trajectory
of the oscillator chain. The order of approximation for DDL and EDMD
is $k=3$ . \label{fig:oscillator_chain_2}}
\end{figure}

\section{Conclusions}

We have given a new mathematical justification for the broadly used
DMD procedure to eliminate the shortcomings of prior proposed justifications.
Specifically, we have shown that under specific non-degeneracy conditions
on the $n$-dimensional dynamical system, on $d\leq n$ observable
functions defined for that system, and on the actual data from these
observables, DMD gives a leading-order approximation to the observable
dynamics on an attracting $d$-dimensional spectral submanifold (SSM)
of the system. 

This result covers both discrete and continuous dynamical systems
even for $n=\infty$. Our Theorem 1 only makes explicit non-degeneracy
assumptions on the observables which will hold with probability one
in practical applications. This is to be contrasted with prior approaches
to DMD and its variants based on the Koopman operator, whose assumptions
on the observables fail with probability one on generic observables.

Our approach also yields a systematic procedure that gradually refines
the leading-order DMD approximation of the reduced observable dynamics
on SSMs to higher orders. This procedure, which we call data-driven
linearization (DDL), builds a nonlinear coordinate transformation
under which the observable becomes linear on the attracting SSM. We
have shown on several examples how DDL indeed outperforms DMD and
extended DMD (EDMD), as expected. In addition to this performance
increase, DDL also enables a prediction of truly nonlinear forced
response from unforced data within its training range. Although we
have only illustrated this for periodically forced water sloshing
experiments in a tank, recent results on aperiodically time-dependent
SSMs by \citet{haller24} allow us to predict more general forced
response using DDL trained on unforced observable data.

Despite all these advantages, DDL (as any linearization method) remains
applicable only in parts of the phase space where the dynamics are
linearizable. Yet SSMs continue to exist across basin boundaries and
hence are able to carry characteristically nonlinear dynamics with
multiple coexisting attractors. For such nonlinearizable dynamics,
data-driven nonlinear SSM-reduction algorithms, such as SSMLearn and
fastSSM, are preferable and have been showing high accuracy and predictive
ability in a growing number of physical settings (see, e.g., \citet{cenedese22a,cenedese22b,axas22,kaszas22,alora23b,kaszas24,liu24b}).\vskip 1cm

\textbf{Acknowledgement} We are grateful to Matthew Kvalheim and Shai
Revzen for several helpful comments on an earlier version of this
manuscript.\vskip 1cm

\textbf{Funding }This work was supported by the Swiss National Science
Foundation.

\vskip 1cm

\textbf{Data Availability} All data and codes used in this work are
downloadable from the repository \url{https://github.com/haller-group/DataDrivenLinearization}.

\vskip 1cm

\textbf{Competing Interests }The authors declare that they have no
conflict of interest.

\vfil\eject

\setcounter{page}{1}
\begin{center}
{\LARGE{}Supplementary Information for}{\LARGE\par}
\par\end{center}

\begin{center}
\emph{\LARGE{}Data-Driven Linearization of Dynamical Systems}{\LARGE{} }{\LARGE\par}
\par\end{center}

\begin{center}
{\large{}George Haller}\footnote{{\large{}Corresponding author. Email: georgehaller@ethz.ch}}{\large{}
and Bálint Kaszás}\\
{\large{} Institute for Mechanical Systems}\\
{\large{} ETH Zürich}\\
{\large{} Leonhardstrasse 76, 8092 Zürich, Switzerland}{\large\par}
\par\end{center}

\appendix

\section{DMD and the Koopman operator\label{sec:DMD-and-Koopman}}

\subsection{Observable dynamics and the Koopman operator\label{subsec:Koopman}}

A classic approach, introduced first by \citet{koopman31} and revived
recently by multiple authors (see, e.g., by \citet{budisic12,mezic13,kutz16}
and the references cited therein), describes observable evolution
via the Koopman operator $\boldsymbol{\mathcal{K}}^{t}\colon\mathcal{G}^{d}\to\mathcal{G}^{d}$,
defined on a Banach space $\mathcal{G}^{d}$ of $d$-dimensional observable
functions as the pull-back operation on observables under the flow
map $\mathbf{F}^{t}$ of system \eqref{eq:nonlinear system0}. Specifically,
\begin{equation}
\boldsymbol{\mathcal{K}}^{t}\left[\boldsymbol{\phi}\right]\,(\mathbf{x}_{0}):=\boldsymbol{\phi}\left(\mathbf{F}^{t}(\mathbf{x}_{0})\right),\label{eq:Koopman_definition}
\end{equation}
or, in more compact notation, 
\begin{equation}
\boldsymbol{\mathcal{K}}^{t}\left[\boldsymbol{\phi}\right]=\boldsymbol{\phi}\circ\mathbf{F}^{t},\quad\boldsymbol{\phi}\in\mathcal{G}^{d}.\label{eq:conjugacy on an observable}
\end{equation}

As $\mathcal{G}^{d}$ is a complete vector space, linear combinations
of $\ell$ observables $\phi_{1},\ldots,\phi_{\ell}\in\mathcal{G}^{d}$
are also observables in $\mathcal{G}^{d}$ and hence $\boldsymbol{\mathcal{K}}^{t}$
can be evaluated on them. Such an evaluation gives 
\begin{align}
\boldsymbol{\mathcal{K}}^{t}\left(c_{1}\phi_{1}+\ldots+c_{1}\phi_{\ell}\right)(\mathbf{x}_{0}) & =c_{1}\boldsymbol{\mathcal{K}}^{t}\left[\phi_{1}\right](\mathbf{x}_{0})+\ldots+c_{\ell}\boldsymbol{\mathcal{K}}^{t}\left[\phi_{\ell}\right](\mathbf{x}_{0})\label{eq:Koopman is linear}
\end{align}

\emph{This last equation shows that the mapping $\boldsymbol{\mathcal{K}}^{t}$
is linear with respect to the choice of observables while the system's
initial condition $\mathbf{x}_{0}$ is kept fixed. In contrast, DMD
seeks to find an approximate mapping that is linear with respect to
the choice of the system's initial conditions while the observable
is kept fixed.} Indeed, in applications of DMD, eq. (\ref{eq:approximation by DMD})
is posed for a fixed observable vector evaluated on various initial
conditions. EDMD is also trained on system trajectories launched from
different initial conditions, observed under the same general function
of an initially fixed set of observables. 

On a broader note, the linearity of the Koopman operator with respect
to changes in the choice of an observable, as seen in eq. (\ref{eq:Koopman is linear}),
does not imply linear dynamics for observations of the system under
a fixed observable. For instance, let $\phi_{1}(\mathbf{x})$ denote
the observation of the displacement, measured in meters, of the end
point of a nonlinear beam whose current state in its phase space is
denoted by $\mathbf{x}$. If, instead, we want to observe the position
of the endpoint in millimeters, then we switch the observable $\phi_{2}(\mathbf{x})=10^{3}\phi_{1}(\mathbf{x}),$
and hence for the observed initial state $\mathbf{x}_{0}$ of the
beam, we have the change \linebreak{}
\[
\phi_{2}(\mathbf{x}_{0})=c\phi_{1}(\mathbf{x}_{0}),\qquad c=10^{3},
\]
in the initial observation. It is then no surprise that this change
in the initial observable will transform to the same change in the
current observations of the endpoint at time $t$, i.e., 
\[
\boldsymbol{\mathcal{K}}^{t}\left[\phi_{2}\right](\mathbf{x}_{0})=c\boldsymbol{\mathcal{K}}^{t}\left[\phi_{1}\right](\mathbf{x}_{0}).
\]
This simple fact, however, just reflects that the final results at
time $t$ change by the same multiple as their initial conditions
if we decide to change the physical units in our measurements. Clearly,
this fact does not imply any linear dynamics for the displacement
of the endpoint of the beam. Indeed, $\boldsymbol{\mathcal{K}}^{t}\left[\phi_{2}\right](\mathbf{x}_{0})$
may well be a chaotic signal.

Another example would be tracking the total population of a continent
and denoting the current state of the population by $\mathbf{x}$.
Let $\phi_{1}(\mathbf{x})$ denote specifically the observable returning
the population of country $A$ and $\phi_{2}(\mathbf{x})$ denote
the observable returning the population of country $B$ within the
same continent. If we decide now to jointly observe the population
of these two countries, then we are in effect passing to a third observable
at at initial state $\mathbf{x}_{0}$ to obtain 
\[
\phi_{3}(\mathbf{x}_{0})=\phi_{2}(b_{0})+\phi_{1}(\mathbf{x}_{0}),
\]
for the initial observations. To obtain the total population of countries
$A$ and $B$ at the current state $\mathbf{x}$ of the total population,
we can simply write 
\[
\boldsymbol{\mathcal{K}}^{t}\left[\phi_{3}\right](\mathbf{x}_{0})=\boldsymbol{\mathcal{K}}^{t}\left[\phi_{1}\right](\mathbf{x}_{0})+\boldsymbol{\mathcal{K}}^{t}\left[\phi_{2}\right](\mathbf{x}_{0}).
\]
Again, this fact by itself does not say anything about the linearity
or nonlinearity of the population growth in either country $A$ or
country $B$. 

Finally, \emph{the Koopman operator is as much nonlinear as it is
linear with respect to the changes in the observables} as it commutes
with any nonlinear single or multi-variable function of observables,
as longs as that function is also an observable. Indeed, for any function
\[
\mathbf{g}\colon\left(\mathcal{G}^{d}\right)^{\ell}\to\mathcal{G}^{d},
\]
the observable $\boldsymbol{\psi}(\mathbf{x}_{0})=\mathbf{g}\left(\phi_{1}(\mathbf{x}_{0}),\ldots,\phi_{\ell}(\mathbf{x}_{0})\right)$
satisfies
\[
\boldsymbol{\mathcal{K}}^{t}\left[\mathbf{g}\left(\phi_{1},\ldots,\phi_{\ell}\right)\right]=\mathbf{g}\left(\boldsymbol{\mathcal{K}}^{t}\left[\phi_{1}\right],\ldots,\boldsymbol{\mathcal{K}}^{t}\left[\phi_{\ell}\right]\right).
\]

To illustrate this general property of the Koopman operator, let $\mathbf{x}$
denote the instantaneous velocity field of a 2D Navier-Stokes flow
with density $\rho$ and let $\phi_{ij}(\mathbf{x})=\left|v_{ij}\right|$
be the specific velocity magnitude along the $(i,j)$ location of
a fixed spatial grid. If we now want to track the local kinetic energy
of the flow at $(i,j)$, we do not have to reformulate and re-solve
the Navier--Stokes equation for $\phi_{\mathrm{kin}}(\mathbf{x})=\frac{1}{2}\rho\phi_{ij}^{2}(\mathbf{x})=\left|x_{ij}\right|^{2}$.
Rather, we can use our already available observations of $\phi_{ij}(\mathbf{x})$
of the velocity field starting from an initial velocity field $\mathbf{x}_{0}$
to obtain $\boldsymbol{\mathcal{K}}^{t}\left[\phi_{\mathrm{kin}}\right](\mathbf{x}_{0})=\frac{1}{2}\rho\left[\boldsymbol{\mathcal{K}}^{t}\left[\phi_{ij}\right](\mathbf{x}_{0})\right]^{2}.$
This, however, does not imply a statement that ``the dynamics of
the 2D Navier--Stokes equation are quadratic in the space of observables''.

\subsection{Differential equation for the dynamics of observables under the Koopman
operator}

We can directly verify that
\[
\boldsymbol{\mathcal{K}}^{t_{1}+t_{2}}\boldsymbol{\phi}=\mathcal{\boldsymbol{\mathcal{K}}}^{t_{2}}\boldsymbol{\mathcal{K}}^{t_{1}}\boldsymbol{\phi}=\boldsymbol{\mathcal{K}}^{t_{1}}\mathcal{\boldsymbol{\mathcal{K}}}^{t_{2}}\boldsymbol{\phi},\quad\left(\boldsymbol{\mathcal{K}}^{t}\right)^{-1}\boldsymbol{\phi}=\boldsymbol{\mathcal{K}}^{-t}\boldsymbol{\phi},
\]
i.e., $\boldsymbol{\mathcal{K}}^{t}$ defines a flow map on the observable
space $\mathcal{G}^{d}$, as long as system \eqref{eq:nonlinear system0}
generates a flow $\mathbf{F}^{t}$ on the phase space $\mathcal{P}$.
To describe the dynamics generated by this flow map, let us select
an arbitrary scalar observable $\boldsymbol{\phi}\in\mathcal{G}^{d}$
and differentiate the defining relation (\ref{eq:Koopman_definition})
with respect to $t$ to obtain\emph{
\begin{equation}
\frac{d}{dt}\left\{ \text{\ensuremath{\mathcal{\boldsymbol{\mathcal{K}}}^{t}\left[\boldsymbol{\phi}\right]}}(\boldsymbol{x}_{0})\right\} =D\boldsymbol{\phi}\left(\mathbf{F}^{t}(x_{0})\right)\mathbf{f}\left(\mathbf{F}^{t}(x_{0})\right).\label{eq:first DE for Koopman}
\end{equation}
}Since we have
\[
D_{\mathbf{x}_{0}}\boldsymbol{\phi}\left(\mathbf{F}^{t}(\mathbf{x}_{0})\right)=D\boldsymbol{\phi}\left(\mathbf{F}^{t}(\mathbf{x}_{0})\right)D\mathbf{F}^{t}(\mathbf{x}_{0}),
\]
and $\mathbf{F}^{t}$ is a diffeomorphism, we can rewrite (\ref{eq:first DE for Koopman})
as
\begin{equation}
\frac{d}{dt}\left\{ \text{\ensuremath{\mathcal{\boldsymbol{\mathcal{K}}}^{t}\left[\boldsymbol{\phi}\right]}}(\mathbf{x}_{0})\right\} =D_{\mathbf{x}_{0}}\boldsymbol{\phi}\left(\mathbf{F}^{t}(\mathbf{x}_{0})\right)\left[D\mathbf{F}^{t}(\mathbf{x}_{0})\right]^{-1}\mathbf{f}\left(\mathbf{F}^{t}(\mathbf{x}_{0})\right)=D_{\mathbf{x}_{0}}\left\{ \mathcal{\boldsymbol{\mathcal{K}}}^{t}\left[\boldsymbol{\phi}\right](\mathbf{x}_{0})\right\} \left[D\mathbf{F}^{t}(\mathbf{x}_{0})\right]^{-1}\mathbf{f}\left(\mathbf{F}^{t}(\mathbf{x}_{0})\right).\label{eq:k1}
\end{equation}
Note, however, that
\begin{equation}
\left[D\mathbf{F}^{t}(\mathbf{x}_{0})\right]^{-1}\mathbf{f}\left(\mathbf{F}^{t}(\mathbf{x}_{0})\right)=\mathbf{f}(\mathbf{x}_{0}),\label{eq:k2}
\end{equation}
where we used the fact that $\mathbf{f}\left(\mathbf{F}^{t}(\mathbf{x}_{0})\right)$
is a solution of the equation of variations $\dot{\boldsymbol{\xi}}=D\mathbf{F}^{t}(\mathbf{x}_{0})\boldsymbol{\xi}$
and hence we have $\mathbf{f}\left(\mathbf{F}^{t}(\mathbf{x}_{0})\right)=D\mathbf{F}^{t}(\mathbf{x}_{0})\mathbf{f}\left(\mathbf{x}_{0}\right)$.
Therefore, we obtain from eqs. (\ref{eq:k1}) and (\ref{eq:k2}) that
\begin{equation}
\frac{d}{dt}\left(\mathcal{\boldsymbol{\mathcal{K}}}^{t}\boldsymbol{\phi}\right)=\boldsymbol{\mathcal{L}}\mathcal{\boldsymbol{\mathcal{K}}}^{t}\boldsymbol{\phi}\label{eq:linear DE for Koopman operator}
\end{equation}
 with the Liouville operator $\boldsymbol{\mathcal{L}}:\mathcal{G}^{d}\mapsto\mathcal{G}^{d}$
defined as
\[
\boldsymbol{\mathcal{L}}\colon\boldsymbol{\phi}\mapsto\mathcal{\boldsymbol{L}}\boldsymbol{\phi}:=D\boldsymbol{\phi}\,\mathbf{f}.
\]

\subsection{An example: The phase space variable as observable\label{subsec:The phase space variable as observable}}

As an illustration of the solution structure of the functional differential
equation (\ref{eq:linear DE for Koopman operator}), consider the
simplest nontrivial case wherein the observable function $\boldsymbol{\phi}$
is just the identity map on $\mathcal{P}=\mathbb{R}^{d}$, i.e., $\hat{\boldsymbol{\phi}}=\mathbf{I}\colon\mathbb{R}^{d}\to\mathbb{R}^{d}$.
In that case, we specifically have the expressions
\[
\hat{\boldsymbol{\phi}}(\mathbf{x}_{0})=\mathbf{x}_{0},\qquad\mathcal{\boldsymbol{\mathcal{K}}}^{t}\hat{\boldsymbol{\phi}}(\mathbf{x}_{0})=\mathbf{F}^{t}(\mathbf{x}_{0}),\qquad\frac{d}{dt}\left\{ \text{\ensuremath{\mathcal{\boldsymbol{\mathcal{K}}}^{t}\hat{\boldsymbol{\phi}}}}\right\} =\dot{\mathbf{F}}^{t}(\mathbf{x}_{0})=\dot{\mathbf{x}},\qquad D_{\mathbf{x}_{0}}\left\{ \mathcal{\boldsymbol{\mathcal{K}}}^{t}\hat{\boldsymbol{\phi}}(\mathbf{x}_{0})\right\} =D\mathbf{F}^{t}(\mathbf{x}_{0}).
\]
Substituting these formulas into the functional differential equation
(\ref{eq:linear DE for Koopman operator}), we obtain 
\begin{equation}
\dot{\mathbf{x}}=D\mathbf{F}^{t}(\mathbf{x}_{0})\mathbf{f}(\mathbf{x}_{0})=\mathbf{f}(\mathbf{x}).\label{eq:simplifies back to noninear system}
\end{equation}
Therefore, for the most commonly used observable in classic dynamical
systems, the phase space variable $\mathbf{x}$, the linear functional
differential equation (\ref{eq:linear DE for Koopman operator}) simplifies
to the original nonlinear ODE (\eqref{eq:nonlinear system0}). 

This again underscores that the linearity of the Koopman operator
does not imply linear dynamics for individual observations $\mathbf{x}(t;\mathbf{x}_{0})$
of the evolution of the initial condition $\mathbf{x}_{0}$. Rather,
it implies that for any real constant $c$, \emph{the scaled initial
observation $\tilde{\boldsymbol{\phi}}\,(\mathbf{x}_{0})=c\hat{\boldsymbol{\phi}}(\mathbf{x}_{0})$
of $\mathbf{x}_{0}$ will evolve into the identically scaled current
observation }$\mathcal{\boldsymbol{\mathcal{K}}}^{t}\hat{\boldsymbol{\phi}}(\mathbf{x}_{0})=c\mathcal{\boldsymbol{\mathcal{K}}}^{t}\hat{\boldsymbol{\phi}}(\mathbf{x}_{0})=c\mathbf{x}(t;\mathbf{x}_{0})$
\emph{of the trajectory at time $t$}, i.e., $\mathcal{\boldsymbol{\mathcal{K}}}^{t}$
will return $\mathbf{x}(t;\mathbf{x}_{0})$ scaled by the same constant
$c$. Most importantly, we will still generally have
\[
\mathbf{x}(t;c\mathbf{x}_{0})\neq c\mathbf{x}(t;\mathbf{x}_{0}),
\]
unless the dynamical system (\ref{eq:simplifies back to noninear system})
is linear. Therefore, the dynamics of the fixed observable $\mathbf{x}(t;\mathbf{x}_{0})$
is not linear in the usual sense, i.e., with respect to changes in
initial conditions. Rather, it is linear with respect to changes in
the observable in which the evolution of the same trajectory is observed.

Equation (\ref{eq:simplifies back to noninear system}) also illustrates
that passing from the autonomous nonlinear ODE \eqref{eq:nonlinear system0}
to the functional differential equation (\ref{eq:linear DE for Koopman operator})
gives as infinite-dimensional linear formulation that is at least
as complicated to solve as the original finite-dimensional nonlinear
ODE formulation.

\subsection{Koopman eigenfunctions\label{subsec:Koopman-eigenfunctions}}

There are, nevertheless, non-generic sets of observables restricted
to which e.q. (\ref{eq:linear DE for Koopman operator}) becomes a
constant-coefficient linear system of ODEs. Examples are observables
falling in the span of eigenfunctions of the Koopman operator $\mathcal{\boldsymbol{\mathcal{K}}}^{t}$,
if and where such eigenfunctions exist. 

Indeed, if for some constant $\lambda\in\mathbb{C}$, the eigenvalue
problem
\begin{equation}
\mathcal{\boldsymbol{\mathcal{K}}}^{t}\left[\boldsymbol{\phi}\right](\mathbf{x}_{0})=e^{\lambda t}\boldsymbol{\phi}(\mathbf{x}_{0}),\quad\mathbf{x}_{0}\in\mathcal{D},\label{eq:Koopman eigenproblem}
\end{equation}
has a solution $\boldsymbol{\phi}\colon\mathcal{D}\to\mathbb{C}^{d}$
over some domain $\mathcal{D}\subset\mathcal{P}$ of the phase space,
then we obtain
\[
\frac{d}{dt}\text{\ensuremath{\mathcal{\boldsymbol{\mathcal{K}}}^{t}\left[\boldsymbol{\phi}\right]}}=\lambda\mathcal{\boldsymbol{\mathcal{K}}}^{t}\left[\boldsymbol{\phi}\right].
\]

Consequently, if $\boldsymbol{\phi}_{1}(\boldsymbol{x}_{0}),\ldots,\boldsymbol{\phi}_{k}(\boldsymbol{x}_{0})\in\mathbb{C}^{d}$
are Koopman eigenfunctions with corresponding Koopman eigenvalues
$\lambda_{1},\ldots,\lambda_{k}\in\mathbb{C}$ and domains of definition
$\mathcal{D}_{1},\ldots,\mathcal{D}_{k}\subset\mathcal{P}$, then
any observable $\boldsymbol{w}\in\mathcal{\Phi}^{d}\left(\cap_{j=1}^{k}\mathcal{D}_{j}\right)$
of the form
\begin{equation}
\boldsymbol{w}(\mathbf{x}_{0})=c_{1}\boldsymbol{\phi}_{1}(\mathbf{x}_{0})+\ldots+c_{k}\boldsymbol{\phi}_{k}(\mathbf{x}_{0})\label{eq:observer in Koopman eigenspace}
\end{equation}
 satisfies
\[
\text{\ensuremath{\mathcal{\boldsymbol{\mathcal{K}}}^{t}\left[\boldsymbol{w}\right]}}=\sum_{j=1}^{k}c_{j}\mathcal{\boldsymbol{\mathcal{K}}}^{t}\left[\boldsymbol{\phi}^{j}\right]=\sum_{j=1}^{k}c_{j}e^{\lambda_{j}t}\boldsymbol{\phi}_{j}.
\]
 Therefore, the coordinate representation $\boldsymbol{W}(t)$ of
$\mathcal{\boldsymbol{\mathcal{K}}}^{t}\left[\boldsymbol{w}\right]$
in the basis $\boldsymbol{\phi}_{1},\ldots,\boldsymbol{\phi}_{j},$
given by
\[
\mathbf{W}(t)=\left(\begin{array}{c}
c_{1}e^{\lambda_{1}t}\\
\vdots\\
c_{k}e^{\lambda_{k}t}
\end{array}\right),
\]
satisfies the $k$-dimensional autonomous linear ODE
\[
\frac{d}{dt}\mathbf{W}=\boldsymbol{\Lambda}\mathbf{W},\quad\boldsymbol{\Lambda}=\left(\begin{array}{ccc}
\lambda_{1} & 0 & 0\\
0 & \ddots & 0\\
0 & 0 & \lambda_{k}
\end{array}\right).
\]
 More generally, if $\boldsymbol{\phi}_{1},\ldots,\boldsymbol{\phi}_{k}$
 are just linearly independent observables in the spectral subspace
$\mathrm{span}\left\{ \boldsymbol{\phi}_{1},\ldots,\boldsymbol{\phi}_{k}\right\} $,
then a similar ODE holds for $W$ in the basis $\boldsymbol{\phi}_{1},\ldots,\boldsymbol{\phi}_{k}.$
In that case, $\Lambda$ is not a diagonal matrix but its eigenvalues
are still $\lambda_{1},\ldots,\lambda_{k}$. 

It is often forgotten, however, that Koopman eigenfunctions satisfying
the eigenvalue problem (\ref{eq:Koopman eigenproblem}) generally
only exist on a subset $\mathcal{D}$ of the phase space. This is
made precise by the following simple observation:
\begin{prop}
\label{prop:Koopman blow-up} At least one principal Koopman eigenfunction
blows up (i.e., becomes unbounded) at the boundary of a domain of
attraction or repulsion of a fixed point around which the underlying
dynamical system admits a local $C^{1}$linearization.
\end{prop}
\begin{proof}
We only prove the statement for domains of attraction of fixed points
of continuous dynamical systems. Domains of repulsion can be handled
in the same fashion in backward time and the proof for discrete dynamical
systems is similar. Let $\mathbf{x}_{0}$ be a point on the boundary
$\partial\mathcal{B}$ of a domain of attraction $\mathcal{B}$ of
a fixed point $p$ of a continuous dynamical system. Then, arbitrarily
close to $\mathbf{x}_{0}$, there are initial conditions $\hat{\mathbf{x}}_{0}$
in the domain of attraction with arbitrarily long times of flight
to the ball $B_{\rho}(p)$ around $p$ in which the dynamical system
can be $C^{1}$linearized.

\citet{lan13} extended local $C^{1}$ linearization near an attracting
fixed point via the backward-time flow map to a global $C^{1}$ linearizing
transformation $\mathbf{x}=\mathbf{h}(\mathbf{y})$ within $\mathcal{B}$
that maps the original ODE within $\mathcal{B}$ to $\dot{\mathbf{y}}=\mathbf{A}\mathbf{y}$
defined on $\mathbb{R}^{n}$. By the construction of this linearizing
transformation, initial conditions $\hat{\mathbf{x}}_{0}$ whose times
of flight to $B_{\rho}(p)$ is large will be mapped into initial conditions
$\hat{\mathbf{y}}_{0}$ of the linearized system that are far from
the origin $\mathbf{y}=\mathbf{0}$. 

The unique principal Koopman eigenfunctions of the nonlinear system
are known to be 
\begin{equation}
\boldsymbol{\phi}_{j}(\mathbf{x})=\left\langle \mathbf{h}^{-1}(\mathbf{x}),\mathbf{v}_{j}\right\rangle ,\quad j=1,\ldots,n,\label{eq:nonlinear Koopman eigenfunctions}
\end{equation}
with $\mathbf{v}_{j}$ denoting the left eigenvectors or $\mathbf{A}$.
As we have just concluded that $\mathbf{y}=\mathbf{h}^{-1}(\mathbf{x})$
will admit arbitrarily large values arbitrarily close to any point
$\mathbf{x}_{0}\in\mathcal{\partial\mathcal{B}}$, it follows that
there will be at least one $j$ for which the principal Koopman eigenfunction
$\boldsymbol{\phi}_{j}(\mathbf{x})$ blows up at $\mathbf{x}_{0}\in\partial\mathcal{B}$. 
\end{proof}
We note that the statement of Proposition \ref{prop:Koopman blow-up}
can also be deduced from the more general results in Theorem 3 of
\citet{kvalheim23}.

As a consequence of Proposition \ref{prop:Koopman blow-up}, one cannot
simply patch together Koopman eigenfunctions outside boundaries of
domains of attraction and thereby provide a viable Koopman-linearization
for the whole phase space, as is often suggested. Indeed, even if
found somehow, Koopman eigenfunctions would become unmanageable for
the purposes of observer expansions even before reaching basin boundaries,
as shown explicitly by the examples of \citet{page19}.

\section{Proof of Theorem \ref{thm: DMD} and technical remarks \label{sec: Proof of Theorem 1}}

Under assumption (A1) of hyperbolicity of the theorem and under assumption
(A2) on the smoothness class $C^{2}$ of the dynamical system, a refinement
of the Hartman--Grobman theorem (see \citet{guckenheimer83}) by
\citet{hartman60} guarantees the local existence of a near-identity,
linearizing change of coordinates of the form\footnote{\citet{hartman60} only states that $h\in C^{1}$ in his main theorem.
He adds, however, that his proof of the theorem also shows that $Dh$
is uniformly Hölder continuous in $U$ with Hölder exponent $\beta\in\left(0,1\right)$,
which implies eq. (\ref{eq:Hartman linearization}). } 
\begin{equation}
\mathbf{x}=\mathbf{y}+\mathbf{h}(\mathbf{y}),\qquad\mathbf{y}\in U\subset\mathbb{R}^{n},\qquad\mathbf{h}(\mathbf{y}),\mathbf{h}^{-1}(\mathbf{y})=\mathcal{O}\left(\left|\mathbf{y}\right|^{1+\beta}\right),\label{eq:Hartman linearization}
\end{equation}
for some $\beta\in\left(0,1\right)$, under which system \eqref{eq:more specific nonlinear system}
takes the exact linear form 
\begin{equation}
\dot{\mathbf{y}}=\mathbf{A}\mathbf{y}.\label{eq:linearized system}
\end{equation}

As pointed out by \citet{haller23} in the context of the more restrictive,
$C^{\infty}$ linearization theorem of \citet{sternberg58}, the inverse
of the linearizing transformation (\ref{eq:Hartman linearization})
maps the $d$-dimensional spectral subspace $E$ of the linearized
system (\ref{eq:linearized system}) into a $d$-dimensional, normally
attracting spectral submanifold (SSM), 
\begin{equation}
\mathcal{W}(E)=\left(\mathbf{I}+\mathbf{h}\right)^{-1}(E),\label{eq:W(E) construction}
\end{equation}
for system \eqref{eq:more specific nonlinear system}. The internal
dynamics of this SSM govern the longer-term behavior of all trajectories
near the origin. Generally, there will be infinitely many such invariant
manifolds, but they are all tangent to $E$ at the origin. We note
that, $\mathcal{W}(E)$, as defined in eq. (\ref{eq:W(E) construction}),
is only known to be $C^{1}$ under our current set of assumptions.
Under further nonresonance assumption on the spectrum of $\mathbf{A},$
the SSM can be shown to be as smooth as the underlying dynamical system
(see \citet{cabre03,haller16}).

Based on these results, we apply the linearizing transformation (\ref{eq:Hartman linearization})
followed by a linear change of coordinates to $\left(\boldsymbol{\xi},\boldsymbol{\eta}\right)$,
where $\boldsymbol{\xi}$ are coordinates along the spectral subspace
$E$ and $\boldsymbol{\eta}$ are coordinates along the spectral subspace
$F.$ Specifically, let 
\begin{align}
\mathbf{x} & =\mathbf{T}\mathbf{y}+\mathbf{h}(\mathbf{T}\mathbf{y}),\quad\mathbf{y}=\left(\boldsymbol{\xi},\boldsymbol{\eta}\right)^{\mathrm{T}}\in\mathbb{R}^{d}\times\mathbb{R}^{n-d},\label{eq:main transformation}
\end{align}
with the matrix $\mathbf{T}$ defined in formula (\ref{eq:Tdef}). 

Under the transformation (\ref{eq:main transformation}), the nonlinear
system \eqref{eq:more specific nonlinear system}becomes
\begin{equation}
\left(\begin{array}{c}
\dot{\boldsymbol{\xi}}\\
\dot{\boldsymbol{\eta}}
\end{array}\right)=\left(\begin{array}{cc}
\boldsymbol{\Lambda}_{E} & 0\\
0 & \boldsymbol{\Lambda}_{F}
\end{array}\right)\left(\begin{array}{c}
\boldsymbol{\xi}\\
\boldsymbol{\eta}
\end{array}\right),\qquad\boldsymbol{\Lambda}_{E}=\left(\mathbf{T}^{-1}\mathbf{A}\mathbf{T}\right)\vert_{E},\qquad\boldsymbol{\Lambda}_{F}=\left(\mathbf{T}^{-1}\mathbf{A}\mathbf{T}\right)\vert_{F}.\label{eq:decoupled linear system}
\end{equation}
In these coordinates, the smoothest invariant manifold tangent to
$E$ is 
\[
\mathcal{W}(E)=\left\{ \left(\boldsymbol{\xi},\boldsymbol{\eta}\right)\,:\,\boldsymbol{\eta}=0\right\} ,
\]
within which the dynamics restricted to $\mathcal{W}(E)$ are then
given by 
\begin{equation}
\dot{\boldsymbol{\xi}}=\boldsymbol{\Lambda}_{E}\boldsymbol{\xi}.\label{eq:xi_ODE}
\end{equation}
From eq. (\ref{eq:decoupled linear system}), we obtain the estimates
\begin{align*}
\left|\boldsymbol{\xi}\left(t;\boldsymbol{\xi}_{0}\right)\right| & \leq\left|\boldsymbol{\xi}_{0}\right|e^{\left(\mathrm{Re}\text{\ensuremath{\lambda_{1}+\epsilon_{1}}}\right)t},\\
\left|\boldsymbol{\eta}\left(t;\boldsymbol{\eta}_{0}\right)\right| & \leq\left|\boldsymbol{\eta}_{0}\right|e^{\left(\mathrm{Re}\text{\ensuremath{\lambda_{d+1}+\epsilon_{d+1}}}\right)t},
\end{align*}
where $\epsilon_{j}\geq0$ is an arbitrarily small constant that can
be chosen zero if the algebraic multiplicity of $\lambda_{j}$ is
equal to its geometric multiplicity. 

With the help of the $\left(\boldsymbol{\xi},\boldsymbol{\eta}\right)$
coordinates and eq. (\ref{eq:Hartman linearization}), the class $C^{2}$
observable $\boldsymbol{\phi}$ be locally be written as 
\begin{align}
\boldsymbol{\phi}(x) & =\boldsymbol{\phi}\left(\mathbf{T}\mathbf{y}+\mathbf{h}(\mathbf{T}\mathbf{y})\right)=D\boldsymbol{\phi}(\mathbf{0})\left[\mathbf{T}\mathbf{y}+\mathbf{h}(\mathbf{Ty})\right]+\mathcal{O}\left(\left|\mathbf{y}+\mathbf{h}(\mathbf{y})\right|^{2}\right)\nonumber \\
 & =D\boldsymbol{\phi}(\mathbf{0})\mathbf{T}\left(\begin{array}{c}
\boldsymbol{\xi}\\
\boldsymbol{\eta}
\end{array}\right)+D\boldsymbol{\phi}(\mathbf{0})\mathbf{h}(\mathbf{T}\mathbf{y})+\mathcal{O}\left(\left|\boldsymbol{\xi}\right|^{2},\left|\boldsymbol{\xi}\right|\left|\boldsymbol{\eta}\right|,\left|\boldsymbol{\eta}\right|^{2}\right)\nonumber \\
 & =D\boldsymbol{\phi}(\mathbf{0})\mathbf{T}_{E}\boldsymbol{\xi}+\mathcal{O}\left(\left|\boldsymbol{\eta}\right|\right)+\mathcal{O}\left(\left|\left(\boldsymbol{\xi},\boldsymbol{\eta}\right)\right|^{1+\beta}\right)+\mathcal{O}\left(\left|\boldsymbol{\xi}\right|^{2},\left|\boldsymbol{\xi}\right|\left|\boldsymbol{\eta}\right|,\left|\boldsymbol{\eta}\right|^{2}\right)\label{eq:g near origin}\\
 & =D\boldsymbol{\phi}(\mathbf{0})\mathbf{T}_{E}\boldsymbol{\xi}+\mathcal{O}\left(\left|\boldsymbol{\eta}\right|\right)+\mathcal{O}\left(\left|\left(\boldsymbol{\xi},\boldsymbol{\eta}\right)\right|^{1+\beta}\right)\nonumber 
\end{align}
for some $\beta\in\left(0,1\right)$.

We also express the data matrix $\boldsymbol{\Phi}$ defined in eq.
(\ref{eq:DMD data matrix definitions}) in terms of data matrices
with respect to the coordinates $\mathbf{y}=\left(\boldsymbol{\xi},\boldsymbol{\eta}\right)$
by letting

\begin{equation}
\boldsymbol{\Phi}=\boldsymbol{\phi}\left(\mathbf{TY}-\mathbf{h}(\mathbf{TY})\right),\qquad\mathbf{Y}=\left(\begin{array}{c}
\boldsymbol{\Xi}\\
\boldsymbol{H}
\end{array}\right),\label{eq:Xi and H}
\end{equation}
where the functions $\boldsymbol{\phi}$ and $\boldsymbol{h}$ are
applied to the matrices involved column by column. From eqs. (\ref{eq:g near origin})
and (\ref{eq:Xi and H}), we obtain that
\begin{align}
\Phi & =D\boldsymbol{\phi}(\mathbf{0})\mathbf{T}_{E}\boldsymbol{\Xi}+\mathcal{O}\left(\left|\boldsymbol{H}\right|\right)+\mathcal{O}\left(\left|\left(\boldsymbol{\Xi},\boldsymbol{H}\right)\right|^{1+\beta}\right),\nonumber \\
\hat{\Phi} & =D\boldsymbol{\phi}(\mathbf{0})\mathbf{T}_{E}e^{\boldsymbol{\Lambda}_{E}\Delta t}\boldsymbol{\Xi}+\mathcal{O}\left(\left|\boldsymbol{H}\right|\right)+\mathcal{O}\left(\left|\left(\boldsymbol{\Xi},\boldsymbol{H}\right)\right|^{1+\beta}\right),\label{eq:Phi and Phi hat}
\end{align}
where we have used the boundedness of the sampling time $\Delta t$
and the near-identity nature of the linearizing mapping (\ref{eq:main transformation}),
which implies
\begin{equation}
\mathcal{O}\left(\left|\boldsymbol{\Xi}\right|^{\beta}\right)=\mathcal{O}\left(\left|\mathbf{P}_{E}\mathbf{X}\right|^{\beta}\right),\qquad\mathcal{O}\left(\left|\boldsymbol{H}\right|^{\beta}\right)=\mathcal{O}\left(\left|\mathbf{P}_{F}\mathbf{X}\right|^{\beta}\right).\label{eq:orders are equal}
\end{equation}
 From the equations (\ref{eq:Phi and Phi hat}), we obtain
\begin{align}
\hat{\boldsymbol{\Phi}}\boldsymbol{\Phi}^{\mathrm{T}} & =\left[D\boldsymbol{\phi}(0)\mathbf{T}_{E}e^{\boldsymbol{\Lambda}_{E}\Delta t}\boldsymbol{\Xi}+\mathcal{O}\left(\left|\boldsymbol{H}\right|\right)+\mathcal{O}\left(\left|\left(\boldsymbol{\Xi},\boldsymbol{H}\right)\right|^{1+\beta}\right)\right]\left[D\boldsymbol{\phi}(0)\mathbf{T}_{E}\boldsymbol{\Xi}+\mathcal{O}\left(\left|\boldsymbol{H}\right|\right)+\mathcal{O}\left(\left|\left(\boldsymbol{\Xi},\boldsymbol{H}\right)\right|^{1+\beta}\right)\right]^{\mathrm{T}}\nonumber \\
 & =D\boldsymbol{\phi}(0)\mathbf{T}_{E}e^{\boldsymbol{\Lambda}_{E}\Delta t}\boldsymbol{\Xi}\boldsymbol{\Xi}^{\mathrm{T}}\mathbf{T}_{E}^{\mathrm{T}}\left[D\boldsymbol{\phi}(\mathbf{0})\right]^{\mathrm{T}}+\mathcal{O}\left(\left|\boldsymbol{H}\right|^{2},\left|\boldsymbol{H}\right|\left|\left(\boldsymbol{\Xi},\boldsymbol{H}\right)\right|^{1+\beta},\left|\left(\boldsymbol{\Xi},\boldsymbol{H}\right)\right|^{2+2\beta}\right),\nonumber \\
\left(\boldsymbol{\Phi}\boldsymbol{\Phi}^{\mathrm{T}}\right)^{\dagger} & =\left\{ \left[D\boldsymbol{\phi}(\mathbf{0})\mathbf{T}_{E}\boldsymbol{\Xi}+\mathcal{O}\left(\left|\boldsymbol{H}\right|\right)+\mathcal{O}\left(\left|\left(\boldsymbol{\Xi},\boldsymbol{H}\right)\right|^{1+\beta}\right)\right]\left[D\boldsymbol{\phi}(\mathbf{0})\mathbf{T}_{E}\boldsymbol{\Xi}+\mathcal{O}\left(\left|\boldsymbol{H}\right|\right)+\mathcal{O}\left(\left|\left(\boldsymbol{\Xi},\boldsymbol{H}\right)\right|^{1+\beta}\right)\right]^{\mathrm{T}}\right\} ^{\dagger}\nonumber \\
 & =\left[D\boldsymbol{\phi}(\mathbf{0})\mathbf{T}_{E}\boldsymbol{\Xi}\boldsymbol{\Xi}^{\mathrm{T}}\mathbf{T}_{E}^{\mathrm{T}}\left[D\boldsymbol{\phi}(\mathbf{0})\right]^{\mathrm{T}}+\mathcal{O}\left(\left|\boldsymbol{H}\right|^{2},\left|\boldsymbol{H}\right|\left|\left(\boldsymbol{\Xi},\boldsymbol{H}\right)\right|^{1+\beta},\left|\left(\boldsymbol{\Xi},\boldsymbol{H}\right)\right|^{2+2\beta}\right)\right]^{\dagger}.\label{eq:first step}
\end{align}

By the first assumption in (A4) of the theorem, $\boldsymbol{\Phi}\boldsymbol{\Phi}^{\mathrm{T}}$
is invertible, and hence $\left(\boldsymbol{\Phi}\boldsymbol{\Phi}^{\mathrm{T}}\right)^{\dagger}=\left(\boldsymbol{\Phi}\boldsymbol{\Phi}^{\mathrm{T}}\right)^{-1}$.
Therefore, by the second equation in (\ref{eq:first step}), for small
enough $\left|\boldsymbol{\Xi}\right|$ and $\left|\boldsymbol{H}\right|$,
the matrix $D\boldsymbol{\phi}(\mathbf{0})\mathbf{T}_{E}\boldsymbol{\Xi}\boldsymbol{\Xi}^{\mathrm{T}}\mathbf{T}_{E}^{\mathrm{T}}\left[D\boldsymbol{\phi}(\mathbf{0})\right]^{\mathrm{T}}$
is also invertible and 
\begin{align}
\left(\boldsymbol{\Phi}\boldsymbol{\Phi}^{\mathrm{T}}\right)^{\dagger} & =\left[D\boldsymbol{\phi}(\mathbf{0})\mathbf{T}_{E}\boldsymbol{\Xi}\boldsymbol{\Xi}^{\mathrm{T}}\mathbf{T}_{E}^{\mathrm{T}}\left[D\boldsymbol{\phi}(\mathbf{0})\right]^{\mathrm{T}}\right]^{-1}+\mathcal{O}\left(\left|\boldsymbol{H}\right|^{2},\left|\boldsymbol{H}\right|\left|\left(\boldsymbol{\Xi},\boldsymbol{H}\right)\right|^{1+\beta},\left|\left(\boldsymbol{\Xi},\boldsymbol{H}\right)\right|^{2+2\beta}\right).\label{eq:second step1}
\end{align}
 Then, by assumption (A3) of the theorem, we can write
\begin{equation}
\left(\boldsymbol{\Phi}\boldsymbol{\Phi}^{\mathrm{T}}\right)^{\dagger}=\left(\mathbf{T}_{E}^{\mathrm{T}}\left[D\boldsymbol{\phi}(\mathbf{0})\right]^{\mathrm{T}}\right)^{-1}\left(\boldsymbol{\Xi}\boldsymbol{\Xi}^{\mathrm{T}}\right)^{-1}\left(D\boldsymbol{\phi}(\mathbf{0})\mathbf{T}_{E}\right)^{-1}+\mathcal{O}\left(\left|\boldsymbol{H}\right|^{2},\left|\boldsymbol{H}\right|\left|\left(\boldsymbol{\Xi},\boldsymbol{H}\right)\right|^{1+\beta},\left|\left(\boldsymbol{\Xi},\boldsymbol{H}\right)\right|^{2+2\beta}\right).\label{eq:second step2}
\end{equation}

Substitution of formulas (\ref{eq:second step1})-(\ref{eq:second step2})
into eq. (\ref{eq:general solution of DMD problem}) gives
\begin{align*}
\boldsymbol{\mathcal{D}} & =\hat{\boldsymbol{\Phi}}\boldsymbol{\Phi}^{\mathrm{T}}\left(\boldsymbol{\Phi}\boldsymbol{\Phi}^{\mathrm{T}}\right)^{\dagger}=\hat{\boldsymbol{\Phi}}\boldsymbol{\Phi}^{\mathrm{T}}\left(\boldsymbol{\Phi}\boldsymbol{\Phi}^{\mathrm{T}}\right)^{-1}\\
 & =\left[D\boldsymbol{\phi}(0)\mathbf{T}_{E}e^{\boldsymbol{\Lambda}_{E}\Delta t}\boldsymbol{\Xi}\boldsymbol{\Xi}^{\mathrm{T}}\mathbf{T}_{E}^{\mathrm{T}}\left[D\boldsymbol{\phi}(\mathbf{0})\right]^{\mathrm{T}}+\mathcal{O}\left(\left|\boldsymbol{H}\right|^{2},\left|\boldsymbol{H}\right|\left|\left(\boldsymbol{\Xi},\boldsymbol{H}\right)\right|^{1+\beta},\left|\left(\boldsymbol{\Xi},\boldsymbol{H}\right)\right|^{2+2\beta}\right)\right]\\
 & \times\left[\left(\mathbf{T}_{E}^{\mathrm{T}}\left[D\boldsymbol{\phi}(\mathbf{0})\right]^{\mathrm{T}}\right)^{-1}\left(\boldsymbol{\Xi}\boldsymbol{\Xi}^{\mathrm{T}}\right)^{-1}\left(D\boldsymbol{\phi}(\mathbf{0})\mathbf{T}_{E}\right)^{-1}+\mathcal{O}\left(\left|\boldsymbol{H}\right|^{2},\left|\boldsymbol{H}\right|\left|\left(\boldsymbol{\Xi},\boldsymbol{H}\right)\right|^{1+\beta},\left|\left(\boldsymbol{\Xi},\boldsymbol{H}\right)\right|^{2+2\beta}\right)\right]\\
 & =D\boldsymbol{\phi}(0)\mathbf{T}_{E}e^{\boldsymbol{\Lambda}_{E}\Delta t}\left(D\boldsymbol{\phi}(\mathbf{0})\mathbf{T}_{E}\right)^{-1}+\mathcal{O}\left(\left|\boldsymbol{\Xi}\right|^{-2}\left|\boldsymbol{H}\right|^{2},\left|\boldsymbol{\Xi}\right|^{-2}\left|\boldsymbol{H}\right|\left|\left(\boldsymbol{\Xi},\boldsymbol{H}\right)\right|^{1+\beta},\left|\boldsymbol{\Xi}\right|^{-2}\left|\left(\boldsymbol{\Xi},\boldsymbol{H}\right)\right|^{2+2\beta}\right)\\
 & =D\boldsymbol{\phi}(0)\mathbf{T}_{E}e^{\boldsymbol{\Lambda}_{E}\Delta t}\left(D\boldsymbol{\phi}(\mathbf{0})\mathbf{T}_{E}\right)^{-1}+\mathcal{O}\left(\left|\boldsymbol{\Xi}\right|^{\beta}\right),
\end{align*}
where we have use the second assumption in (A4) of the Theorem. This
completes the proof of Theorem \ref{thm: DMD}, given the order-of-magnitude
relations (\ref{eq:orders are equal}).
\begin{rem}
\label{rem:oscillatory modes}In systems where the slowest decaying
modes are oscillatory, the slow spectral subspace $E$ is always even
dimensional, as it is spanned by the real and imaginary parts of the
generalized eigenvectors of $\mathbf{A}$ that correspond to complex
conjugate pairs of eigenvalues. For such predominantly oscillatory
systems, therefore, the number $d$ of observables used in DMD has
to be an even number for assumption (A3) to hold.
\end{rem}
\begin{rem}
\label{rem:higher smoothness under nonresonance}In most applications
the dynamical system will have higher degree of smoothness, i.e.,
we will have $\mathbf{f}\in C^{r}$ for some \emph{$r\in\mathbb{N}^{+}\cup\left\{ \infty,a\right\} $,
}with $C^{a}$ referring to the space of analytic functions. Then,
under further nonresonance conditions of the spectrum of $\mathbf{A}$,
the spectral submanifold $\mathcal{W}(E)$ will also be of class $C^{r}$,
as we discuss in Section \ref{sec:DDL}. In those cases, for the approximate
topological equivalence holds in the statement of Theorem \ref{thm: DMD}
hold on the full domain of attraction of the $\mathbf{\mathbf{x}}=\mathbf{0}$
fixed point within $\mathcal{W}(E)$, as one deduces from the linearization-in-the-large
results of \citet{lan13} and \citet{kvalheim21}.\footnote{Note that these global linearization results rely on the existence
of a $C^{1}$ local linearization, which generally only exists for
class $C^{2}$ dynamical systems. Consequently, one cannot use the
available global linearization results within $\mathcal{W}(E)$ under
the general assumptions of Theorem \ref{thm: DMD}, which only guarantee
$C^{1}$ differentiability for the reduced flow in $\mathcal{W}(E)$.} In Theorem \ref{thm: DMD}, our objective was to give a minimal set
of conditions under which DMD can be justified as an approximate,
leading-order, $d$-dimensional reduced model for the original nonlinear
system. These minimal conditions only require hyperbolicity, i.e.,
robustness of the spectrum under perturbations, without insisting
on the lack of resonances in the spectrum. This is advantageous in
data-driven applications in which details of the spectrum are generally
not known.
\end{rem}
\begin{rem}
\label{rem:smoothness in infinite dimensions}As in the case of finite-dimensional
systems, linearization theorems guaranteeing higher-degree of smoothness
for the linearizing transformations are also available in infinite
dimensions, but these require various non-resonance conditions on
the spectrum of $\mathbf{A}$ (see, e.g., \citet{elbialy01}). In
Theorem \ref{thm:DMD-infinite-dim}, as in the finite-dimensional
case, our objective was to present a minimal set of assumptions under
which DMD can be justified without the detailed knowledge of the spectrum
of $\mathbf{A}$. 
\end{rem}
\begin{rem}
\label{rem:mixed mode SSM in infinite dimensions}We note that the
more general Theorem 1.5 of \citet{newhouse17} also shows $C^{1,\alpha}$
linearizability for a class of systems that have so called \emph{$\alpha-$hyperbolic}
(as opposed to stable hyperbolic) fixed points. This would enable
us to waive the requirement in Theorem \ref{thm:DMD-infinite-dim}
that $\mathbf{A}$ is a contraction and also allow for expanding directions
within the spectral subspace $E$, as we did in our remarks after
Theorems \ref{thm: DMD} and \ref{thm:DMD-maps}. The defining properties
of these $\alpha$-hyperbolic systems are, however, complicated to
verify and do not hold for general hyperbolic fixed points even in
finite-dimensional Banach spaces. For instance, the simple three-dimensional
example of \citet{hartman60} (eq. (8) in that reference) of a non-$C^{1}$
linearizable system has a hyperbolic fixed point that is not $\alpha$-hyperbolic
in the sense of \citet{newhouse17}. 
\end{rem}

\section{Necessity of assumptions (A2)-(A4) of Theorem \ref{thm: DMD}\label{sec:Necessity-of-assumptions}}

Assumption (A2) of Theorem \ref{thm: DMD} requires $C^{2}$ smoothness
for the dynamical system \eqref{eq:more specific nonlinear system}.
We now show on an example that if only $C^{1}$ smoothness holds for
the dynamical system, then the DMD approximation can be less accurate.
Consider the 1D system 
\begin{equation}
\dot{x}=-x+x^{1+\alpha}.\label{eq:smooth_nonsmooth}
\end{equation}

The origin $x=0$ is an asymptotically stable fixed point. For $\alpha>0$,
the system is $C^{1}$ at the origin, but the second derivative $\frac{\partial^{2}\dot{x}}{\partial x^{2}}=\alpha^{2}x^{\alpha-1}$
is singular at $x=0$ for $\alpha<1$. We compare the accuracy of
a simple DMD prediction for $\alpha=1$, which corresponds to $C^{2}$
smoothness and for $\alpha=0.9$, which corresponds to only $C^{1}$-smoothness
at $x=0$. The DMD prediction errors can be seen in Fig. \ref{fig:assumption1}
as a function of the initial distance of the trajectory from the origin.
Although the $C^{1}$-smooth system is close to the $C^{2}$system,
a large difference can be observed in the prediction accuracy of DMD
asymptotically as $x\to0$. 

\begin{figure}
\centering{}\includegraphics[width=0.5\textwidth]{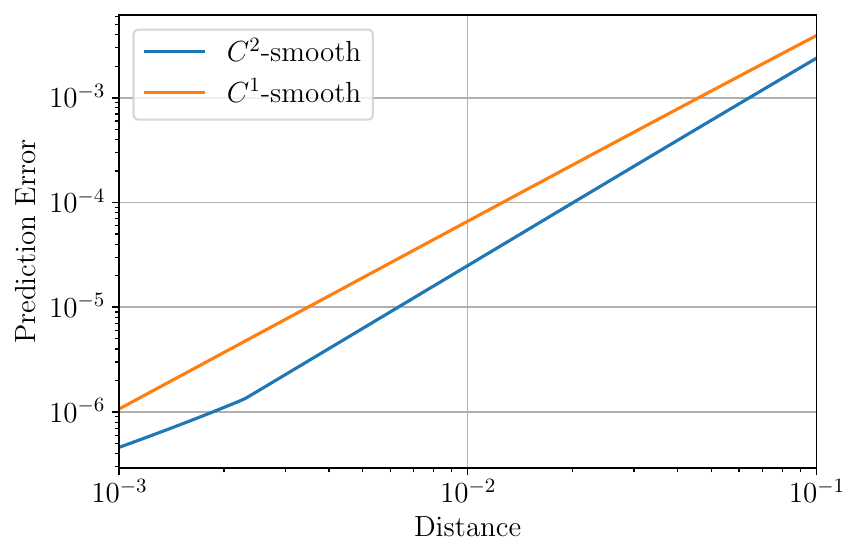}\caption{Prediction accuracy of DMD on $C^{1}-$ and $C^{2}-$smooth systems
\eqref{eq:smooth_nonsmooth}. The orange curve shows the accuracy
of DMD for $\alpha=0.9$, while the blue curve corresponds to $\alpha=1$.
\label{fig:assumption1}}
\end{figure}

Assumption (A3) of Theorem \ref{thm: DMD} requires a specific nondegeneracy
condition to hold for the observable in order for DMD to give a meaningful
approximation. We now illustrate on an example what happens if this
nondegeneracy condition is not satisfied. Consider the chain of nonlinear
oscillators in Section \ref{subsec:Oscillator-chain}. Let us denote
the normal modes of the mechanical system \eqref{eq:oscillator chain second order}
linearized around the stable fixed point $\mathbf{q=0},$ as $\hat{\mathbf{q}}_{j}(t)=\hat{\mathbf{q}}_{j}(0)e^{\lambda_{j}t}$,
where the mode shapes $\hat{\mathbf{q}}_{j}(0)\in\mathbb{R}^{5}$
satisfy
\[
\left(\lambda_{j}^{2}\mathbf{M}+\lambda_{j}\mathbf{C}+\mathbf{K}\right)\hat{\mathbf{q}}_{j}(0)=0.
\]

Therefore, in the phase space spanned by the variable
\[
\mathbf{x}=\left(q_{1},q_{2},...,q_{5},\dot{q_{1}},\dots,\dot{q}_{5}\right)\in\mathbb{R}^{10},
\]
 the eigenvectors of the linear part of \eqref{eq:oscillator chain second order}are
given as 
\[
\mathbf{e}_{j}=\left(\hat{\mathbf{q}}_{j}(0),\lambda_{j}\hat{\mathbf{q}}_{j}(0)\right)^{T}.
\]

With our choice of parameters, the slowest pair of eigenvalues are
\[
\lambda_{1,2}=-0.0012\pm i0.2825,
\]

and hence there exists a 2D slow SSM tangent to the slow spectral
subspace 
\[
E=\text{span}\left([\text{Re}(\mathbf{e}_{1}),\text{Im}(\mathbf{e}_{1})]\right)=\text{span}\left(\left[\hat{\mathbf{q}}_{1}(0),\text{Re}\left(\lambda_{1}\right)\hat{\mathbf{q}}_{1}(0)\right],\left[0,\text{Im}\left(\lambda_{1}\right)\hat{\mathbf{q}}_{1}(0)\right]\right),
\]
where $\mathbf{e}_{1}$ is the eigenvector corresponding to $\lambda_{1}$. 

As the dimension of $\mathcal{W}\left(E\right)$ is $d=2$, we consider
three possible 2D observables with respect to which Theorem \ref{thm: DMD}
could potentially be applied:

\[
\boldsymbol{\phi}_{1}(\mathbf{x})=\left(\begin{array}{c}
\langle\text{Re}(\mathbf{e}_{1}),\mathbf{x}\rangle\\
\langle\text{Im}(\mathbf{e}_{1}),\mathbf{x}\rangle
\end{array}\right)\quad\boldsymbol{\phi}_{2}(\mathbf{x})=\left(\begin{array}{c}
q_{1}\\
\dot{q}_{1}
\end{array}\right)\quad\boldsymbol{\phi}_{3}(\mathbf{x})=\left(\begin{array}{c}
q_{1}\\
q_{2}
\end{array}\right).
\]

The Jacobians of these observables with respect to $x$ can be computed
as
\begin{align*}
D\boldsymbol{\phi}_{1}(\mathbf{0}) & =\left(\text{Re}(\mathbf{e}_{1})\,\text{Im}(\mathbf{e}_{1})\right)^{\mathrm{T}}=\left(\begin{array}{ccccc}
* & * & * & \cdots & *\\
0 & 0 & 0 & \cdots & *
\end{array}\right),\\
D\boldsymbol{\phi}_{2}(\mathbf{0}) & =\left(\begin{array}{ccccc}
1 & 0 & 0 & \cdots & 0\\
0 & 1 & 0 & \cdots & 0
\end{array}\right),\\
D\boldsymbol{\phi}_{3}(\mathbf{0}) & =\left(\begin{array}{ccccc}
1 & 0 & 0 & \cdots & 0\\
0 & 0 & 1 & \cdots & 0
\end{array}\right),
\end{align*}
where $*$ denotes nonzero entries. Therefore, we require\emph{
\begin{equation}
\mathrm{rank\,}\left[D\boldsymbol{\phi}_{1,2,3}\left(\mathbf{0}\right)\vert_{E}\right]=2.\label{eq:nondegeneracy of observables for DMD1-2}
\end{equation}
}The restriction of the Jacobians of the observables to $E$ is $D\boldsymbol{\phi}_{i}(\mathbf{0})[\text{Re}(\mathbf{e}_{1}),\text{Im}(\mathbf{e}_{1})]$.
It can be seen from the structure of the matrices, that for $\boldsymbol{\phi}_{1}$
and $\boldsymbol{\phi}_{2}$, the rank of the restricted Jacobian
is 2, but for $\boldsymbol{\phi}_{3}$, it is only of rank 1. Therefore,
the observable vector $\boldsymbol{\phi}_{3}$ cannot be used to parametrize
the SSM and to derive a reduced-order model. 

Figures \ref{fig:osc_chain_ssm}a-b show that observing the dynamics
via $\boldsymbol{\phi}_{1}(x)$ and $\boldsymbol{\phi}_{2}(x)$ indeed
results in a successful embedding of the dynamics on $\mathcal{W}\left(E\right)$
near the fixed point. At the same time, Fig. \ref{fig:osc_chain_ssm}c
shows that using the observable $\boldsymbol{\phi}_{3}(x)$ does not
correctly reproduce the dynamics near the origin.

\begin{figure}
\centering{}\includegraphics[width=0.99\textwidth]{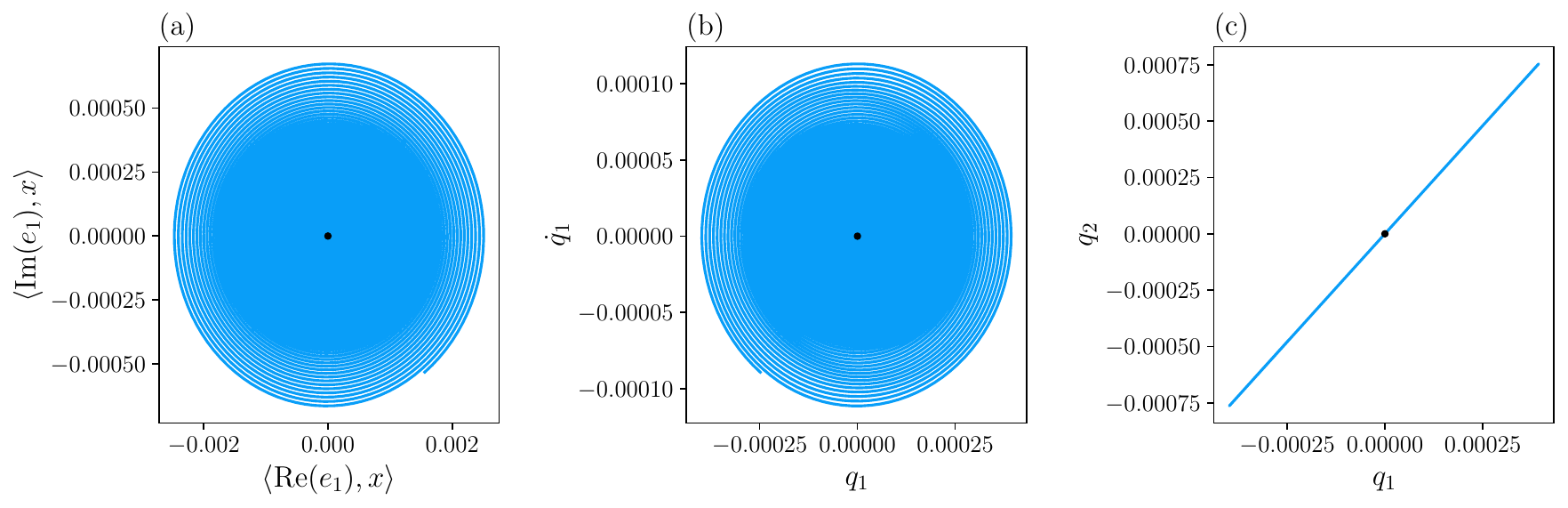}\caption{The dynamics near the $\mathbf{x}=\mathbf{0}$ point on the slow 2D
SSM $\mathcal{W}\left(E\right)$ of system \eqref{eq:oscillator chain second order}
represented via different observables (a) Observable $\boldsymbol{\phi}_{1}(\mathbf{x})$
(b) Observable $\boldsymbol{\phi}_{2}(\mathbf{x})$ (c) Observable
$\boldsymbol{\phi}_{3}(\mathbf{x})$.\label{fig:osc_chain_ssm}}
\end{figure}

Assumption (A4) of Theorem \ref{thm: DMD} states that the collected
data has to be rich enough to a neighborhood of the SSM near the fixed
point. To illustrate this, we modify the linear system in \eqref{eq:3d_lin_system1}
to have three real eigenvalues by letting $\boldsymbol{\Lambda}=\text{diag}\left(a,b,c\right)\in\mathbb{R}^{3}$.
We then collect the resulting trajectories in the phase space in the
observable data matrix 
\[
\boldsymbol{\Phi}=\left(\begin{array}{cccc}
x_{1}(0) & ax_{1}(0) &  & a^{n}x_{1}(0)\\
x_{2}(0) & bx_{2}(0) & \cdots & b^{n}x_{2}(0)\\
x_{3}(0) & cx_{3}(0) &  & c^{n}x_{3}(0)
\end{array}\right).
\]

Assuming $c<b<a<0$, we seek to identify from the data matrix $\Phi$
the 2D slow eigenspace $\mathcal{W}\left(E\right)=E$ corresponding
to the eigenvalues $a$ and $b$. Assume now that $x_{2}(0)=x_{3}(0)=0$
holds. In this case, the second and third rows of $\Phi$ are identically
zero and hence its rank is only 1, which violates assumption (A4)
of Theorem \ref{thm: DMD}. Indeed, the trajectory observed in $\boldsymbol{\Phi}$
is restricted to the slowest 1D eigenspace $\mathcal{W}\left(E_{1}\right)=E_{1}$
corresponding to the eigenvalue $a$ and hence cannot be used to approximate
the shape and internal dynamics of $\mathcal{W}\left(E\right)$. 

In practice, a generic trajectory of a nonlinear system would not
be restricted to a lower-dimensional invariant subspace and hence
the data matrix $\Phi$ would have full rank with probability one.
However, if the available data is close to a lower-dimensional invariant
subspace, the condition number of the data matrix may become prohibitively
large. So a near-failure of assumption (A4) is already detrimental
to the quality of the results obtained from DMD.

\section{Proof of Theorem \ref{thm:DDL}\label{sec:Proof-of-Theorem-DDL}}

For $\mathbf{f}_{2}\in C^{r}$ and under the nonresonance conditions
\begin{equation}
\lambda_{k}\neq\sum_{j=1}^{n}m_{j}\lambda_{j},\qquad j,k=1,\ldots,n,\qquad\sum_{j=1}^{n}m_{j}\geq2,\label{eq:nonresonance condition2-2-1}
\end{equation}
the $C^{\infty}$ linearization theorems of \citet{sternberg57} and
\citet{poincare1892} guarantee the local existence of a near-identity,
linearizing change of coordinates of the form 
\begin{equation}
\mathbf{x}=\mathbf{y}+\mathbf{h}(\mathbf{y}),\qquad\mathbf{y}\in U\subset\mathbb{R}^{n},\qquad\mathbf{h}(\mathbf{y}),\mathbf{h}^{-1}(\mathbf{y})=\mathcal{O}\left(\left|\mathbf{y}\right|^{2}\right),\label{eq:near-identity linearizing transformation}
\end{equation}
with $\mathbf{h}\in C^{r}$. Additionally, a strengthening of Sternberg's
linearization theorem by \citet{kvalheim21}\footnote{This generalizes the results of \citet{lan13} in the setting of the
$C^{1}$ linearization theorem by \citet{hartman60}.} implies that the transformation in eq. (\ref{eq:near-identity linearizing transformation})
is unique and globally defined on the full domain of attraction of
the fixed point at the origin. As pointed out by \citet{haller23},
this implies the existence of a unique $d$-dimensional, normally
attracting spectral submanifold $\mathcal{W}(E)\in C^{\infty}$ for
system \eqref{eq:more specific nonlinear system} whose internal dynamics
govern the longer-term behavior of all trajectories near the origin.\footnote{\citet{cabre03} guarantees the existence and uniqueness of $\mathcal{W}(E)\in C^{\infty}$
under weaker nonresonance conditions that allow resonances inside
$E$ but no resonances of the form (\ref{eq:nonresonance condition2-2-1})
with $\lambda_{k}\in\mathrm{Spect}\,\left(A\vert_{F}\right)$ and
$\lambda_{j}\in\mathrm{Spect}\,\left(\mathbf{A}\vert_{E}\right).$
We do not rely on these results here, as the next step in our construction
is local linearization within $\mathcal{W}(E)$, which cannot be carried
out in the presence of resonances inside $E$. } 

By the main assumption of the theorem, we have $\mathrm{Re}\text{\ensuremath{\lambda_{1}<0}}$.
Under this assumption, the infinitely many nonresonance conditions
in (\ref{eq:nonresonance condition2-2-1}) simplify to finitely many.
Indeed, if we define the spectral quotient as the positive integer
$Q$ given by
\[
Q=\left\lfloor \frac{\max_{i}\left|\mathrm{Re}\,\lambda_{i}\right|}{\min_{i}\left|\mathrm{Re}\,\lambda_{i}\right|}\right\rfloor +1,
\]
 Then, for all $m_{j}\in\mathbb{N}$ with $\sum_{j=1}^{n}m_{j}>Q$,
we have the estimate
\begin{align*}
\sum_{j=1}^{n}m_{j}\mathrm{Re}\,\lambda_{j} & \leq-\sum_{j=1}^{n}m_{j}\min_{i}\left|\mathrm{Re}\,\lambda_{i}\right|<-\sum_{j=1}^{n}m_{j}\frac{\max_{i}\left|\mathrm{Re}\,\lambda_{i}\right|}{Q}=-\max_{i}\left|\mathrm{Re}\,\lambda_{i}\right|\frac{\sum_{j=1}^{n}m_{j}}{Q}\\
 & <-\max_{i}\left|\mathrm{Re}\,\lambda_{i}\right|,
\end{align*}
implying 
\begin{equation}
\sum_{j=1}^{n}m_{j}\mathrm{Re}\,\lambda_{j}<\mathrm{Re}\,\lambda_{k},\qquad k=1,\ldots,n,\quad\sum_{j=1}^{n}m_{j}>Q.\label{eq:nonresonance in W(E)-0}
\end{equation}
Consequently, in this case, no integer resonance among the eigenvalues
of $\mathbf{A}$ can occur for $\sum_{j=1}^{n}m_{j}>Q$ and hence
the nonresonance condition (\ref{eq:nonresonance condition2-2-1})
can be relaxed to assumption (A1) of Theorem \ref{thm:DDL}. 

Based on these results, we apply the $C^{r}$ linearizing transformation
(\ref{eq:near-identity linearizing transformation}) followed by a
linear change of coordinates to $\left(\boldsymbol{\xi},\boldsymbol{\eta}\right)$uses
in the proof of Theorem \ref{thm: DMD} (see eq. (\ref{eq:main transformation})).
As se have seen there, this coordinate change maps the spectral submanifold
$\mathcal{W}(E)$ into the $\boldsymbol{\eta}=\mathbf{0}$ plane.
On $\mathcal{W}(E)$, the reduced dynamics obeys the linear ODE 
\begin{equation}
\dot{\boldsymbol{\xi}}=\boldsymbol{\Lambda}_{E}\boldsymbol{\xi},\qquad\boldsymbol{\Lambda}_{E}=\left(\mathbf{T}^{-1}\mathbf{A}\mathbf{T}\right)\vert_{E},\label{eq:linearized dynamics on W(E)}
\end{equation}
where $\boldsymbol{\Lambda}_{E}$, the real Jordan canonical form
of $\mathbf{A}\vert_{E},$ has eigenvalues $\lambda_{1},\ldots,\lambda_{d}$. 

Restricted to $\mathcal{W}(E),$ we can express the $d$-dimensional
observable vector $\boldsymbol{\phi}(\mathbf{x})$ as 
\begin{equation}
\boldsymbol{\varphi}(\boldsymbol{\xi})=\boldsymbol{\phi}\left(\mathbf{T}\left(\begin{array}{c}
\boldsymbol{\xi}\\
\mathbf{0}
\end{array}\right)+\mathbf{h}\left(\begin{array}{c}
\boldsymbol{\xi}\\
\mathbf{0}
\end{array}\right)\right)\in\mathbb{R}^{d}.\label{eq:restricted observable}
\end{equation}
On the manifold $\mathcal{W}(E)$, we can use $\boldsymbol{\varphi}$
locally near the origin as a new variable if the coordinate change
from $\boldsymbol{\xi}$ to $\boldsymbol{\varphi}$ is a $C^{r}$
diffeomorphism on $\mathcal{W}(E)$, i.e., $D_{\boldsymbol{\xi}}\boldsymbol{\varphi}$
is nonsingular at the origin. To see if this is the case, note that
\[
D_{\boldsymbol{\xi}}\boldsymbol{\psi}(\boldsymbol{\xi})=D\boldsymbol{\phi}\left(\mathbf{T}\left(\begin{array}{c}
\boldsymbol{\xi}\\
\mathbf{0}
\end{array}\right)\right)\mathbf{T}\left(\begin{array}{c}
\mathbf{I}_{d\times d}\\
\mathbf{0}
\end{array}\right),
\]
which gives
\[
D_{\boldsymbol{\xi}}\boldsymbol{\gamma}(\mathbf{0})=D\boldsymbol{\phi}\left(\mathbf{0}\right)\mathbf{T}\left(\begin{array}{c}
\mathbf{I}_{d\times d}\\
\mathbf{0}
\end{array}\right)=D\boldsymbol{\phi}\left(\mathbf{0}\right)\mathbf{T}_{E},
\]
with the operator $\mathbf{T}_{E}$ defined in (\ref{eq:Tdef}). By
assumption (\ref{eq:nondegeneracy of observables for DMD-1-2}) of
the theorem, we have 
\begin{equation}
\det\,\left[D\boldsymbol{\phi}\left(\mathbf{0}\right)\mathbf{T}_{E}\right]\neq0,\label{eq:nondegeneracy for implicit function theorem}
\end{equation}
and hence, by the inverse function theorem, $\boldsymbol{\xi}$ can
be expressed from the relationship (\ref{eq:restricted observable})
near the origin as a smooth function of $\boldsymbol{\varphi}$. Specifically,
there exists a $C^{r}$ diffeomorphism $\boldsymbol{\psi}\colon\mathbb{R}^{d}\to\mathbb{R}^{d}$
such that 
\begin{equation}
\boldsymbol{\xi}=\boldsymbol{\psi}(\boldsymbol{\varphi}),\qquad\boldsymbol{\psi}(\boldsymbol{0})=\boldsymbol{0},\quad D_{\boldsymbol{\varphi}}\boldsymbol{\psi}(\boldsymbol{\varphi})=\left[D_{\boldsymbol{\xi}}\boldsymbol{\varphi}(\boldsymbol{\xi})\right]^{-1}.\label{eq:G_def}
\end{equation}
Differentiating the first equation in (\ref{eq:G_def}), substituting
eq. (\ref{eq:linearized dynamics on W(E)}), using the relations (\ref{eq:G_def})
and Taylor expanding the result around $\boldsymbol{\varphi}=\mathbf{0}$,
we obtain the reduced dynamics on $\mathcal{W}(E)$ in the form
\begin{align}
\dot{\boldsymbol{\varphi}} & =\left[D_{\boldsymbol{\varphi}}\boldsymbol{\psi}(\boldsymbol{\varphi})\right]^{-1}\boldsymbol{\Lambda}_{E}\boldsymbol{\psi}(\boldsymbol{\varphi})\nonumber \\
 & =D_{\boldsymbol{\xi}}\boldsymbol{\varphi}(\mathbf{0})\boldsymbol{\Lambda}_{E}\left[D_{\boldsymbol{\xi}}\boldsymbol{\varphi}(\mathbf{0})\right]^{-1}\boldsymbol{\varphi}+\mathbf{q}\left(\boldsymbol{\varphi}\right),\qquad\mathbf{q}\left(\boldsymbol{\varphi}\right)=o\left(\left|\boldsymbol{\varphi}\right|\right)\label{eq:DMD nonlinear system}
\end{align}
 for some $C^{r}$ function $\mathbf{q}\colon U\subset\mathbb{R}^{d}\to\mathbb{R}^{d}$
defined in a neighborhood of $\boldsymbol{\varphi}=\mathbf{0}$. This
proves statement (i) of the theorem.

We have seen that $D_{\boldsymbol{\xi}}\boldsymbol{\varphi}(\mathbf{0})$
is nonsingular and hence the linear part of the nonlinear ODE in (\ref{eq:DMD nonlinear system})
has the same eigenvalues as $\boldsymbol{\Lambda}_{E}$ does.  $\left(\mathbf{T}^{-1}\mathbf{A}\mathbf{T}\right)\vert_{E}$.
Consequently, the origin of system (\ref{eq:linearized observer dynamics})
is a hyperbolic fixed point. If the nonresonance conditions (\ref{eq:nonresonance condition2-2})
hold for the full spectrum of $\mathbf{A}$, then they also hold for
any subset of this spectrum, i.e., for the eigenvalues of $\boldsymbol{\Lambda}_{E}$
as well. Consequently, from the linearization theorems of \citet{sternberg57}
and \citet{kvalheim21}, we can also deduce a similar near-identity,
$C^{r}$ linearizing diffeomorphism 
\[
\boldsymbol{\varphi}=\boldsymbol{\gamma}+\boldsymbol{\ell}(\boldsymbol{\gamma}),\qquad\boldsymbol{\ell}(\boldsymbol{\gamma})=o\left(\left|\boldsymbol{\gamma}\right|\right),\quad\boldsymbol{\ell}\in C^{r},
\]
for \emph{$r\in\mathbb{N}^{+}\cup\left\{ \infty,a\right\} $} under
the nonresonance assumption (\ref{eq:nonresonance condition2-2}).
In the new $\boldsymbol{\gamma}$ coordinates, the dynamics on the
SSM $\mathcal{W}\left(E\right)$ is governed by the linear ODE
\begin{equation}
\dot{\boldsymbol{\gamma}}=D_{\boldsymbol{\xi}}\boldsymbol{\gamma}(\mathbf{0})\boldsymbol{\Lambda}_{E}\left[D_{\boldsymbol{\xi}}\boldsymbol{\gamma}(\mathbf{0})\right]^{-1}\boldsymbol{\gamma},\label{eq:linearized observer dynamics}
\end{equation}
 which proves statement (ii) of the theorem. 

The partial differential equation (\ref{eq:invariance PDE}) can then
be obtained by differentiating eq. (\ref{eq:linearizing transformation on W(E)})
and substituting eqs. (\ref{eq:gamma ODE}) and (\ref{eq:linearized flow on SSM})
into the resulting equation. The expansions (\ref{eq:DDL transformation formal})
and (\ref{eq:DDL transformation convergent}) then follow from the
guaranteed smoothness class of the linearizing transformation for
general $r$ and for $r=a$, respectively. This concludes the proof
of statement (iii) of the theorem.

\section{Reduced dynamics on periodically forced SSMs \label{sec:Proof-of-Theorem-DDL-forced}}

Let us denote the time-$T$ map for the periodically forced system
by $\mathbf{P}_{\varepsilon}^{T}$. Since $\mathbf{P}_{0}^{T}$ represents
the period-$T$ sampling of the flow map of the unperturbed system,
it has a fixed point at $\mathbf{x}=\mathbf{0}$. By the implicit
function theorem, $\mathbf{P}_{\varepsilon}^{T}$ also has a fixed
point $O(\varepsilon)-$close to $\mathbf{x}=\mathbf{0}$. This corresponds
to a periodic orbit of \eqref{eq:forced system} close to the origin.
As explained by \citet{haller23}, the $C^{\infty}$ linearization
results of \citet{sternberg57} and \citet{poincare1892} also establish
the smoothness of the linearizations in $\varepsilon.$ Applied to
the discrete dynamical system defined by the Poincaré-map, these results
guarantee the existence of a linearizing transformation of the form

\[
\mathbf{x}=\mathbf{y}+\mathbf{h}^{\varepsilon}(\mathbf{y}),\qquad\mathbf{y}\in U\subset\mathbb{R}^{n},\qquad\mathbf{h}(\mathbf{y}),\mathbf{h}^{-1}(\mathbf{y})=\mathcal{O}\left(\left|\mathbf{y}\right|^{2},\varepsilon|\mathbf{y}|^{2}\right),
\]
 with $\mathbf{h}\in C^{r}$ in both $\mathbf{y}$ and $\varepsilon$. 

We now apply the $\varepsilon$-dependent linearizing transformation
followed by the linear change of coordinates to the coordinates $\left(\boldsymbol{\xi},\boldsymbol{\eta}\right)$
introduced in the proof of Theorem \ref{thm:DDL}. Specifically, let
\begin{equation}
\mathbf{x}=\mathbf{Ty}+\mathbf{h}^{\varepsilon}(\mathbf{Ty}),\quad\mathbf{y}=\left(\boldsymbol{\xi},\boldsymbol{\eta}\right)^{\mathrm{T}}\in\mathbb{R}^{d}\times\mathbb{R}^{n-d},\label{eq:transformation_epsilon}
\end{equation}
with the matrix $\mathbf{T}$ defined in formula (\ref{eq:Tdef}),
i.e., containing eigenvalues of $\mathbf{A}$. As the change of coordinates
(\ref{eq:transformation_epsilon}) is smooth in $\varepsilon$, it
can be written as $\mathbf{x}=\mathbf{T}\mathbf{y}+\mathbf{h}^{0}(\mathbf{T}\mathbf{y})+O(|y|^{2}\varepsilon).$
Differentiation of \eqref{eq:transformation_epsilon} with respect
to time yields
\[
\dot{\mathbf{x}}=\left(\mathbf{T}+D\mathbf{h}^{\varepsilon}(\mathbf{T}\mathbf{y})\mathbf{T}\right)\dot{\mathbf{y}}=\mathbf{A}(\mathbf{Ty}+\mathbf{h}^{\varepsilon}(\mathbf{Ty}))+\tilde{\mathbf{f}}(\mathbf{Ty}+\mathbf{h}^{\varepsilon}(\mathbf{Ty}))+\varepsilon\mathbf{F}(\mathbf{x}+h^{\varepsilon}(\mathbf{Ty}),t),
\]
which can be rearranged as
\begin{align}
\dot{\mathbf{y}} & =\left(\mathbf{T}+D\mathbf{h}^{0}(\mathbf{Ty})\mathbf{T}+O(|\mathbf{y}|\varepsilon)\right){}^{-1}\left[\mathbf{A}\left(\mathbf{T}\mathbf{y}+\mathbf{h}^{0}(\mathbf{Ty})+O(|\mathbf{y}|^{2}\varepsilon)\right)\right]\label{eq:epsilon_linearization_transformed}\\
 & +\left(\mathbf{T}+D\mathbf{h}^{0}(\mathbf{Ty})\mathbf{T}+O(|\mathbf{y}|\varepsilon)\right){}^{-1}\left[\tilde{\mathbf{f}}(\mathbf{Ty}+\mathbf{h}^{0}(\mathbf{Ty})+O(|\mathbf{y}|^{2}\varepsilon))+\varepsilon\mathbf{F}(\mathbf{Ty}+\mathbf{h}^{0}(\mathbf{Ty}),t)\right].\nonumber 
\end{align}
Since $y+h^{0}(y)$ linearizes \eqref{eq:forced system} for $\varepsilon=0$,
\eqref{eq:epsilon_linearization_transformed} simplifies to
\begin{equation}
\dot{\mathbf{y}}=\boldsymbol{\Lambda}\mathbf{y}+\left(\mathbf{T}+D\mathbf{h}^{0}(\mathbf{Ty})\mathbf{T}\right){}^{-1}\varepsilon\mathbf{F}(\mathbf{0},t)+O(|\mathbf{y}|^{2}\varepsilon,\mathbf{y}\varepsilon^{2}),\quad\mathbf{y}=(\boldsymbol{\xi},\boldsymbol{\eta})^{\mathrm{T}}.\label{eq:linearized dynamics with epsilon}
\end{equation}
The smoothest SSM $\mathcal{W_{\varepsilon}}(E,t)$ of the periodic
orbit is $O(\varepsilon)-$close to that of the fixed point and can
be written as 
\[
\mathcal{W_{\varepsilon}}(E,t)=\left(\begin{array}{c}
\boldsymbol{\xi}\\
O(\varepsilon)
\end{array}\right).
\]

We now neglect the small, time-dependent correction to the $\boldsymbol{\eta}$-component
of $\mathcal{W_{\varepsilon}}(E,t).$ The reduced dynamics can then
be written as the projection of \eqref{eq:linearized dynamics with epsilon}
to the $\boldsymbol{\eta}=0$ plane, which becomes 
\begin{equation}
\dot{\boldsymbol{\xi}}=\boldsymbol{\Lambda}_{E}\boldsymbol{\xi}+\left(\mathbf{T}+D\mathbf{h}^{0}(\mathbf{Ty})\mathbf{T}\right){}^{-1}\vert_{E}\varepsilon\mathbf{F}(0,t)+O(|\mathbf{y}|^{2}\varepsilon,y\varepsilon^{2}),\qquad\boldsymbol{\Lambda}_{E}=\left(\mathbf{T}^{-1}\mathbf{AT}\right)\vert_{E}.\label{eq:linearized dynamics with epsilon xi}
\end{equation}

The restriction of the observable vector $\boldsymbol{\phi}$ to the
subspace $E$ can be written as $\boldsymbol{\varphi}(\boldsymbol{\xi}).$
Repeating the arguments of \ref{sec:Proof-of-Theorem-DDL}, the change
of coordinates $\boldsymbol{\xi}=\boldsymbol{\psi}(\boldsymbol{\varphi})$
can be carried out, which transforms \eqref{eq:linearized dynamics with epsilon xi}
to
\begin{align}
\dot{\boldsymbol{\varphi}} & =D_{\boldsymbol{\xi}}\boldsymbol{\varphi}(\mathbf{0})\boldsymbol{\Lambda}_{E}\left[D_{\boldsymbol{\xi}}\boldsymbol{\varphi}(\mathbf{0})\right]^{-1}\boldsymbol{\varphi}+\mathbf{q}\left(\boldsymbol{\varphi}\right)+\varepsilon\left\{ \left[\mathbf{I}+D\mathbf{h}^{0}\left(D_{\boldsymbol{\xi}}\boldsymbol{\varphi}(\mathbf{0})\right)^{-1}\right]\mathbf{T}_{E}\right\} ^{-1}\mathbf{F}(\mathbf{0},t)+O(|\boldsymbol{\varphi}|^{2}\varepsilon,\left|\boldsymbol{\varphi}\right|\varepsilon^{2}),\label{eq:linearized dynamics with epsilon gamma}\\
\mathbf{q}\left(\boldsymbol{\varphi}\right) & =o\left(\left|\boldsymbol{\varphi}\right|\right),\nonumber 
\end{align}
 which proves formula (\ref{eq:gamma ODE-1}).

\section{Forced response from approximate DDL \label{sec:approximate ddl}}

As stated in Section \ref{subsec:predicting forced response from DDL},
the linearizing change of coordinates $\boldsymbol{\varphi}=\boldsymbol{\kappa}\left(\boldsymbol{\gamma}\right)=\boldsymbol{\gamma}+\boldsymbol{\ell}(\boldsymbol{\gamma})$
transforms the reduced dynamics of the forced system \eqref{eq:linearized dynamics with epsilon gamma}
to the form $\dot{\boldsymbol{\gamma}}=\mathbf{B}\boldsymbol{\gamma}+\varepsilon(\mathbf{I}+D\boldsymbol{\ell}(\boldsymbol{\gamma}))^{-1}\hat{\mathbf{F}}(\mathbf{0},t).$
Therefore, a trajectory of the forced system may be computed by integrating
\eqref{eq:linearized flow on SSM-1} and transforming to the reduced
coordinates $\boldsymbol{\varphi}$ as
\[
\boldsymbol{\varphi}(t)=\boldsymbol{\gamma}\left(t\right)+\boldsymbol{\ell}\left(\boldsymbol{\gamma}(t)\right).
\]

An approximate version of \eqref{eq:linearized flow on SSM-1} may
be obtained by assuming \eqref{eq:approximate ddl assumption}, i.e.,
$\varepsilon(\mathbf{I}+D\boldsymbol{\ell}(\boldsymbol{\gamma}))^{-1}=\varepsilon\mathbf{I}+O\left(\varepsilon|\boldsymbol{\gamma}|\right)\approx\varepsilon\mathbf{I}.$
This results in the forced linear system
\begin{equation}
\dot{\boldsymbol{\gamma}}=\mathbf{B}\boldsymbol{\gamma}+\varepsilon\hat{\mathbf{F}}\left(\mathbf{0},t\right).\label{eq:approximate_ddl_dynamics}
\end{equation}

However, we must keep in mind that the $O\left(\varepsilon|\gamma|\right)$
terms are already retained when we compute $\boldsymbol{\varphi}(t)=\boldsymbol{\gamma}(t)+\boldsymbol{\ell}\left(\boldsymbol{\gamma}(t)\right)$,
therefore terms of the same order must also be retained in the dynamics
\eqref{eq:linearized flow on SSM-1}. Nevertheless, assumption \eqref{eq:approximate ddl assumption}
makes the dynamics \eqref{eq:approximate_ddl_dynamics} an inhomogeneous
linear system of ODEs, which is simple to solve. We refer to DDL under
assumption \eqref{eq:approximate ddl assumption} as Approximate DDL.
We compare forced response predictions of the analytic linearization,
DDL, Approximate DDL and DMD on the forced-damped Duffing system discussed
in Section \ref{subsec:Duffing}. 

In Fig. \ref{fig:appendix_duffing}, we repeat the results already
presented in Fig. \ref{fig:duffing_2}, which should now be compared
to the forced response predicted using Approximate DDL. Panel (a)
shows that although Approximate DDL is accurate only for low amplitudes,
it clearly outperforms the DMD-model. For higher amplitudes, predictions
are not accurate, since the forced linear system \eqref{eq:approximate_ddl_dynamics},
by construction, cannot show nonlinear behavior. In contrast, the
analytic linearization and the full DDL continue to yield accurate
predictions for the forced response, even in the nonlinear regime. 

\begin{figure}
\begin{centering}
\includegraphics[width=0.8\textwidth]{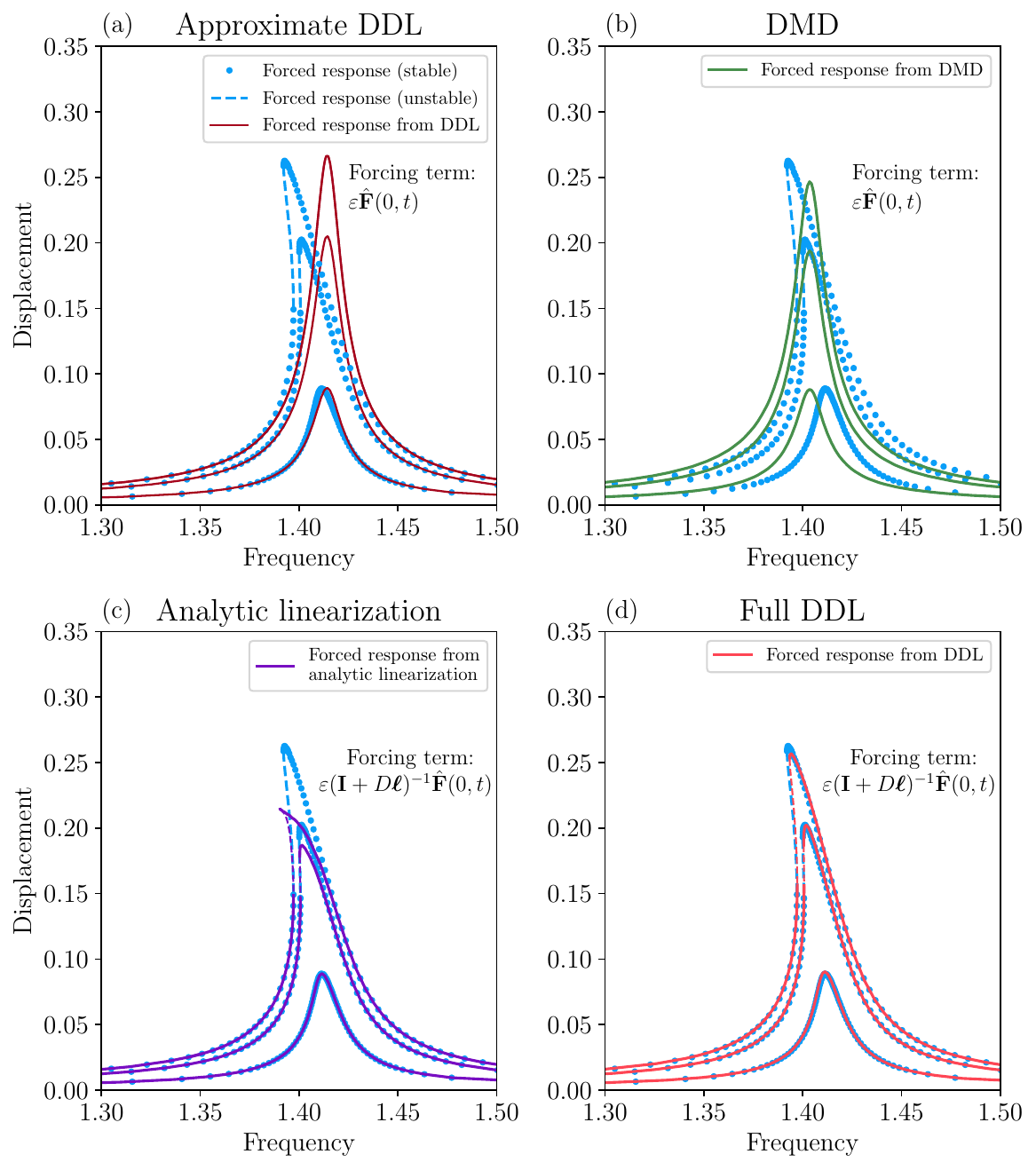}\caption{Periodic response of the forced Duffing oscillator \eqref{eq:duffing_transformed}.
Panel (a) shows DDL-predictions under the simplifying assumption \eqref{eq:approximate ddl assumption}.
Panels (b)-(d) are the same as Fig. \ref{fig:duffing_2}.\label{fig:appendix_duffing}}
\par\end{centering}
\end{figure}

\end{document}